\documentclass[10pt,a4paper]{article}

\usepackage{amsthm, amsmath, amsopn, latexsym, array }
\usepackage{amsfonts, amssymb}
\usepackage{enumitem, mdwlist}
\usepackage{dsfont}
\usepackage[charter]{mathdesign}
\usepackage{eucal}

\usepackage[usenames, dvipsnames]{xcolor}
\usepackage{xcolor}

\usepackage{hyperref}
\hypersetup{
    colorlinks = true,
    citecolor = {NavyBlue},
    linkcolor = {BrickRed},
    urlcolor = magenta,
}

\usepackage[margin=2.8cm,nohead]{geometry}

\usepackage{leftidx}

\usepackage{minibox}

\usepackage[utf8]{inputenc}
\usepackage[english]{babel}
\usepackage{tikz}
\usepackage[all]{xy}
\usepackage{verbatim}
\usepackage{float}
\usepackage[active]{srcltx}
\usepackage{amscd}

\usepackage{manyfoot}

\newlist{todolist}{itemize}{2}
\setlist[todolist]{label=$\square$}
\usepackage{pifont}
\newtheorem{theorem}{Theorem}[section]
\newtheorem{proposition}[theorem]{Proposition}
\newtheorem{lemma}[theorem]{Lemma}
\newtheorem{corollary}[theorem]{Corollary}
\newtheorem{definition}[theorem]{Definition}
\newtheorem{claim}[theorem]{Claim}

\theoremstyle{definition}
\newtheorem{remark}[theorem]{Remark}

\newcommand{\m}{m}

\newcommand{\moduli}{\mathfrak{M}}

\newcommand{\mE}{\mathscr{E}}
\newcommand{\mF}{\mathscr{F}}
\newcommand{\mU}{\mathscr{U}}
\newcommand{\mW}{\mathscr{W}}
\newcommand{\mG}{\mathscr{G}}
\newcommand{\sP}{\mathbb{P}(\mathscr{F})}
\newcommand{\mK}{\mathscr{K}}
\newcommand{\mI}{\mathscr{I}}
\newcommand{\mO}{\mathscr{O}}
\newcommand{\mL}{\mathscr{L}}
\newcommand{\mH}{\mathscr{H}}
\newcommand{\mR}{\mathscr{R}}

\newcommand{\Ext}{\mathrm{Ext}}
\newcommand{\Gr}{\mathbb{G}}

\newcommand{\VV}{\mathbb{V}}
\newcommand{\UU}{\mathbb{U}}
\newcommand{\NN}{\mathbb{N}}
\newcommand{\ZZ}{\mathbb{Z}}

\newcommand{\QQ}{\mathbb{Q}}
\newcommand{\CC}{\mathbb{C}}
\newcommand{\PP}{\mathbb{P}}
\newcommand{\II}{\mathbb{I}}
\newcommand{\PD}{\check{\mathbb{P}}}
\newcommand{\GG}{\mathbb{G}}
\newcommand{\FF}{\mathbb{F}}

\newcommand{\lr}{\longrightarrow}
\newcommand{\llr}{\longrightarrow}
\newcommand{\ts}{\otimes}
\newcommand{\frf}{\mathfrak{f}}
\newcommand{\frc}{\mathfrak{c}_0}
\newcommand{\cok}{\mathrm{coker}}

\newcommand{\xr}{\xrightarrow}
\newcommand{\Hilb}{\mathrm{Hilb}}

\title{Triple planes with $p_g=q=0$}
\begin{document}
\sloppy 

\date{}
\author{Daniele Faenzi, Francesco Polizzi and Jean Vallès}

\maketitle

\begin{abstract}
We show that general triple planes with genus and irregularity zero belong to at most
12 families, that we call surfaces of type I to XII, and we prove that the corresponding Tschirnhausen
bundle is a direct sum of two line bundles in cases I, II, III,
whereas is a rank 2  Steiner bundle in the remaining cases. 

We also provide existence results and explicit descriptions for
surfaces of type I to VII, recovering all classical examples and
discovering several new ones. In particular, triple planes of type VII
provide counterexamples to a wrong claim made in 1942 by Bronowski.

Finally, in the last part of the paper we discuss some moduli problems related to our constructions.
\end{abstract}

\maketitle

\Footnotetext{{}} {\textit{Mathematics Subject Classification}: 14E20}
\Footnotetext{{}} {\textit{Keywords and phrases}: Triple plane, Steiner bundle, adjunction theory}

\setcounter{tocdepth}{2}
\tableofcontents

\section*{Introduction}

A \emph{triple plane} is a finite ramified cover $f \colon X \to
\PP^2$ of degree $3$. Let $B \subset \PP^2$ be the branch locus of $f$;
then we say that $f$ is a \emph{general} triple plane if the
following conditions are satisfied:
\begin{enumerate}[itemsep=1pt, label=\bf{\roman*)}]
\item  $f$ is unramified over $\PP^2 \backslash B$; 
\item $f^*B=2R+R_0$, where $R$ is
irreducible and non-singular and $R_0$ is reduced;
\item $f_{|R} \colon R \to B$ coincides with
the normalization map of $B$.
\end{enumerate}

The aim of this paper is to address the problem of classifying those smooth, projective
surfaces $X$ with $p_g(X)=q(X)=0$ that arise as general triple
planes. We incidentally remark that the corresponding classification problem for 
\emph{double} planes is instead easy because, by the results of \cite[\S 22]{BHPV04}, a smooth double cover $f \colon X \to \PP^2$ with $p_g(X)=q(X)=0$ has either a smooth branch locus of degree $2$ (in which case $X$ is isomorphic to a quadric surface $S_2 \subset \PP^3$ and $f$ is the projection from a point $p \notin S_2$), or a smooth branch locus of degree $4$ (in which case $X$ is the blow-up  of a cubic surface $S_3 \subset \PP^3$ at one point $p \in S_3$ and $f$ is the resolution of the projection from $p$).

Some results toward the classification in the triple cover case were obtained by Du Val in
\cite{DuVal33,DuVal35}, where he described those general triple
planes whose branch curve has degree at most $14$. Du Val's papers are
written in the ``classical'', a bit old-fashioned (and sometimes difficult to read)
language and make use of ad-hoc constructions based on synthetic
projective geometry (see Remark \ref{rem:Du-Val-I-V} and Remark \ref{rem:Du-Val-VI} for a outline on Du Val's work). The methods that
we propose here are completely different: in fact, they are a mixture
of adjunction theory and vector bundles techniques, that 
allow us to treat the problem in a unified way.

The first cornerstone in our work is the general structure theorem
for triple covers given in \cite{Mi85,CasEk96}. More
precisely, we relate the existence of a triple cover $f \colon X
\to \PP^2$ to the existence of a ``sufficiently general"
element $\eta \in H^0(\PP^2, \, S^3 \mE^{\vee} \otimes \wedge^2 \mE)$,
where $\mE$ is a rank $2$ vector bundle on $\PP^2$ such
that $f_* \mO_X = \mO_{\PP^2} \oplus \mE$. Such a bundle is called the
\emph{Tschirnhausen bundle} of the cover, and it turns out that the pair $(\mE, \,
\eta)$ completely encodes the geometry of  $f$. 
Some of the invariants depend directly on $\mE$,
for instance, setting $b:=-c_1(\mE)$ and $h:=c_2(\mE)$,
 the branch curve $B$ has degree $2b$ and contains $3h$ ordinary
cusps as only singularities, see Proposition \ref{branch-locus}.
However the $X$ and $f$ themselves also depend on $\eta$;
we call $\eta$ the \textit{building section} of the cover.

So we can try to study general triple planes with $p_g=q=0$ by analyzing
their Tschirnhausen bundles together with the building sections. In fact, we show that these triple
planes can be classified in (at most) $12$ families, that we call surfaces of type \ref{I}, \ref{II},\ldots, \ref{XII}. We are also able to show that
surfaces of type \ref{I}, \ref{II},\ldots, \ref{VII} actually exist. In the cases
\ref{I}, \ref{II},\ldots, \ref{VI} we rediscover (in the modern language) the
examples described by Du Val. On the other hand, not only the triple
planes of type \ref{VII} (which have sectional genus equal to $6$ and
branch locus of degree $16$) are completely new, but they also provide
explicit counterexamples to a wrong claim made by Bronowski in
\cite{Br42}, see Remark \ref{rmk:Bron}. 

A key point in our analysis is the fact that in cases \ref{I}, \ref{II}, \ref{III} the
bundle $\mE$ splits as the sum of two line bundles, whereas in the remaining cases
\ref{IV} to \ref{XII} it is indecomposable and it has a resolution of the form
\begin{equation*}
0  \to \mO_{\PP^2}(1-b)^{b-4} \to
\mO_{\PP^2}(2-b)^{b-2} \to \mE \to 0.
\end{equation*}
This shows that $\mE(b-2)$ is a so-called
\emph{Steiner bundle} (see \S \ref{subsec:steiner} for more details on this topic), so we can use all the known results about
Steiner bundles in order to get information on $X$. For instance,
in cases \ref{VI} and \ref{VII} the geometry of the triple plane is tightly
related to the existence of \emph{unstable lines} for $\mE$, see
\S \ref{subsec:VI}, \ref{subsec:VII}. 

The second main ingredient in our classification
procedure is adjunction theory, see \cite{SoV87,Fu90}. For example, if we write $H=f^*L$, where $L \subset \PP^2$ is a
general line, we prove that the divisor $D=K_X+2H$ is very ample
(Proposition \ref{prop.K+2H}), so we consider the corresponding adjunction mapping
\begin{equation*}
\varphi_{|K_X+D|} \colon X \to \PP(H^0(X, \, \mO_X(K_X+D))).
\end{equation*}
Iterating the adjunction process if necessary, we
can achieve further information about the geometry of $X$.
Furthermore, when $b \geq 7$ a more refined analysis of the adjunction map allows us to start the process with $D=H$, see 
Remark \ref{rem:H}.

As a by-product of our classification, it turns out that general triple planes
$f \colon X \to \PP^2$ with sectional genus $0 \leq g(H) \leq 5$
(i.e., those of type \ref{I},\ldots,\ref{VI}) can be realized via an embedding of
$X$ into $\Gr(1, \, \PP^3)$ as a surface of bidegree $(3, \, n)$, such
that the triple covering $f$ is induced by the projection from  a
general element of the family of planes of $\Gr(1, \, \PP^3)$ that are
$n$-secant to $X$. In this way, we relate our work to the work of
Gross \cite{G93}, see Remark \ref{rmk:I}, \ref{rmk:II},
\ref{rmk:III}, \ref{rmk:IV}, \ref{rmk:V}, \ref{rmk:VI}. On the other
hand, this is not true for surfaces of type \ref{VII}: here the
only case where the triple cover is induced by an embedding in the
Grassmannian is \ref{VII}.7, where $X$ is a \emph{Reye
  congruence}, namely an Enriques surface having bidegree $(3, \, 7)$
in $\Gr(1, \, \PP^3)$, see Remark \ref{rmk:Enriques}. 

We have not been able so far to use our method beyond
case \ref{VII}; thus the existence of surfaces of type \ref{VIII} to \ref{XII} is still an open problem. Furthermore, there are some interesting unsettled
questions also in case \ref{VII}, especially  regarding the number of
what we call the \emph{unstable conics} for the Tschirnhausen bundle, see \S
\ref{subsec:further-cons} for more details. 

\medskip

Let us explain now how this work is organized.
In \S \ref{sec:prel} we set up notation and terminology and we
collect the background material which is needed in the sequel of the
paper. In particular, we recall the theory of triple covers based on
the study of the Tschirnhausen bundle (Theorems \ref{thm triple:1}
and \ref{thm triple:2}) and we state the main results on adjunction
theory for  surfaces (Theorem
\ref{adjunction-thm}).

In \S \ref{sec.general} we start the analysis of general triple
planes $f \colon X \to \PP^2$ with $p_g(X)=q(X)=0$. We
compute the numerical invariants (degree of the branch locus, number of its cusps, $K_X^2$, sectional genus) for the surfaces in the $12$
families \ref{I} to \ref{XII} (Proposition
\ref{prop:numerical}) and we describe their Tschirnhausen bundle (Theorem \ref{teo:res}).

Next, \S \ref{sec:class} is devoted to the detailed description of
surfaces of type \ref{I} to \ref{VII}. This
description leads to a complete classification in cases \ref{I} to \ref{VI} (Propositions \ref{prop.TypeI}, \ref{prop.TypeII}, \ref{prop.TypeIII}, \ref{prop.TypeIV}, \ref{prop.TypeV}, \ref{prop.TypeVI}) whereas in case \ref{VII} we provide many examples, leaving only few cases unsolved (Proposition \ref{prop.typeVII}).

Finally, in \S \ref{sec:moduli} we study some moduli problems related to our constructions. 

\medskip

Part of the computations in this paper was carried out by using the Computer
Algebra System {\tt Macaulay2}, see \cite{Mac2}. The scripts are
included in the Appendix.

\bigskip

\noindent \textbf{Acknowledgments.}
D. Faenzi and J. Vallès were partially supported by ANR projects GEOLMI ANR-11-BS03-0011
and Interlow ANR-09-JCJC-0097-01.
F. Polizzi was partially supported by
Progetto MIUR di Rilevante Interesse Nazionale \emph{Geometria
delle Variet$\grave{a}$ Algebriche e loro Spazi di Moduli}, by
the \emph{Gruppo di Ricerca Italo-Francese di Geometria Algebrica}
(GRIFGA) and by GNSAGA-INdAM. He thanks the Laboratoire de
Mathématiques et leurs applications de
l'Université de Pau et des Pays de l'Adour for
the invitation and the hospitality. 
This project was started in that occasion.
Finally, we are greatly indebted with the referee, whose careful reading, constructive criticisms 
and useful observations considerably helped us improving the paper.

\section{Basic material} \label{sec:prel}

\subsection{Notation and conventions}

We work over the field $\CC$ of complex numbers. Given a complex
vector space $V$, we write $\PP(V)$ for the projective space of
$1$-dimensional quotient spaces of $V$, and $\PP^n=\PP(\CC^{n+1})$.
 Similarly, if $\mathscr{E}$ is a locally free sheaf over a scheme, we use $\PP(\mE)$ for the projective bundle of its quotients of rank $1$.
We write $\PD(V)$ for $\PP(V^\vee)$, so that $\PD^n$ is the projective
space of hyperplanes of $\PP^n$. We put $\Gr(k,\PP(V))$ for the
Grassmannian of $(k+1)$-dimensional vector subspaces of $V$.

By ``surface'' we mean a projective, non-singular surface
$S$, and for such a surface $\omega_S=\mO_S(K_S)$ denotes the
canonical class, $p_g(S)=h^0(S, \, K_S)$ is the geometric genus,
$q(S)=h^1(S, \, K_S)$ is the
 irregularity and $\chi(\mO_S)=1-q(S)+p_g(S)$ is the
 holomorphic
Euler-Poincaré characteristic. We write
$P_m(S)= h^0(S, \, m K_S)$ for the $m$-th plurigenus of $S$.
 
If $k \leq n$ are non-negative integers we denote by $S(k, \, n)$ the
rational normal scroll of type $(k, \, n)$ in $\PP^{k+n+1}$, i.e. the
image of $\PP(\mO_{\PP^1}(k) \oplus \mO_{\PP^1} (n))$ by the linear system
given by the tautological relatively
ample line bundle (see
\cite[Lecture 8]{Har92} for more details). A cone over
a rational normal curve $C \subset \PP^n$ of degree $n$ may be thought of as the
scroll $S(0, \, n) \subset \PP^{n+1}$. 

For $n \geq 1$, we write $\FF_n$ for the Hirzebruch surface
$\PP(\mO_{\PP^1} \oplus \mO_{\PP^1} (n))$; every divisor in
$\textrm{Pic}(\FF_n)$ can be written as $a \frc+ b \frf$, where $\frf$
is the fibre of the $\PP^1$-bundle map $\FF_n \to \PP^1$ and $\frc$ is
the unique section with negative self-intersection, namely
$\frc^2=-n$. Note that the morphism $\FF_n \to \PP^{n+1}$ associated
with the tautological linear system $|\frc+n \frf|$ contracts $\frc$
to a point and is an isomorphism outside $\frc$, so its image is the
cone $S(0, \, n)$. 

For $n=0$, the surface $\FF_0$ is isomorphic to $\PP^1 \times \PP^1$;
every divisor in $\textrm{Pic}(\PP^1 \times \PP^1)$ is written as
$a_1L_1+ a_2L_2$, where the $L_i$ are the two rulings. 

The blow-up of $\PP^2$ at the points $p_1, \ldots, p_k$  is
denoted by $\PP^2(p_1, \ldots, p_k)$.
If $\sigma \colon \widetilde{X} \to X$ is the blow-up of a surface $X$ at $k$ points, with exceptional divisors $E_1, \ldots, E_k$, and $L$ is a line bundle on $X$, we will write $L + \sum a_i E_i$ instead of $\sigma^*L + \sum a_i E_i$.

The Chern classes of coherent sheaves on $\PP^2$ will usually be written
as integers, namely for a sheaf $\mE$ we write $c_i(\mE)=d_i$, where $d_i \in \ZZ$ is such that $c_i(\mE)=d_i(c_1(\mO_{\PP^2}(1)))^i$.
If $\mE$ is a vector bundle, its dual vector bundle is indicated by $\mE^{\vee}$ and its $k$-th symmetric power by $S^k \mE$. 

We will use basic material and terminology on vector bundles, more
specifically on stable vector bundles on $\PP^2$, we refer to \cite{OkSchSp88}.

\subsection{Triple covers and sections of vector bundles}

A \emph{triple cover} is a finite flat morphism $f \colon X \to Y$ of
degree $3$. Our varieties $X$ and $Y$
will be smooth, irreducible projective manifolds.
With a triple cover is associated an exact sequence
\begin{equation} \label{seq:Tsch}
0 \to \mO_Y \to f_*\mO_X \to \mE \to 0,
\end{equation}
where $\mE$ is a vector bundle of rank $2$ on $Y$, called the
\emph{Tschirnhausen bundle} of $f$.

\begin{proposition} \label{bir-inv-tr}
The following holds$:$
\begin{enumerate}[itemsep=2pt, label=\bf{\roman*)}]
\item \label{p1} $f_*\mO_X =\mO_Y \oplus \mE;$
\item \label{p2} $f_* \omega_X =\omega_Y \oplus (\mE^{\vee} \otimes \omega_Y);$
\item \label{p3} $f_* \omega_X^2 = S^2 \mE^{\vee} \otimes \omega_Y^2$.
\end{enumerate}
\end{proposition}
\begin{proof}
The trace map yields a splitting of sequence \eqref{seq:Tsch},
hence \ref{p1} follows. Duality for finite flat morphisms implies
$f_*\omega_X= (f_* \mO_X)^{\vee} \otimes \omega_Y$, hence we
obtain \ref{p2}. For \ref{p3} see \cite[Lemma 8.2]{Pa89}.
\end{proof}

In order to reconstruct $f$ from $\mE$ we need an
extra datum, namely the building section, which is a global section of $S^3
\mE^{\vee} \otimes \wedge^2 \mE$. Moreover, we can naturally see $X$ as
sitting into the $\PP^1$-bundle $\PP(\mE^\vee)$ over $Y$. This is the content of the next two results, see 
\cite[Theorem 1.5]{CasEk96}, \cite[Proposition 4]{FS01}, \cite[Theorem 1.1]{Mi85}
\begin{theorem} \label{thm triple:1}
Any triple cover $f \colon X \to Y$ is determined by a rank $2$ vector
bundle $\mE$ on $Y$ and a global section $\eta \in H^0(Y, S^3
\mE^{\vee} \otimes \wedge^2 \mE)$, and conversely. Moreover, if 
 $S^3 \mE^{\vee} \otimes \wedge^2 \mE$ is globally
generated, a general global section $\eta$ defines a triple cover $f \colon X \to Y$.
\end{theorem}

\begin{theorem} \label{thm triple:2}
Let $ f \colon X \to Y$ be a triple cover.
 Then there exists a unique embedding $i \colon X \to
\PP(\mE^{\vee})$ such that the following diagram commutes$:$
\begin{equation*}
\xymatrix@C-4ex@R-3ex{
X \ar[rrrr]^-{f} \ar[rrd]_-{i} & & & &  Y \\
& & \PP(\mE^{\vee}) \ar[rru]_-{\pi} & & }
\end{equation*}
According to \emph{Theorem \ref{thm triple:1}}, this embedding induces an
isomorphism of $X$ with the zero-scheme $D_0(\eta)\subset \PP (\mE^\vee)$ of a global section $\eta$ of
the line bundle $\mO_{\PP (\mE^\vee)}(3) \ts \pi^*( \wedge^2 \mE)$.
\end{theorem}

\subsection{Adjunction theory} \label{subsec:adjunction theory}
We refer to  \cite[Chapter 10]{BeSo95}, \cite[Theorem 1.10]{DES93}, \cite[Theorem 2.5]{LaPa84}, \cite[Proposition 1.5]{So79},\cite[\S 0]{SoV87} for basic material on adjunction theory.

\begin{theorem} \label{adjunction-thm}
Let $X \subset \PP^n$ be a smooth surface and $D$ its
hyperplane class. Then $|K_X+D|$ is non-special and has dimension
$N=g(D)+p_g(X)-q(X)-1$. Moreover
\begin{enumerate}[itemsep=2pt, label=\bf{\Alph*)}]
\item $|K_X+D|= \emptyset$ if and only if
\begin{enumerate}[itemsep=2pt, label=\bf{\arabic*)}]
\item $X \subset \PP^n$ is a scroll over a curve of
genus $g(D)=q(X)$ or
\item $X= \PP^2$,
$D=\mO_{\PP^2}(1)$ or $D=\mO_{\PP^2}(2)$.
\end{enumerate}
\item If $|K_X+D| \neq \emptyset$ then $|K_X+D|$ is base-point
free. In this case $(K_X+D)^2=0$ if and only if
\begin{enumerate}[itemsep=2pt, label=\bf{\arabic*)}]\setcounter{enumii}{2}
\item $X$ is a Del Pezzo surface and $D=-K_X$ $($in particular $X$ is
rational$\,)$ or 
\item $X \subset \PP^n$
is a conic bundle.
\end{enumerate}
If $(K_X+D)^2>0$ then the adjunction map
\begin{equation*}
\varphi_{|K_X+D|} \colon X \to X_1 \subset \PP^N
\end{equation*}
defined by the complete linear system $|K_X+D|$ is birational onto a
smooth surface $X_1$ of degree $(K_X+D)^2$ and blows down precisely the
$(-1)$-curves $E$ on $X$ with $DE=1$, unless
\begin{enumerate}[itemsep=2pt, label=\bf{\arabic*)}]\setcounter{enumii}{4}
\item $X=\PP^2(p_1, \ldots, p_7), \quad D=6L-\sum_{i=1}^7 2E_i$, 
\item $X=\PP^2(p_1, \ldots, p_8), \quad D=6L-\sum_{i=1}^7 2E_i -E_8$,
\item $X=\PP^2(p_1, \ldots, p_8), \quad D=9L-\sum_{i=1}^8 3E_i$,
\item $X= \PP(\mE)$, where $\mE$ is an
indecomposable rank $2$ vector bundle over an elliptic curve and
$D=3B$, where $B$ is an effective divisor on $X$ with $B^2=1$.
\end{enumerate}
\end{enumerate}
\end{theorem}

We can apply Theorem \ref{adjunction-thm} repeatedly, obtaining a
 sequence of surfaces and adjunction maps
\begin{equation*}
X =:X_0 \stackrel{\varphi_1}{\llr} X_1 \stackrel{\varphi_2}{\llr} X_2
\stackrel{}{\llr} \cdots \stackrel{}{\llr} X_{n-1}
\stackrel{\varphi_{n}}{\llr} X_n \to \cdots
\end{equation*}
At each step we must control the numerical data arising from the
adjunction process. We have
\begin{equation*}
(D_{i+1})^2=(K_{X_i}+D_i)^2, \quad K_{X_{i+1}}D_{i+1}=(K_{X_i}+D_i)K_{X_i}.
\end{equation*}
For the computation of
\begin{equation*}
(K_{X_{i+1}})^2=(K_{X_i})^2+\alpha_i
\end{equation*}
we also need to know the number $\alpha_i$ of \emph{exceptional
lines} on $X_i$, i.e. the number of smooth curves $E \subset X_i$ such that
$K_{X_i}E=E^2=-1,\; ED_{i}=1$. Notice that by the Hodge Index Theorem
(see \cite[Exercise 1.9 p. 368]{H77}) we have
\begin{equation*}
\det \left( \begin{array}{cc} (D_i)^2 & K_{X_i}D_i \\
 K_{X_i}D_i & (K_{X_i})^2\\
\end{array} \right) \leq 0
\end{equation*}
and the equality holds if and only if $K_{X_i}$ and $D_i$ are
numerically dependent.

\begin{proposition} \label{prop:contracted}
Let $E \subset X_{n-1}$ be a curve contracted by the $n$-th
adjunction map $\varphi_{n} \colon X_{n-1} \to X_n$. Then, setting
$\psi:=\varphi_{n-1} \circ
\varphi_{n-2} \circ \cdots \circ \varphi_1$ and $E^*:=\psi^*E$, we have
\begin{equation*}
(E^*)^2=-1, \quad K_XE^*=-1, \quad DE^*=n.
\end{equation*}
\end{proposition}
\begin{proof}
Since $E$ is contracted by $\varphi_n$, we have $E^2=-1$,
$K_{X_{n-1}}E=-1$, $D_{n-1}E=1$. The map $\psi$ is
birational, so $(E^*)^2=E^2=-1$. Moreover
\begin{equation*}
\psi_*K_X=K_{X_{n-1}}, \quad \psi_*D = D_{n-1}-(n-1)K_{X_{n-1}}.
\end{equation*}
Applying the projection formula we obtain
\begin{equation*}
\begin{split}
K_XE^*&=(\psi_*K_X)E=K_{X_{n-1}}E=-1; \\
DE^*&=(\psi_*D)E=(D_{n-1}-(n-1)K_{X_{n-1}})E=n.
\end{split}
\end{equation*}
This completes the proof.
\end{proof}

\subsection{Steiner bundles} \label{subsec:steiner}

We collect here some material about coherent sheaves presented by
matrices of linear forms, usually called Steiner sheaves, and more
specifically about those that are locally free (Steiner bundles).
We refer to \cite{ancona-ottaviani:steiner} for basic results on this
topic.

\subsubsection{Steiner sheaves and their projectivization} \label{a}

Let $U$, $V$ and $W$ be finite-dimensional $\CC$-vector spaces. Consider the
projective spaces $\PP(V)$ and
$\PP(U)$, and identify $V$ and $U$ with
$H^0(\PP(V), \, \mO_{\PP(V)}(1))$ and $H^0(\PP(U), \, \mO_{\PP(U)}(1))$,
respectively. 
Any element $\phi \in U \ts V \ts W$ gives rise to two maps
\begin{equation} \label{eq:M-N}
   W^{\vee} \ts \mO_{\PP(V)}(-1) \xrightarrow{M_\phi} U \ts \mO_{\PP(V)}, \qquad W^{\vee} \ts \mO_{\PP(U)}(-1) \xrightarrow{N_\phi} V \ts \mO_{\PP(U)}.
\end{equation}
Set $\mF:= \cok \, M_\phi$.
We say that $\mF$ is a {\it Steiner sheaf}, and we denote its projectivization by $\PP(\mF)$; this
 is a projective bundle precisely when
$\mF$ is locally free (and in this case $\dim(U) \geq \dim(V)+\dim(W)-1)$. Let $\mathfrak{p} \colon \sP \to \PP(V)$ be the bundle map
and $\mO_{\sP}(\xi)$ be the tautological, relatively ample line
bundle on $\sP$, so that
\begin{equation*}
H^0(\sP, \mO_{\sP}(\xi)) \simeq H^0(\PP(V),\mF) \simeq  U.
\end{equation*}
Since $\mF$ is a quotient of $U \otimes \mO_{\PP(V)}$, we get a natural
embedding
\begin{equation*}
\sP \subset \PP(U \otimes \mO_{\PP(V)}) \simeq  \PP(V) \times \PP (U).
\end{equation*}

The map $\mathfrak{q}$ associated with the linear system $|\mO_{\sP}(\xi)|$ is just the restriction to $\sP$ of the second
projection from $\PP(V) \times \PP (U)$. On the other hand, setting $\ell :=
\mathfrak{p}^*(\mO_{\PP(V)}(1))$, the linear system $|\mO_{\sP}(\ell)|$ is
naturally associated with the map $	\mathfrak{p}$. In this procedure the roles of $U$ and $V$ can be reversed. In other words, 
setting $\mG = \cok \,N_\phi$, we get a second Steiner sheaf, this time
on $\PP(U)$, and a second projective bundle $\PP(\mG)$ with maps
$\mathfrak{p}'$ and $\mathfrak{q}'$ to $\PP(U)$ and $\PP(V)$, respectively. 
So we have two incidence diagrams
\begin{equation*}
  \xymatrix@-3ex{
    & \sP \ar[dr]^{\mathfrak{q}} \ar[dl]_{\mathfrak{p}} & \\
    \PP(V) & & \PP(U),
  } \quad
  \xymatrix@-3ex{
    & \PP(\mG) \ar[dr]^{\mathfrak{q}'} \ar[dl]_{\mathfrak{p}'} & \\
    \PP(U) & & \PP(V).
  }
\end{equation*}

The link between $\sP$ and $\PP(\mG)$ is provided by the
following result.

\begin{proposition} \label{prop:P-and-P'}
  Let $\phi \in U \ts V \ts W$ and set $\m=\dim W$.
  Then:
  \begin{enumerate}[itemsep=2pt, label=\bf{\roman*)}]
  \item \label{primofatto} the schemes $\sP$ and $\PP(\mG)$ are both
    identified with the same $\m$-fold linear section of $\PP(V) \times \PP(U)$. Moreover, under this identification, $\mathfrak{q} =\mathfrak{p}'$ and $\mathfrak{p}=\mathfrak{q}';$
  \item \label{secondofatto} for any non-negative integer $k$, there are natural isomorphisms
    \begin{equation*}
    \mathfrak{p}_*\mathfrak{q}^*(\mO_{\PP(U)}(k)) \simeq S^k \mF, \qquad
    \mathfrak{q} _*\mathfrak{p}^*(\mO_{\PP(V)}(k)) \simeq S^k \mG.
    \end{equation*}
\end{enumerate}
\end{proposition}

\begin{proof}
Set $M:=M_{\phi}$. By construction, the scheme $\sP$ is defined as the set
\begin{equation*}
\sP=\{([v], \, [u]) \in \PP(V) \times \PP (U) \; | \; u \in \cok \, M(v) \},
\end{equation*}
where $v \colon V \to \CC$ (resp.  $u \colon U \to \CC$) is a $1$-dimensional quotient of $V$ (resp. of $U$) and $M(v) \colon W^\vee \to U$ is the evaluation of $M$ at the point $[v]$.
Now, we get that $u$ is defined on $\cok \, M(v)$ if and only if $u
\circ M(v) = 0$. This clearly amounts to require $(u\circ M(v) )(w) = 0$ for all $w\in W^\vee$, that is $u \ts v \ts w (\phi)=0$ for all $w \in W^\vee$.
Summing up, we have
\begin{equation}
  \label{PEphi}
  \sP = \{([v], \, [u]) \; | \; u \ts v \ts w
  (\phi)=0 \; \textrm{for all} \; w \in W^\vee\}.
\end{equation}

The same argument works for $\PP(\mG)$ by interchanging the roles
of $v$ and $u$, hence $\sP$ and $\PP(\mG)$ are 
both identified with the same subset of $\PP(V) \times \PP(U)$.
Since each element $w_i$ of a basis of $W^\vee$ gives a linear
equation of the form $u \ts v \ts w_i(\phi)=0$, we have that
$\sP$ is an $\m$-fold linear section (of codimension $\m$ or
smaller) of $\PP(V) \times \PP(U)$.

Note that, in view of the identification above, the map $\mathfrak{p}$ is just
the projection from $\PP(V) \times \PP(U)$ onto $\PP(V)$, restricted to
the set given by \eqref{PEphi}. The same holds for $\mathfrak{q}'$, hence we are
allowed to identify $\mathfrak{p}$ and $\mathfrak{q}'$. Analogously, both $\mathfrak{q}$ and $\mathfrak{p}'$ are
given as projections onto the factor $\PP(V)$.  We have thus proved \ref{primofatto}.
Now let us check \ref{secondofatto}. For any non-negative integer $k$ we have 
\begin{equation*}
\begin{split}
\mathfrak{q}^*(\mO_{\PP(U)}(k)) & \simeq \mO_{\sP}(k \xi), \\
\mathfrak{p}^*(\mO_{\PP(V)}(k)) & \simeq (\mathfrak{q}')^*(\mO_{\PP(V)}(k)) \simeq \mO_{\PP(\mG)}(k \xi'),
\end{split}
\end{equation*}
where $\xi'$ is the  tautological relatively ample line
bundle on $\PP(\mG)$.
Therefore the claim follows from the canonical isomorphisms
\begin{equation*}
\mathfrak{p}_*(\mO_{\sP}(k \xi)) \simeq S^k \mF, \qquad
\mathfrak{p}'_*(\mO_{\PP(\mG)}(k \xi')) \simeq S^k \mG.
\end{equation*}
\end{proof}

\begin{remark} \label{rem:trilinear-form}
We can rephrase the content of Proposition \ref{prop:P-and-P'} by using coordinates as follows. Take bases 
\begin{equation*}
\{z_i\}, \, \{x_j \}, \, \{y_k\}
\end{equation*}
for $U$, $V$, $W$, respectively. With respect to these bases, the tensor $\phi \in U \otimes V \otimes W$ will correspond to a trilinear form
\begin{equation*}
\phi = \sum a_{ijk}z_i x_j y_k,
\end{equation*}
for a certain table of coefficients $a_{i j k} \in \CC$. 
Write $\VV$ and $\UU$ for the symmetric algebras on $V$ and $U$.
Then $\phi$ induces two linear maps of graded vector spaces:
\begin{equation*}
W^\vee \otimes \VV(-1) \to U \otimes \VV, \quad
W^\vee \otimes \UU(-1) \to V \otimes \UU,
\end{equation*}
both defined as
 \begin{equation*}
  w \otimes \Psi \to \bigg(\sum a_{ijk}z_i x_j y_k(w) \bigg) \Psi,
 \end{equation*}
where $\Psi$ lies in $\VV$ or in $\UU$.
The sheafification of these maps gives precisely the two maps of
vector bundles $M_{\phi}$ and $N_{\phi}$ written in \eqref{eq:M-N},
whose defining matrices of linear forms are, respectively:
\begin{equation*}
\bigg(\sum \nolimits_j a_{ijk} x_j \bigg)_{ik} \quad \textrm{and} \quad  \bigg(\sum \nolimits_i a_{ijk}z_i \bigg)_{jk}.
\end{equation*}

An important observation is that $\PP(\mF)$ and $\PP(\mG)$ are both
identified with the zero locus of the same set of $m$ bilinear equations
in $\PP(V) \times \PP(U)$, namely 
\begin{equation*}
\PP(\mF)=\PP(\mG)=\left\{(x, \, z) \; \Big\lvert  \; \sum \nolimits_{i, \, j} a_{ij1}
  z_i x_j=\ldots=\sum\nolimits_{i, \, j} a_{ijm}  z_ix_j=0\right\}. 
\end{equation*}
This shows that $\PP(\mF)=\PP(\mG)$ is the intersection
of $m$ divisors of bidegree $(1,1)$ in $\PP(V) \times \PP(U)$. We can
thus write a presentation of the form:
\begin{equation}
  \label{prekoszul}
\cdots \to W^\vee \otimes \mO_{\PP(V) \times \PP(U)}(-1,-1) \to \mO_{\PP(V) \times \PP(U)} \to \mO_{\PP(\mF)} \to 0.
\end{equation}
\end{remark}

We will mostly use this setup when $\PP(V)=\PP^2$, in order to study
the geometry of a Steiner bundle $\mF$ of rank $2$ admitting the resolution
\begin{equation} \label{basic}
0 \to W^\vee \ts \mO_{\PP^2}(-1) \xrightarrow{M} U \ts \mO_{\PP^2} \to \mF \to 0,
\end{equation}
and to compare it with the geometry of the sheaf $\mG$ obtained by ``flipping'' the tensor $\phi$ as
explained above and whose presentation is
\begin{equation} \label{basic-flipped}
W^\vee \ts \mO_{\PP(U)}(-1) \xrightarrow{N} \mO_{\PP(U)}^3 \to \mG \to 0.
\end{equation}

\subsubsection{Unstable lines} \label{b}

Let us assume  now $\dim V=3$ and consider a Steiner bundle $\mF$ of rank $2$ on
$\PP^2=\PP(V)$. To be consistent with the notation that will appear later, we set 
 $\dim U = b-2$ and
$\dim W = b-4$, for $b \ge 4$, and
we write $\mF_b$ instead of $\mF$. The sheafified minimal graded free resolution of $\mF_b$ is then
\begin{equation} \label{eq:FF}
  0  \to \mO_{\PP^2}(-1)^{b-4} \stackrel{M}{\lr} \mO_{\PP^2}^{b-2} \to  \mF_b \to 0,
\end{equation}
where $M$ is a $(b-2) \times (b-4)$ matrix of linear forms. 

\medskip

Given a line $L \subset \PP^2$, there is an integer $a$ such that
\begin{equation*}
 \mF_b|_L = \mO_L(a)\oplus \mO_L(b-4-a).
\end{equation*}
Since $\mF_b$ is globally generated, the same is true for  $\mF_b|_L$ and so 
\begin{equation*}
0 \le a \le b-4.
\end{equation*}

\begin{definition}
A line $L \subset \PP^2$ is said to be \emph{unstable} for $\mF_b$ if $a=0$, i.e.
\begin{equation*}
 \mF_b|_L \simeq \mO_L\oplus \mO_L(b-4).
\end{equation*}
\end{definition}
Here are some useful characterizations of unstable lines.
\begin{lemma} \label{lem:unstable-line}
The following are equivalent$:$
\begin{enumerate}[itemsep=2pt, label=\bf{\roman*)}]
\item \label{u:1} the line $L \subset \PP^2$ is unstable line $\mF_b;$
\item \label{u:2} the cohomology group $H^0(L, \, \mF^{\vee}_b|_L)$ is
  non-zero$;$
\item \label{u:3} there is a nonzero global section of $\mF_b$ whose
  vanishing locus contains $b-4$ points of $L$ $($counted with multiplicity$)$. 
\end{enumerate}
\end{lemma}
\begin{proof}
We first prove \ref{u:1} $\Leftrightarrow$ \ref{u:2}. The restriction
$\mF_b|_L$ splits, so there is an integer $a$ such
that $\mF_b|_L = \mO_L(a)\oplus\mO_L(b-4-a)$, and since $\mF_b$ is
globally generated we have $0 \le a \le b-4$. Condition \ref{u:1}
corresponds to $a=0$ or $a=b-4$, and this clearly implies
\ref{u:2}. Conversely, if \ref{u:2} holds, then  $a
\le 0$ or $a \ge b-4$; this implies either $a=0$ or $a=b-4$, hence \ref{u:1} holds.

\bigskip

In order to check \ref{u:2} $\Leftrightarrow$ \ref{u:3}, we
first claim that, given a line $L \subset \PP^2$, the restriction map induces an isomorphism
\begin{equation} \label{eq:iso-F}
H^0(\PP^2,\mF_b) \stackrel{\simeq}{\longrightarrow} H^0(L, \mF_b|_L). 
\end{equation}
Indeed, looking at \eqref{eq:FF}, we see that we have
\begin{equation*}
H^0(\PP^2,\mF_b(-1))=H^1(\PP^2,\mF_b(-1))=0,
\end{equation*}
so our claim follows by taking cohomology in 
\begin{equation*}
0 \to \mF_b(-1) \to \mF_b \to \mF_b|_L \to 0.
\end{equation*}

\bigskip
Now let us prove \ref{u:2} $\Rightarrow$ \ref{u:3}. 
Assuming \ref{u:2}, we get a short exact sequence
\begin{equation*}
0 \to \mO_L \to \mF^{\vee}_b|_L \to \mO_L(4-b) \to 0,
\end{equation*}
so by dualizing we have
\begin{equation*}
0 \to \mO_L(b-4) \stackrel{\iota}{\to} \mF_b|_L \to \mO_L \to 0.
\end{equation*}
Composing $\iota$ with a non-zero map $\mO_L \to \mO_L(b-4)$, we
obtain a global section of $\mF_b|_L$ vanishing at $b-4$ points
counted with multiplicity. Using the isomorphism \eqref{eq:iso-F} we
can lift this section to a global section of
$\mF_b$ and we get \ref{u:3}. 

Conversely, assume that \ref{u:3} holds. 
Then there is a global
section $s$ of $\mF_b$ whose vanishing locus $Z$ contains a subscheme 
of $L$ of length $b-4$.
Put $Z'=Z\cap L$, so that $Z'$ has length $c \ge b-4$.
 Since $H^0(\PP^2,\mF_b(-1))=0$ it follows that $Z$ contains
no divisors, i.e. it has pure dimension $0$, so we have an exact sequence
\begin{equation*}
0 \to \mO_{\PP^2} \to \mF_b \to \mI_{Z/\PP^2}(b-4) \to 0.
\end{equation*}
Applying  $-\otimes_{\mO_{\PP^2}} \mO_L$ to the exact sequence 
\begin{equation*}
0 \to \mI_{Z/\PP^2}(b-4) \to \mO_{\PP^2}(b-4) \to \mO_Z \to 0
\end{equation*}
and using $\mathscr{T}or^{\mO_{\PP^2}}_1( \mO_Z, \, \mO_L) \simeq \mO_{Z'}$, we get
\begin{equation*} 
0 \to \mO_{Z'} \to \mI_{Z/\PP^2} (b-4)|_L \to \mO_L(b-4) \to \mO_{Z'} \to 0.
\end{equation*}
The image of the middle map is $\mI_{Z'/L}(b-4) \simeq \mO_L(b-c-4)$, so from the above sequence we obtain
\begin{equation}
  \label{accorcio}
   0 \to \mO_{Z'} \to \mI_{Z/\PP^2} (b-4)|_L \to \mO_L(b-c-4) \to 0.  
\end{equation}
The scheme $Z'$ is $0$-dimensional, so we infer
\begin{equation*}
\Ext^1(\mO_L(b-4-c),\mO_{Z'}) \simeq H^1(L,\mO_{Z'} \ts \mO_L(c-b+4))=0
\end{equation*}
and this means that \eqref{accorcio} splits, i.e.
\begin{equation} 
\mI_{Z/\PP^2}(b-4)|_L \simeq \mO_L(b-c-4) \oplus \mO_{Z'}.
\end{equation}
Therefore, we have a surjection $\mF_b|_L \to \mO_L(b-c-4)$. Since
$b-c-4 \le 0$, the dual of this surjection gives a non-zero global section of
$\mF_b^\vee|_L$ and the proof is finished.
Note that, since we have now proved $\mF_b|_L \simeq \mO_L \oplus
\mO_L(b-4)$, the existence of a surjection $\mF_b|_L \to \mO_L(b-c-4)$
actually gives $c=b-4$, i.e. $Z'=Z\cap L$ has length precisely $b-4$.
\end{proof}

The set of unstable lines of $\mF_b$ has a
natural structure of subscheme of $\PD^2$, given as follows. First
define the point-line incidence $\II$ in $\PP^2 \times \PD^2$ by the
condition that the point lies in the line. One has $\II \simeq
\PP(T_{\PP^2}(-1))$ and $T_{\PP^2}(-1)$ is a Steiner bundle. By Lemma \ref{lem:unstable-line}, 
a line $L$ is unstable for $\mF$ if and only if $H^0(L,\,\mF^{\vee}_b|_L)\neq 0$,
 i.e., by Serre duality, 
if and only if $H^1(L, \,\mF_b(-2)|_L)\neq 0$, which happens if and
only if $L$ lies in the support of $R^1\mathfrak{q}_* (\mathfrak{p}^* \mF_b (-2) \ts \mO_{\II})$.
We denote the set of unstable lines, endowed with this scheme structure, by $\mW(\mF_b)$.

Let us now give a summary of the behaviour of the unstable
lines of $\mF_b$ for small values of $b$.

\begin{itemize}
\item[$b=4$.] We have $\mF_{4} \simeq \mO_{\PP^2}^2$, so $\mW(\mF_4)$ is empty.
\item[$b=5$.] There is an isomorphism $\mF_5 \simeq
  T_{\PP^2}(-1)$. Therefore $\mW(\mF_5)=\PD^2$, because $T_{\PP^2}$ is a uniform bundle of splitting type $(1, \, 2)$, see \cite[\S 2]{OkSchSp88}.
\item[$b=6$.]
The scheme $\mW(\mF_6)$ is a smooth conic in $\PD^2$, and the unstable
lines of $\mF_6$ are the tangent lines to the dual conic, see
\cite[Proposition 6.8]{DK} and \cite[Proposition 2.2]{Val00b}.
\item[$b=7$.] The 
scheme $\mW(\mF_7)$ is either a set of $6$ points in general linear
position and contained in no conic or consists of a smooth conic in
$\PD^2$, see \cite[Théorème 3.1]{Val00b}. The former case is the
general one, and when it occurs $\mF_7$ is a so-called
\emph{logarithmic bundle}. Instead, the latter case occurs if and only
if $\mF_7 $ is a so-called \emph{Schwarzenberger bundle}, whose matrix
$M$, up to a linear change of coordinates, has the form 
 \begin{equation} \label{eq:Schw-7}
      M= \ltrans{\left(
        \begin{array}{ccccc}
          x_0 & x_1 & x_2 & 0 & 0 \\
          0 & x_0 & x_1 & x_2 & 0 \\
          0 & 0 & x_0 & x_1 & x_2 \\
        \end{array}
      \right)},
    \end{equation}
see \cite[Theorem 3]{faenzi-matei-valles}, \cite[Théorème 3.1]{Val00b}.
\item[$b\ge 8$.]
Unstable lines do not always exist in this range.
The scheme $\mW(\mF_b)$ is either finite of length $\le b-1$ or consists of a smooth conic in $\PD^2$. In the latter case, $\mF_b$ is a Schwarzenberger bundle, whose matrix $M$, up to a linear change of coordinates, is a $(b-2) \times (b-4)$ matrix having the same form as \eqref{eq:Schw-7}.
We can actually state a more precise result, see again \cite[Proposition 3.11 and proof of Theorem 5.3]{ancona-ottaviani:steiner}.

\begin{proposition} \label{prop:unstable-b-8}
If $\mF_b$ contains a finite number $\alpha_1$
of unstable lines, then $0 \leq \alpha_1 \leq b-1$. More precisely, the following holds.
\begin{enumerate}[itemsep=2pt, label=\bf{\roman*)}]
\item If $0 \leq \alpha_1 \leq b-2$ then, up to a linear change of coordinates, the matrix $M$ is of type
 \begin{equation*}
      M=\ltrans{\left(
        \begin{array}{ccc|ccc}
          a_{1,1} H_1 & \cdots & a_{1, \alpha} H_{\alpha} &  \\
   \vdots & \vdots & \vdots & M' \\
          a_{b-4,1} H_1 & \cdots & a_{b-4, \alpha} H_{\alpha} &  \\
        \end{array}
      \right)},
\end{equation*}
for some $(b-2-\alpha) \times (b-4)$ matrix $M'$ of linear forms. In this case the unstable lines are given by
\begin{equation*}
H_1=0, \quad H_2=0, \ldots,  H_{\alpha_1}=0.
\end{equation*}
\item If $\alpha_1=b-1$ then $\mF_b$ is a logarithmic bundle. In this case, the matrix $M$ is of type
\begin{equation*}
      M=\ltrans{\left(
        \begin{array}{cccc}
          a_{1,1} H_1 & a_{1,2}H_2 & \cdots & a_{1, b-2} H_{b-2}\\
   \vdots & \vdots & \vdots & \vdots  \\
         a_{b-4,1} H_1 & a_{b-4,2}H_2 & \cdots & a_{b-4, b-2} H_{b-2}\\
      \end{array}
      \right)},
 \end{equation*}
where $H_1, \ldots, H_{b-2}$ are such that the linear form
\begin{equation*}
H_{b-1} := \sum_{j=1}^{b-2} a_{i,j} H_j
\end{equation*}
does not depend on $i \in  \{1, \ldots, b-4\}$. The unstable lines are given by
\begin{equation*}
H_1=0, \quad H_2=0, \ldots,  H_{b-1}=0.
\end{equation*}
  \end{enumerate}
 \end{proposition}
\end{itemize}

\begin{remark} 
Using Proposition \ref{prop:unstable-b-8}, we can give another proof of
the implication \ref{u:1} $\Rightarrow$ \ref{u:3} in
Lemma \ref{lem:unstable-line}.  
Indeed, we can take a basis $s_1,\ldots,s_{b-2}$ of $H^0(\PP^2,\,
\mF_b)$ such that the homogeneous ideal $I_k$ of the vanishing locus
of $s_k$ is defined by the maximal minors of the matrix obtained by
deleting the $k$-th row of $M$, namely by $b-3$ forms of degree
$b-4$. Assume now that the unstable line $L$ is defined by the
equation $H_i=0$. Then, if $k \neq i$, all the minors defining $I_k$
are divisible by $H_i$, except the one obtained by deleting the $k$-th
and $i$-th rows of $M$; so $s_k$ vanishes at $b-4$ points on $L$.  
\end{remark}

\begin{remark} 
In Proposition \ref{prop:unstable-b-8} 
we denoted the number of unstable lines of
$\mathscr{F}_b$ by $\alpha_1$. Further on, the notation $\alpha_1$
will be reserved to the number of exceptional lines contracted by the first adjunction map $\varphi_{|K_X + D|} \colon X \to X_1$, see \S
\ref{subsec:adjunction theory}. The reason is that when we consider a
general triple plane $f \colon X \to \PP^2$ whose (twisted)
Tschirnhausen bundle is isomorphic to $\mathscr{F}_b$, with $b \geq
7$, these two numbers are in fact the 
same (see \S \ref{subsub:adj}, in particular Proposition \ref{prop:adjunction}).
\end{remark}

\subsection{Criteria for a rank-2 vector bundle to be Steiner}

Here we present two simple criteria to check whether a vector bundle
of rank 2 on $\PP^2$ is a Steiner one. Both of them consist in fixing the numerical
data and adding a single cohomology vanishing. In the second one, the
condition is on a zero-dimensional subscheme from which the bundle is
constructed via the Serre correspondence, provided that the
Cayley-Bacharach property is satisfied.

To state the first result, fix an integer $b \ge 4$ and note that, if $\mF$ is
a Steiner bundle of type $\mF_b$, then 
  \begin{align}     \label{allorasteiner1}
    c_1(\mF)=b-4, \qquad c_2(\mF)={{b-3}\choose 2}
  \end{align}    
and $H^i(\PP^2,\mF(-1))=0$ for all $i$. Likewise, for $b \le 2$
assume that $\mF$ fits into
  \begin{equation}
    \label{costeiner}
  0 \to \mF \to \mO_{\PP^2}(-1)^{4-b} \to \mO_{\PP^2}^{2-b} \to 0.
  \end{equation}
Then,  using the
standard convention on binomial coefficients with negative arguments, we see that 
\eqref{allorasteiner1} still holds; furthermore, we have $H^i(\PP^2, \, \mF(-1))=0$ for
all $i$. Note that $\mF$ fits into \eqref{costeiner} if and
only if $\mF^\vee(-1)$ is of type $\mF_b$.
One may extend the notation $\mF_b$ to all $b$ in $\ZZ$ as a bundle fitting into the long exact sequence
\begin{equation*}
0 \to \mO_{\PP^2}(-1)^{\max(b-4,0)} \to \mO_{\PP^2}^{\max(b-2,0)}
\to \mF_b \to \mO_{\PP^2}(-1)^{\max(4-b,0)} \to
\mO_{\PP^2}^{\max(2-b,0)} \to 0,
\end{equation*} 
where the value $b=3$ corresponds to $\mF_3=\mO_{\PP^2}(-1)
\oplus \mO_{\PP^2}$. 

\begin{proposition} \label{Fsteiner}
  Fix an integer $b \in \ZZ$ and let $\mF$ be a vector bundle of rank
  $2$ on $\PP^2$ satisfying \eqref{allorasteiner1}. Then the following holds$:$
  \begin{enumerate}[itemsep=2pt, label=\bf{\roman*)}]
  \item \label{o:1}   for $b\ge 4$,  the bundle $\mF$ is of type $\mF_b$ if and only if
    $H^0(\PP^2,\mF(-1))=0$. If this happens, then $\mF$ is stable for $b\ge 5;$
  \item \label{o:2} for $b\le 2$, the bundle $\mF^\vee(-1)$ is of type $\mF_b$
    if and only if $H^2(\PP^2,\mF(-1))=0$. If this happens, then $\mF$ is stable for $b \le 1;$
  \item \label{o:3} for $b=3$,  we have $\mF\simeq \mO_{\PP^2}(-1) \oplus
    \mO_{\PP^2}$ if and only if $H^0(\PP^2, \, \mF(-1))=0$
    or, equivalently, $H^2(\PP^2, \, \mF(-1))=0$.
  \end{enumerate}
\end{proposition}

\begin{proof}
  In each case, only one direction needs to be proved. \medskip
  
  \ref{o:1} Let us assume $b\ge 4$ and $H^0(\PP^2,\mF(-1))=0$ and let us show that
  $\mF$ is of the form $\mF_b$.
  First, since $\mF$ is locally free of rank $2$ and $c_1(\mF)=b-4$,
  there is the canonical isomorphism
  \begin{equation*}
  \mF^\vee \simeq \mF(4-b).
  \end{equation*}
  Then, for any integer $t \le 2$, by Serre duality we have
  \begin{equation}
    \label{annullot}
  h^2(\PP^2, \, \mF(-t))=  h^0(\PP^2, \, \mF^\vee(t-3))=h^0(\PP^2, \,\mF(t-b+1))=0,    
  \end{equation}
  because by our assumptions $t-b+1 \le -1$ and already $h^0(\PP^2, \, \mF(-1))=0$.

  Now, using \eqref{allorasteiner1} and the Riemann-Roch theorem we deduce
  $\chi(\PP^2, \, \mF(-1))=0$, so $h^1(\PP^2, \, \mF(-1)) = 0$ because
  we know that $h^0(\PP^2, \, \mF(-1)) = h^2(\PP^2, \, \mF(-1)) = 0$.
  Again by Riemann-Roch, using \eqref{annullot} with $t=2$ we obtain
  $h^1(\PP^2, \,\mF(-2))=b-4$.

Let us look at $h^i(\PP^2, \, \mF)$. First, by using 
  \eqref{annullot} with $t=0$, we see that this vanishes for $i=2$. 
 Now take a line $L$ in $\PP^2$, tensor with $\mF(t)$ the short exact sequence
  \begin{equation} \label{usolaretta}
  0 \to \mO_{\PP^2}(-1) \to \mO_{\PP^2} \to \mO_{L} \to 0
  \end{equation}
  and pass to cohomology. Since we proved that $h^1(\PP^2, \, \mF(-1))=h^2(\PP^2, \, \mF(-2))= 0$, we deduce   $h^1(L, \, \mF(-1)|_L)=0$. Then, considering the short exact sequence 
  \begin{equation*}
  0 \to \mF(t-1)|_L \to \mF(t)|_L \to \mO_{x} \oplus \mO_{x} \to 0
  \end{equation*}
  and using induction on $t$, we obtain
  $h^1(L, \, \mF(t)|_L)=0$ for any $t \ge 0$. Therefore we get
  $h^1(\PP^2, \, \mF)=0$, that in turn yields, again by Riemann-Roch, $h^0(\PP^2, \, \mF)=b-2$.

\medskip
  We can now use Beilinson's theorem, see for instance \cite[Chapter
  2, \S 3.1.3]{OkSchSp88}. The Beilinson table of $\mF$, displaying the values of $h^j(\PP^2,\mF(-i))$, is 
  \begin{table} [H] 
    \begin{center}
      \begin{tabular}{c|c|c|c}
        $ $ & $\mF(-2)$ & $\mF(-1)$ & $\mF$ \\
        \hline $h^2$ & $ 0 $ & $ 0 $ & $ 0 $ \\
        \hline $h^1$ & $b-4$ & $ 0 $ & $ 0 $  \\
        \hline $h^0$ & $ 0 $ & $ 0 $ & $ b-2$ \\
      \end{tabular} \caption{The Beilinson table of $\mF$}
    \end{center} 
  \end{table}
This gives in turn the resolution of $\mF$
\begin{equation} \label{eq:reso-F}
0 \to H^1(\PP^2, \, \mF(-2)) \otimes \mO_{\PP^2}(-1) \to H^0(\PP^2, \, \mF) \otimes \mO_{\PP^2} \to \mF \to 0,
\end{equation}
which has the desired form. In fact, \eqref{eq:reso-F} becomes
\eqref{basic} if we set
\begin{equation} \label{eq:W-U}
W:= H^1(\PP^2, \, \mF(-2))^\vee, \quad U := H^0(\PP^2, \, \mF), \quad \PP^2 = \PP(V).
\end{equation}
The stability of $\mF$ for $b\ge 5$ follows from Hoppe's criterion, see \cite[Lemma 2.6]{Hop84}.
\medskip \medskip

\ref{o:2} Assume now $b \le 2$. Set
$\mF'=\mF^\vee(-1)$ and $b'=6-b$, so that $b'\ge 4$. The Chern classes of$\mF'$ are 
\begin{equation*}
c_1(\mF')=-c_1(\mF)-2=b'-4, \qquad
c_2(\mF')=c_2(\mF)+c_1(\mF)+1={b'-3\choose 2}.
\end{equation*}
Using the assumption $H^2(\PP^2,\mF(-1))=0$ and Serre duality, we get
\begin{equation*}
H^0(\PP^2,\mF'(-1))=H^0(\PP^2,\mF^\vee(-2)) \simeq H^2(\PP^2,\mF(-1))^\vee=0,
\end{equation*}
so by part \ref{o:1} it follows that $\mF'$ is a Steiner bundle of the form $\mF_{b'}$.
\medskip \medskip

\ref{o:3} Finally, assume $b=3$. 
From $H^0(\PP^2, \, \mF(-1))=0$ we deduce
$H^2(\PP^2, \, \mF(-1))=0$ and conversely, because \eqref{annullot} still holds when $(t, \, b)=(1, \,3)$.
We can now conclude by applying \cite[Lemma 3.3]{FV12} to $\mF$.


\end{proof}

\begin{proposition} \label{propsteiner}
  Fix integers $b \ge 5$ and $t\ge 0$, and let $Z \subset \PP^2$ be a $0$-dimensional, local complete intersection subscheme
   of length $l$. Then the following holds$:$
  \begin{enumerate}[itemsep=2pt, label=\bf{\roman*)}]
  \item \label{CB1}  a locally
    free sheaf $\mF$ fitting into
    \begin{equation} \label{eq:Z-fit}
    0 \to \mO_{\PP^2} \xrightarrow{s} \mF(t) \to \mI_{Z/\PP^2}(2t+b-4) \to 0
    \end{equation}
    exists if and only if $Z$ satisfies the Cayley-Bacharach property with
    respect to $\mO_{\PP^2}(2t+b-7)$, i.e. for any subscheme $Z' \subset Z$ of
    length $l-1$ we have
    \begin{equation*}
    h^0(\PP^2, \, \mI_{Z/\PP^2}(2t+b-7))=  h^0(\PP^2, \, \mI_{Z'/\PP^2}(2t+b-7));
    \end{equation*}
  \item \label{CB2} a locally free sheaf $\mF$ as in \ref{CB1} is a Steiner bundle of the form $\mF_b$ if and only if
    \begin{equation}
      \label{Zsteiner}
    l = {{b-3}\choose 2}+t(t+b-4), \qquad H^0(\PP^2, \, \mI_{Z/\PP^2}(t+b-5))=0;      
    \end{equation}
  \item \label{CB3} if \ref{CB1} and \ref{CB2} are satisfied and in addition 
    $h^1(\PP^2,\mI_{Z/\PP^2}(t+b-7))=1$, then the extension \eqref{eq:Z-fit}
  and the proportionality class of the global section $s$ of $\mF(t)$ vanishing at $Z$ are uniquely determined by $Z$.
  \end{enumerate}
\end{proposition}

\begin{proof}
The statement \ref{CB1} follows from \cite[Part II, Theorem
5.1.1]{HuyLehn10}.
\medskip

For \ref{CB2}, take $\mF$ to be a Steiner bundle of
the form $\mF_b$. Then
$c_1(\mF(t))=2t+b-4$ and 
\begin{equation*}
l=c_2(\mF(t))=c_2(\mF)+c_1(\mF)+t^2={{b-3}\choose 2}+t(t+b-4).
\end{equation*}
Also, we have $H^0(\PP^2,\, \mF(-1))=0$, which yields
$H^0(\PP^2,\, \mI_{Z/\PP^2}(t+b-5))=0$.
Conversely, if $Z$ satisfies \eqref{Zsteiner}, then Proposition
\ref{Fsteiner} implies that $\mF$ is of the form $\mF_b$.
\medskip

For \ref{CB3}, by Serre duality we have
\begin{equation} \label{eq:ext-Z}
\Ext^1(\mI_{Z/\PP^2}(2t+b-4), \, \mO_{\PP^2})^\vee \simeq \Ext^1(\mO_{\PP^2},\, \mI_{Z/\PP^2}(2t+b-7)) \simeq H^1(\PP^2, \, \mI_{Z/\PP^2}(2t+b-7)) \simeq \mathbb{C}.
\end{equation}
Since we are assuming that $\mF$ is locally free, the
extension \eqref{eq:Z-fit} has to be non-trivial, and by \eqref{eq:ext-Z} all such non-trivial extensions are equivalent up to a
multiplicative scalar. 
\end{proof}

\section{General triple planes with $p_g=q=0$} \label{sec.general}

\subsection{General triple planes} 

Given a triple plane $f \colon X \to \PP^2$, we denote by $H$ the pullback $H:=f^*L$, where $L \subset \PP^2$ is a
line. The divisor $H$ is ample, as $L$ is ample and $f$ is finite. 

Recall that the Tschirnhausen bundle $\mE$ of $f$ 
is a rank $2$ vector bundle on $\PP^2$ such that $f_* \mO_X \simeq 
\mO_{\PP^2} \oplus \mE$. Proposition \ref{bir-inv-tr} allows us to relate
the invariants of $X$ and $\mE$ as follows.

\begin{proposition} \label{bir-inv}
Let $f \colon X \to \PP^2$ be a triple plane with
Tschirnhausen bundle $\mE$. Then we have:
\begin{align*}
 p_g(X) & =h^0(\PP^2, \, \mE^{\vee}(-3)), \\
 q(X) & = h^1(\PP^2, \, \mE^\vee(-3)), \\
 P_2(X) & = h^0(X, \, 2 K_X)=h^0(\PP^2, \, S^2 \mE^{\vee}(-6)).
\end{align*}
\end{proposition}

\begin{definition} Let $f \colon X \to \PP^2$ be a triple plane and $B \subset
\PP^2$ its branch locus. We say that $f$ is a
\emph{general triple plane} if the following conditions are
satisfied$:$
\begin{enumerate}[itemsep=2pt, label=\bf{\roman*)}]
\item $f$ is unramified over $\PP^2 \backslash B;$
\item $f^*B=2R+R_0$, where $R$ is
irreducible and non-singular and $R_0$ is reduced$;$
\item $f_{|R} \colon R \to B$ coincides with
the normalization map of $B$.
\end{enumerate}
\end{definition}

A useful criterion to check that a triple plane is a general one is provided by the following 
\begin{proposition} 
Let $f \colon X \to \PP^2$ be a triple plane with $X$ smooth. Then either $f$ is general or $f$ is a Galois cover. In the last case, $f$ is totally ramified over a smooth branch locus.
\end{proposition}
\begin{proof}
See \cite[Theorems 2.1 and 3.2]{Tan02}.
\end{proof}

Hence Theorem \ref{thm triple:1} shows that, if $S^3 \mE^{\vee} \otimes \wedge^2 \mE$ is globally generated, the cover associated with a general section $\eta \in H^0(\PP^2, \, S^3 \mE^{\vee} \otimes \wedge^2 \mE)$ is a general triple plane as soon as it is not totally ramified.

Since the curve $R$ is the ramification divisor of $f$ and  the
ramification is simple, we have 
\begin{equation} \label{eq:Hurwitz-X}
K_X = f^*K_{\PP^2}+R=-3H+R.
\end{equation}
Moreover, by \cite[Proposition 4.7 and Lemma 4.1]{Mi85}, we obtain

\begin{proposition} \label{branch-locus}
Let $f \colon X \to \PP^2$ be a general triple plane with Tschirnhausen bundle $\mE$ and define
\begin{equation*}
b:=-c_1(\mE), \quad h:=c_2(\mE).
\end{equation*}
Then the branch curve $B$ has degree $2b$ and contains $3h$
ordinary cusps and no further singularities. Moreover the cusps are
exactly the points where $f$ is totally ramified.
\end{proposition}

Moreover, in view of \cite[Lemma 5.9]{Mi85} and \cite[Corollary 2.2]{CasEk96},
we have the following information on $R$ and $R_0$.

\begin{proposition} \label{RR_0} 
The curves $R$ and $R_0$ are both smooth and isomorphic to the normalization of $B$. Furthermore, they are tangent
at the preimages of the cusps of $B$ and they do not meet elsewhere. Finally, the ramification divisor $R$ is very ample on $X$.
\end{proposition}

This allows us to compute the intersection numbers of $R$ and $R_0$ as follows.

\begin{proposition} \label{inters-ram}
We have
\begin{equation} \label{eq:inters-ram}
R^2 = 2b^2 - 3h, \quad RR_0 = 6h, \quad R_0^2= 4b^2 - 12h.
\end{equation}
\end{proposition}
\begin{proof}
Projection formula yields
\begin{equation*}
 R(2R+R_0)=R(f^*B)=(f_*R)B=B^2=4b^2.
\end{equation*}
By Proposition \ref{RR_0} it follows $RR_0=6h$.
So $2R^2=4b^2-RR_0=4b^2-6h$, which gives the first equality.
From $f^*B=2R+R_0$ we deduce $(2R+R_0)^2=3B^2=12 b^2$, so $R_0^2=12b^2-4R^2-4RR_0=4b^2-12h$.
\end{proof}

\begin{corollary} \label{cor:num-cusps}
We have $3h \geq \frac{2}{3}b^2$.
\end{corollary}
\begin{proof}
Since the divisor $R$ is very ample, the Hodge Index theorem implies
$R^2R_0^2 \leq (RR_0)^2$ and the claim follows.
\end{proof}

\begin{remark} \label{rmk:Bron}
Proposition \ref{inters-ram} and Corollary \ref{cor:num-cusps} were
already established by Bronowski in \cite{Br42}. Note that the (very) ampleness of $R$ implies $R^2 >0$, that is $3h < 2b^2$. In \cite{Br42}, it is also stated that the stronger inequality $3h \leq b^2$, or
equivalently $R_0^2 \geq 0$, holds. This is actually false, and counterexamples will
be provided by our surfaces of type \ref{VII}, see \S \ref{subsec:VII}. Bronowski's mistake is at page $28$ of his paper, where he assumes that one can find a curve algebraically equivalent to $R_0$ and distinct from
it; of course, when $R_0^2 <0$ such a curve cannot exist.
\end{remark}

\begin{proposition} \label{prop.K+2H}
Let $f \colon X \to \PP^2$ be a general triple plane with $q(X)=0$. If $K_X^2 \neq 8$ then $D:=K_X + 2H$ is very ample.
\end{proposition}
\begin{proof}
Since $(2H)^2=12$, by \cite[Theorem 18.5]{Fu90} $D$ is very ample,
unless there exists an effective divisor $Z$ such that $HZ$=1 and
$Z^2=0$. By the projection formula we have
\begin{equation*}
1=HZ=(f^*L)Z=L(f_*Z),
\end{equation*}
hence $f_*Z \subset \PP^2$ is a line. On the other hand,
$HZ=1$ implies that the restriction of $f$ to $Z$ is an isomorphism,
 so $Z$ is a smooth and irreducible rational curve. Since $Z^2=0$, the surface $X$
 is birationally ruled and $Z$ belongs to the ruling. Moreover, all the curves in the ruling are irreducible: in fact, if $Z$ were algebraically equivalent to $Z_1+ Z_2$, then we would obtain
\begin{equation*}
1=HZ = HZ_1+HZ_2,
\end{equation*} 
contradicting the ampleness of $H$. Summing up, $X$ is a minimal,
geometrically ruled surface over a smooth curve; since $q(X)=0$, this
curve is isomorphic to $\PP^1$, that is $X$ is isomorphic to $\FF_n$
for some $n$ and, in particular, $K_X^2=8$. 
\end{proof}

When $D=K_X+2H$ is very ample on $X$ we can study the adjunction maps associated with $D$. Using Proposition \ref{prop:contracted}, we obtain

\begin{proposition} \label{prop:phi_n}
Assume $q(X)=0$ and $K_X^2 \neq 8$ and let $\varphi_n \colon X_{n-1} \to X_n$ be the $n$-th adjunction map with respect to the very ample divisor $D=K_X+2H$. Then $\varphi_n$ is an isomorphism when $n$ is even, whereas when $n$ is odd $\varphi_n$ contracts exactly the $(-1)$-curves $E \subset X$ such that $HE=(n+1)/2$.
\end{proposition}


\subsection{The Tschirnhausen bundle in case $p_g=q=0$} 

Let $f \colon X \to \PP^2$ be a general triple plane
with Tschirnhausen bundle $\mathscr{E}$ and let $B$ be the branch locus of $f$. Recall that, by
Proposition \ref{branch-locus}, the curve $B$ has degree
$2b$ and contains $3h$ ordinary cusps as only singularities.

\newcounter{cover}
\renewcommand{\thecover}{\textnormal{\Roman{cover}}}

\begin{proposition} \label{prop:numerical}
If $\chi(\mO_X)=1$, that is
$p_g(X)=q(X)$, then we have at most the following 
possibilities for the numerical invariants $b, \, h, \, K_X^2, \, g(H):$
\begin{table}[H]
\begin{center}
\begin{tabular}{c|c|c|c|c}
Case & $b$ & $h$ & $K_X^2$ & $g(H)$ \\
 \hline
\refstepcounter{cover} \thecover \label{I} & $2$ & $1$ & $8$ & $0$\\
\hline
\refstepcounter{cover} \thecover \label{II} & $3$ & $2$ & $3$& $1$ \\
\hline
\refstepcounter{cover} \thecover \label{III} & $4$ & $4$ & $-1$ & $2$\\
\hline
\refstepcounter{cover} \thecover \label{IV} & $5$ & $7$ & $-4$ & $3$\\
\hline
\refstepcounter{cover} \thecover \label{V} & $6$ & $11$ & $-6$ & $4$ \\
\hline
\refstepcounter{cover} \thecover \label{VI} & $7$ & $16$ & $-7$ & $5$ \\
\hline
\refstepcounter{cover} \thecover \label{VII} & $8$ & $22$ & $-7$ & $6$\\
\hline
\refstepcounter{cover} \thecover  \label{VIII} & $9$ & $29$ & $-6$ & $7$\\
\hline
\refstepcounter{cover} \thecover 
& $10$ & $37$ & $-4$ & $8$\\
\hline
\refstepcounter{cover} \thecover 
& $11$ & $46$ & $-1$ & $9$\\
\hline
\refstepcounter{cover} \thecover  
& $12$ & $56$ & $3$ & $10$ \\
\hline
\refstepcounter{cover} \thecover \label{XII} & $13$ & $67$ & $8$ & $11$
\end{tabular}
\end{center}
\caption{Possible numerical invariants for a general triple plane with $\chi(\mO_X)=1$}
\end{table}
\end{proposition}
\begin{proof}
Using the projection formula we obtain
\begin{equation} \label{eq:HR}
HR=(f^*L)R=L(f_*R)=LB=2b.
\end{equation}
Since $K_X= -3H +R$ and $H^2=3$ it follows $K_X H=2b-9$, hence
$g(H)=b-2$. Using the \emph{formule di corrispondenza}
(cf. \cite[\S V]{Iv70}) we infer
\begin{equation*}
\left \{ \begin{array}{l}
9h+3 = 4b^2 -6b +K_X^2 \\
2h - 4 = b^2 - 3b.
\end{array} \right.  
\end{equation*}
Therefore $h = \frac{1}{2}(b^2-3b+4)$ and $b^2 -15b +42 - 2K_X^2=0$.
Imposing that the discriminant of this  quadratic equation is
non-negative, we get $K_X^2 \geq -7$; on the 
other hand, the Enriques-Kodaira classification and the Miyaoka-Yau inequality imply that any surface with $p_g=q$ satisfies $K_X^2 \leq
9$, see \cite[Chapter VII]{BHPV04}, so  $-7 \leq K_X^2 \leq 9$. Now a case-by-case analysis
concludes the proof.
\end{proof}
Note that the previous proof shows that
\begin{equation} \label{eq:c1-c2-E}
c_1(\mE)=-b, \quad c_2(\mE)=\frac{1}{2}(b^2-3b+4).
\end{equation}
Moreover, using \eqref{eq:Hurwitz-X}, \eqref{eq:HR} and the first equality in \eqref{eq:inters-ram}, we obtain
\begin{equation} \label{eq:KR}
K_XR = -3HR + R^2 =2b^2-6b-3h.
\end{equation}
From now on, we will restrict ourselves to the case $p_g(X)=q(X)=0$,
 that is, in terms of the Tschirnhausen bundle $\mE$, we suppose
 $h^1(\PP^2, \,\mE)=0$ and $h^2(\PP^2, \, \mE)=0$. 
 Furthermore,
 we will use without further mention the natural isomorphism
\begin{equation*} 
\mE^\vee \simeq \mE(b).
\end{equation*}

\begin{theorem} \label{teo:res}
  Let $f \colon X \to \PP^2$ be a general triple plane with $p_g=q=0$ and let $\mE$ be the corresponding Tschirnhausen bundle. With the notation of Proposition \emph{\ref{prop:numerical}}, the following holds$:$
\begin{enumerate}[itemsep=2pt, label=\bf{\roman*)}]
\item in case \ref{I}, $\mE \simeq \mO_{\PP^2}(-1) \oplus
  \mO_{\PP^2}(-1);$ 
\item
  in case \ref{II}, $\mE \simeq \mO_{\PP^2}(-1) \oplus \mO_{\PP^2}(-2);$
\item
  in case \ref{III}, $\mE \simeq \mO_{\PP^2}(-2) \oplus \mO_{\PP^2}(-2);$
\item
in cases \ref{IV} to \ref{XII}, the vector bundle $\mE$ is stable and
has a sheafified minimal graded free resolution of the form 
  \begin{equation*}
  0  \to \mO_{\PP^2}(1-b)^{b-4} \to \mO_{\PP^2}(2-b)^{b-2} \to \mE \to 0.
  \end{equation*}
In particular, $\mE(b-2)$ is a rank $2$ Steiner bundle on $\PP^2$, see \S \emph{\ref{subsec:steiner}}.
\end{enumerate}
\end{theorem}
\begin{proof}
  Setting $\mF:= \mE(b-2)$, by  using \eqref{eq:c1-c2-E} we obtain
  \begin{equation} \label{eq:c1-c2-F}
    c_1(\mF)=b-4, \quad c_2(\mF)= {b-3 \choose 2}.
  \end{equation}
  Now Proposition \ref{bir-inv} allows us to calculate the cohomology
  groups of $\mF(-i)$, for $i=0,1,2$. We have 
  \begin{align}
    \label{keyvanishing} & h^0(\PP^2, \, \mF(-1)) = h^0(\PP^2, \, \mE(b-3))=h^0(\PP^2, \, \mE^{\vee}(-3))=p_g(X)=0, \\
     \nonumber & h^1(\PP^2, \mF(-1))=h^0(\PP^2, \, \mE(b-3))= h^1(\PP^2, \, \mE^{\vee}(-3))=q(X)=0.
  \end{align}

  Let us now check cases \ref{I} to \ref{III}. By
  \eqref{keyvanishing}, we can apply 
  \cite[Lemma 3.3]{FV12} to $\mE(1)$ in cases \ref{I}
  and \ref{II}, and to $\mE(2)$ in case \ref{III}. The result then 
  follows. \medskip

  In the cases \ref{IV} to \ref{XII}, the conditions
  \eqref{eq:c1-c2-F} and \eqref{keyvanishing} say that Proposition
  \ref{Fsteiner} applies, so $\mF$ is a Steiner bundle of the form
  $\mF_b$. This gives the desired resolution of $\mE$.
\end{proof}

\begin{corollary} \label{cor:Tsch1}
  In cases $\ref{I}$ to $\ref{III}$, general triple planes
  $f \colon X \to \PP^2$ do exist and $X$ is a rational surface.
\end{corollary}
\begin{proof}
  Let us consider case $\ref{I}$. By Theorem \ref{teo:res} we
  have $S^3\mE^{\vee} \otimes \wedge^2 \mE \simeq 
  \mO_{\PP^2}(1)^4$ which is globally generated, so the triple
  cover exists by Theorem \ref{thm triple:1}. Using Proposition
  \ref{bir-inv} we obtain
  \begin{equation*}
    P_2(X)=h^0(\PP^2, \, S^2\mE^{\vee}(-6))=h^0(\PP^2, \,
    \mO_{\PP^2}(-4)^3)=0,
  \end{equation*}
  hence Castelnuovo's Theorem (cf. \cite[Chapter VI, \S 3]{BHPV04}) implies that
  $X$ is a rational surface.
  The argument in cases \ref{II} and \ref{III} is the same.
\end{proof}

\subsection{The projective bundle associated with a triple plane} \label{sub:proj.bundle}

\subsubsection{Triple planes and direct images}
\label{L}

Let $f \colon X \to \PP^2$  be a general triple plane with
$p_g=q=0$ and Tschirnhausen bundle  $\mE$. We assume $b \geq 5$ and we write  $\mF$ as before in order to denote the bundle $\mE(b-2)$. Sometimes, if we want to emphasize the role of $b$, we will use the notation $\mF_b$ instead of $\mF$.
The rest of the notation in this paragraph is borrowed from \S \ref{subsec:steiner}.

 As shown in Theorem
\ref{teo:res}, $\mF$ is a Steiner bundle of rank $2$. Theorem \ref{thm triple:2} implies that $X$ can be realized as a Cartier divisor in
$\sP$, such that the restriction of $\mathfrak{p}\colon \PP(\mF) \to \PP^2$ to $X$ is our covering map $f$.
More precisely, recall that we denote by  $\xi$ the tautological
relatively ample line bundle on $\PP(\mF)$ and by $\ell$ the pull-back
to $\PP(\mF)$ of a line in $\PP^2$. Then the identification
\begin{equation} \label{eq:S3F(6-b)}
S^3 \mE^\vee \otimes \wedge ^2 \mE \simeq S^3 \mF \otimes \mO_{\PP^2}(6-b)
\end{equation}
shows that $X$ lies in the complete linear system $|\mL|$, with
\begin{equation} \label{eq:linear-L}
\mL = \mO_{\sP}(3\xi+(6-b)\ell).
\end{equation}
Recall also the notation $U= H^0(\PP^2,\mF)$, and consider the morphism $\mathfrak{q} \colon \PP(\mF) \to
\PP(U)\simeq \PP^{b-3}$ associated with $|\mO_{\PP(\mF)}(\xi)| \simeq \PP (U)$.
Setting
\begin{equation*}
\mR := \mathfrak{q} _*(\mO_{\sP}((6-b)\ell)),
\end{equation*}
the projection formula yields natural identifications
\begin{equation} \label{eq:S3E}
\begin{split}
 H^0(\PP^2,S^3 \mE^\vee \otimes \wedge^2 \mE
 ) &  \simeq H^0(\PP^2, \, S^3 \mF(6-b)) \\
 & \simeq H^0(\sP, \, \mL)\simeq H^0(\PP^{b-3}, \,\mR(3)).
\end{split}
\end{equation}
In order to get information on the sheaf $\mR$, it is useful to
consider the Koszul resolution of $\sP$ in $\PP(V) \times \PP(U)
\simeq \PP^2 \times \PP^{b-3}$, which is given taking exterior powers
of \eqref{prekoszul}. This reads
\begin{equation}
  \label{KOSZUL}
  \wedge^\bullet (W^\vee \ts \mO_{\PP^2 \times \PP^{b-3}}(-1, \, -1 )) \to \mO_{\sP} \to 0
\end{equation}
with $W^\vee=H^1(\PP^2, \, \mF(-2))$, see Proposition \ref{prop:P-and-P'} and \eqref{eq:W-U}.  
We will write $\mK_i$ for the image of the $i$-th differential 
  \begin{equation*}
  d_i : (\wedge^i W^\vee) \otimes \mO_{\PP^2 \times \PP^{b-3}}(-i, \, -i) \to
  (\wedge^{i-1} W^\vee) \otimes \mO_{\PP^2 \times \PP^{b-3}}(-i+1, \, -i+1)
  \end{equation*}
of the complex \eqref{KOSZUL}. Moreover, we will often use the relation
\begin{equation} \label{eq:Riq}
R^i \mathfrak{q} _* (\mO_{\PP^2 \times \PP^{b-3}}(n_1, \, n_2)) = H^i(\PP^2,
\, \mO_{\PP^2}(n_1)) \otimes \mO_{\PP^{b-3}}(n_2), \quad i \in \NN, \; n_1, \, n_2 \in \ZZ.
\end{equation}

We finally define $Y \subset \PP^{b-3}$ as the image of $\mathfrak{q}$; then the support of $\mR$ is contained in $Y$. 
In \S \ref{subsub:adj} we shall see that, if $b \geq 6$, the morphism $\mathfrak{q}   \colon \sP \to
\PP^{b-3}$ is
generically injective, so $Y \subset \PP^{b-3}$ is a (possibly
singular) irreducible threefold which is generated by the $3$-secant lines to
the canonical curves of genus $g(H)$ representing in $\PP^{b-3}$ the net $|H|$ inducing the triple cover.
The threefold $Y$ is defined by the $3\times 3$ minors of the matrix 
$N$ appearing in the resolution of $\mathfrak{q} _*(\mO_{\PP(\mF)}(\ell))$, namely
\begin{equation*}
 \mO_{\PP^{b-3}}(-1)^{b-4} \xrightarrow{N} \mO_{\PP^{b-3}}^3 \to
 \mathfrak{q}_*(\mO_{\PP(\mF)}(\ell)) \to 0.
\end{equation*}

\subsubsection{Adjunction maps and projective bundles} \label{subsub:adj}

We use the notation of \S \ref{a}. Recall that the canonical line bundle of $\PP(\mF)$ is 
\begin{equation} \label{eq.canPF}
  \omega_{\sP} \simeq \mO_{\sP}(-2 \xi + (b-7)\ell),
\end{equation}
see for instance \cite[Ex. 8.4 p. 253]{H77}. The following result provides a link between the adjunction theory and the  vector bundles techniques used in this paper.

\begin{lemma} \label{lemma.adj}
 Let $f \colon X \to \PP^2$ be a general triple
  plane with $p_g(X)=q(X)=0$. Then $\mathfrak{q} |_X$ coincides with the first
  adjoint map $\varphi_{|K_X + H|} \colon X \to \PP^{b-3}$ 
  associated with the ample divisor $H$.
\end{lemma}
\begin{proof}
  Since $H$ is ample, by Kodaira vanishing theorem we have $h^1(X, \, K_X+H)=h^2(X, \, K_X+H)=0$, so
  Riemann-Roch theorem gives $h^0(X, \, K_X+H)=g(H)=b-2$. Therefore it suffices to show that
  \begin{equation*}
    \omega_X \otimes \mO_X(H) \simeq \mO_{\sP}(\xi)|_X.
  \end{equation*}
 The adjunction formula, together with \eqref{eq:linear-L}and \eqref{eq.canPF},
  yields
  \begin{equation*}
      \omega_X = (\omega_{\sP} \otimes \mL)|_X
       \simeq  \mO_{\sP}(\xi-\ell)|_X.
  \end{equation*}
Since $\ell|_X = \mO_X(H)$, the claim follows.
\end{proof}

\begin{lemma} \label{contracts}
  The morphism $\mathfrak{q}  \colon \sP \to \PP^{b-3}$ contracts precisely the negative sections
  of the Hirzebruch surfaces of the form $\PP(\mF|_L)$, where $L$ is an unstable
  line of $\mF$. Moreover, if $b \ge 6$ then $\mathfrak{q}$ is birational onto its image $Y
  \subseteq \PP^{b-3}$, which is a birationally ruled threefold of degree ${b-4 \choose 2}$.
\end{lemma}

\begin{proof} We first show that $\mathfrak{q}$ contracts the negative sections.
If $L$ is an unstable line of $\mF$, then $\mF|_L \simeq \mO_L \oplus
\mO_L(b-4)$, so $\PP(\mF|_L)$ is isomorphic to the Hirzebruch surface $\FF_{b-4}$. The divisor $\mO_{\PP(\mF)}(\xi)$ cuts on $\PP(\mF|_L)$ the complete linear system $|\frc+(b-4) \frf|$; therefore $\mO_{\PP(\mF)}(\xi) \cdot \frc=0,$ that is $\mathfrak{q}$ contracts $\frc$. In particular, this means that the image of $\PP(\mF|_L)$ via $\mathfrak{q}$ is a cone $S(0, \, b-4) \subset \PP^{b-3}$.

Conversely, we now show that $\mathfrak{q}$ is injective on the
complement of the set of negative sections over unstable lines. More precisely, assuming 
that $x_1$ and $x_2$ are points of $\sP$ not
separated by $\mathfrak{q}$, we will prove that $x_1$ and $x_2$ lie in one of such sections. In fact, since $\mO_{\PP(\mF)}(\xi)$ is very ample when
restricted to the fibres of $\mathfrak{p} \colon \PP(\mF) \to \PP^2$, the points
$\mathfrak{p}(x_1)$ and $\mathfrak{p}(x_2)$ are distinct. Let $L$ be the unique line through
$\mathfrak{p}(x_1)$ and $\mathfrak{p}(x_2)$ and let us restrict  $\mathfrak{q}$  to $\PP(\mF|_L)$. If 
$L$ were not unstable for $\mF$, then $\mF|_L\simeq
\mO_L(a) \oplus \mO_L(b-4-a)$ with $a > 0$ and $b-4-a >0$ (cf. the
proof of Lemma \ref{lem:unstable-line}), and in this situation the restriction of $\mO_{\PP(\mF)}(\xi)$ to $\PP(\mF|_L)$ would be very ample, hence $\mathfrak{q}$ would separate 
$x_1$ and $x_2$, contradiction. This shows that $L$ is
necessarily an unstable line for $\mF$ and that moreover $x_1$ and $x_2$ must both lie on the unique negative section of $\PP(\mF|_L) \simeq \FF_{b-4}$. The same argument also works  if $x_1$ and $x_2$ are infinitely near, and this ends the proof of the first statement.

Regarding the second statement, the subscheme $\mW(\mF_b)$ of
unstable lines has positive codimension in  $\PD^2$ for $b \geq 6$, see \S
\ref{b}. Then $\mathfrak{q}$ is birational onto its image $Y \subset \PP^{b-3}$, and this in particular says that 
$Y$ is a birationally ruled threefold in $\PP^{b-3}$ (of course for $b=6$ the
image is the whole $\PP^3$).

We can now use \eqref{eq:c1-c2-F} and the Chern equation for $\PP(\mF_b)$ in order to compute the degree of $Y$, obtaining
\begin{equation*}
\deg Y= \xi^3=\mathfrak{p}^*(c_1(\mF_b)^2-c_2(\mF_b)) \xi = (b-4)^2-{{b-3}\choose 2}={{b-4}\choose 2}.
\end{equation*}
\end{proof}

\begin{lemma} \label{lem:base-locus}
Let $\mL = \mO_{\PP(\mF)}(3 \xi + (6-b) \ell)$ and let $\frc$ be the negative section of 
the Hirzebruch surface $\PP(\mF|_L)$, where $L$ is an unstable line for $\mF$. If $b \geq 7$, then $\frc$ is contained in 
the base locus of $|\mL|$. 
\end{lemma}
\begin{proof}
By restricting any element of $|\mL|$ to  $\PP(\mF|_L)$ we obtain a divisor $\mL'$ linearly equivalent to 
\begin{equation*}
3(\frc + (b-4) \frf)+ (6-b) \frf = 3\frc + (2b-6) \frf.
\end{equation*}
We have $\mL' \frc = 3 (4-b) + (2b-6) =6-b$, so if $b \geq 7$ we have
$\mL'  \frc <0$ and this in turn implies that $\frc$ is a component of $\mL'$.
Hence $\frc$ is contained in every element of the linear system $|\mL|$.
\end{proof}
Let us come back now to our general triple planes $f \colon X \to \PP^2$.

\begin{proposition} \label{prop:adjunction}
If $b \geq 7$ then the first adjoint map $\varphi_{|K_X + H|} \colon X
\to \PP^{b-3}$ is a birational morphism onto its image $X_1 \subset
\PP^{b-3}$. Furthermore, $X_1$ is a smooth surface and $\varphi_{|K_X
  +H|}$ contracts precisely the $(-1)$-curves $E$ in $X$ such that
$HE=1$. There is one, and only one, curve with this property for each unstable line of $\mF$.
\end{proposition}

\begin{proof}
 By Lemma \ref{contracts} the map $\mathfrak{q}  \colon \PP(\mF) \to \PP^{b-3}$  is birational onto its image and contracts precisely the negative sections of
  $\PP(\mF|_L)$, where $L$ is an unstable line of $\mF$; let $E$ be
  one of these sections. In view of Lemma \ref{lemma.adj} we have  $\varphi_{|K_X + H|}=\mathfrak{q} |_X$, and moreover by Lemma \ref{lem:base-locus} the curve $E$ is contained in $X$, because $X \in |\mL|$ by construction (see \S \ref{sub:proj.bundle}).
We have $f = \mathfrak{p}|_{X}$, hence $f|_{E}=\mathfrak{p}|_{E}$ and, since $\mathfrak{p}|_{E} \colon E \to L$ is an isomorphism, by the projection formula we obtain 
\begin{equation*}
HE = f^*L \cdot E = L \cdot f_* E = L^2=1.
\end{equation*}
Finally, each Hirzebruch surface $\PP(\mF|_L)$ contains precisely one negative section,  
so we are done.
\end{proof}

\begin{remark} \label{rem:H}
When $b \geq 7$, Proposition \ref{prop:adjunction} will allow us to apply the iterated adjunction process described in \S \ref{subsec:adjunction theory} starting from $D=H$, even if $H$ is ample but \emph{not} very ample. 
\end{remark}

\begin{remark} \label{rem:cubic}
Proposition \ref{prop.TypeV} will show that  $\varphi_{|K_X +H|}$ is birational also for $b=6$: more precisely, in this case $X$ is the blow-up at nine points of a cubic surface $S \subset \PP^3$, and $\varphi_{|K_X +H|}$ is the blow-down morphism.
In fact, $\mW(\mF_6)$ is a smooth conic in $\PD^2$,
cf. \S \ref{b}. If $L$ is an unstable line of $\mF_6$, namely a line tangent to this conic, we have
$\PP(\mF_6|_L)\simeq \mathbb{F}_2$ and $\mathfrak{q} \colon \PP(\mF) \to \PP^3$ contracts the unique negative section of this Hirzebruch surface to a point. The locus of points in $\PP^3$
constructed in this way is a twisted cubic $C$, the map
$\mathfrak{q}$ is the blow-up of $\PP^3$ at $C$ and the nine points that we blow-up in $S$ consist of the subset $S \cap C$. 
\end{remark}

\section{The classification in cases \ref{I} to \ref{VII}}
\label{sec:class}

Since all the  triple planes considered in the sequel are general, for the sake of brevity the 
word \emph{general} will be from now on omitted. 

\subsection{Triple planes of type \ref{I}}

In this case the invariants are
\begin{equation*}
  K_X^2=8, \quad b=2, \quad h=1, \quad g(H)=0
\end{equation*}
and the Tschirnhausen bundle splits as $\mE=\mO_{\PP^2}(-1) \oplus \mO_{\PP^2}(-1)$.
The existence of these triple planes follows from Corollary
\ref{cor:Tsch1}, whereas Proposition \ref{prop.TypeI} below provides their complete classification.

\begin{proposition} \label{prop.TypeI} 
  Let $f \colon X \to \PP^2$ be a triple plane of type
  \emph{\ref{I}}. Then $X$ is isomorphic to the cubic scroll $S(1, \,
  2) \subset \PP^4$ and $f$ is the projection of this scroll from a
  general line of $\PP^4$.  
\end{proposition}
\begin{proof}
By Proposition \ref{RR_0} we know that $R$ is very ample, and by
\eqref{eq:KR} we have $K_XR=-7$. Therefore no multiple of $K_X$ can be
effective and $X$ is a rational surface, as predicted by Corollary
\ref{cor:Tsch1}.  The curve $R$ is the normalization of $B$
(Proposition \ref{RR_0}), which is a tricuspidal quartic curve
(Proposition \ref{branch-locus}), hence $g(R)=0$. Then by the first
statement in Theorem \ref{adjunction-thm} we  get   
\begin{equation*}
    \textrm{dim\,}|K_X +R| = g(R)+p_g(X)-q(X)-1 =-1,
  \end{equation*}
that is $|K_X+R|= \emptyset$. The condition $K_X^2=8$ implies that the
$X$ is not isomorphic to $\PP^2$ so, again by Theorem
\ref{adjunction-thm}, part (A), it must be a rational normal scroll,
with the scroll structure arising from the embedding given by
$|R|$. By the first equality in \eqref{eq:inters-ram} we have $R^2=5$,
and there are two different kind of smooth rational normal scrolls of
dimension $2$ and degree $5$, namely 
\begin{itemize}
\item $S(1, \, 4)$, that is $\FF_3$ embedded in $\PP^6$ via $|\frc + 4\frf|$;
\item $S(2, \, 3)$, that is $\FF_1$ embedded in $\PP^6$ via $|\frc + 3\frf|$.
\end{itemize}
In the former case, using \eqref{eq:Hurwitz-X} we obtain $H=\frc+3
\frf$, which is not ample on $\FF_3$; so this case cannot occur. 
In the latter case we have $H=\frc+2 \frf$, that is very ample and
embeds $\FF_1$ in $\PP^4$ as a cubic scroll $S(1, \, 2)$. The triple
plane is now obtained by taking the morphism to $\PP^2$ associated
with a general net of curves inside $|H|$, which corresponds to the
projection of $S(1, \, 2)$ from a general line of $\PP^4$. 
\end{proof}

\begin{remark} \label{rmk:I}
Another description of triple planes of type \ref{I} is the following.  Let
$X'$ be the Veronese surface, embedded in the Grassmannian $\Gr(1, \,
\PP^3)$ as a surface of bidegree $(3, \,1)$, see \cite[Theorem 4.1
$(a)$]{G93}. There is a family of $1$-secant planes to $X'$;
projecting from one of these planes, we 
  obtain a birational model of a triple plane $f \colon X \to \PP^2$
  of type \ref{I} (in fact, $X$ 
  is the blow-up of $X'$ at one point).
\end{remark}

\subsection{Triple planes of type \ref{II}} 

In this case the invariants are
\begin{equation*}
  K_X^2=3, \quad b=3, \quad h=2, \quad g(H)=1
\end{equation*}
and the Tschirnhausen bundle splits as $\mE=\mO_{\PP^2}(-1) \oplus \mO_{\PP^2}(-2)$.
The existence of these triple planes follows from Corollary
\ref{cor:Tsch1}, whereas Proposition \ref{prop.TypeII} below provides their complete classification.

\begin{proposition} \label{prop.TypeII} 
  Let $f \colon X \to \PP^2$ be a triple plane of type \emph{\ref{II}}. Then
  $X$ is isomorphic to a smooth cubic surface  $S \subset \PP^3$ and
  $f$ is the projection of $S$ from a general point of $\PP^3$. The
  branch locus $B$ is a 
  sextic plane curve with six cusps lying on a conic. 
\end{proposition}
\begin{proof}
  By Proposition \ref{prop.K+2H}, the divisor $D:=K_X+2H$ is very
  ample. Using $K_X H=2b-9=-3$ (see the proof of Proposition \ref{prop:numerical}), we obtain 
  \begin{equation*}
 D^2=(K_X+2H)^2= K_X^2+4K_X H+4H^2=3-12+12 =3,   
     \end{equation*}
  hence the map $\varphi_{|D|} \colon X \to  \PP^3$
  is an isomorphism onto a smooth cubic surface $S$. The statement
  about the position of the cusps in the branch locus is a well-known
  classical result, see \cite[p.320]{Zar29}. 
\end{proof}

\begin{remark} \label{rmk:II}
  Other descriptions of triple planes of type \ref{II} are the following.
\begin{itemize}  
 \item Let $X'$ be a smooth Del Pezzo surface of degree $5$, embedded
   in $\Gr(1, \, \PP^3)$ as a surface of 
  bidegree $(3, \,2)$, see \cite[Theorem 4.1 $(b)$]{G93}. There is a
  family of $2$-secant planes to $X'$; projecting from one of 
  these planes, we obtain a birational model of a triple plane
  $f\colon X \to \PP^2$ of type \ref{II} (in fact, $X$ is the blow-up
  of $X'$ at two points). 
\item Let $X'$ be a smooth Del Pezzo surface of degree $6$, embedded
  in $\Gr(1, \, \PP^3)$ as a surface of 
  bidegree $(3, \,3)$, see \cite[Theorem 4.1 $(d)$]{G93}. There is a
  family of $3$-secant planes to $X'$; projecting from one of 
  these planes, we obtain a birational model of a triple plane
  $f\colon X \to \PP^2$ of type \ref{II} (in fact, $X$ is the blow-up of $X'$ at three points).

\end{itemize}
\end{remark}

\subsection{Triple planes of type \ref{III}}

In this case the invariants are
\begin{equation*}
  K_X^2=-1, \quad b=4, \quad h=4, \quad g(H)=2
\end{equation*}
and the Tschirnhausen bundle splits as $\mE=\mO_{\PP^2}(-2) \oplus \mO_{\PP^2}(-2)$.
The existence of these triple planes follows from Corollary
\ref{cor:Tsch1}, whereas Proposition \ref{prop.TypeIII} below provides
their complete classification. 

\begin{proposition} \label{prop.TypeIII}
  Let $f \colon X \to \PP^2$ be a triple plane of type \emph{\ref{III}}.
  Then $X$ is a blow-up at $9$ points $\sigma \colon X \to \FF_n$  of
  a Hirzebruch surface $\FF_n$, with $n \in \{0, \, 1, \, 2, \, 3 \}$,
  and 
  \begin{equation} \label{eq:type-III}
    H= 2\frc + (n+3) \frf-\sum_{i=1}^{9} E_i.
  \end{equation}
\end{proposition}
\begin{proof}
  By Proposition \ref{prop.K+2H}, the divisor $D:=K_X+2H$ is very ample. We have
  \begin{equation*}
    \left( \begin{array}{cc} D^2 & K_X D \\
        K_X D & K_X^2\\
      \end{array} \right) =
    \left( \begin{array}{cc} \; 7 & -3 \\
        -3 &  -1\\
      \end{array} \right),
  \end{equation*}
  in particular $K_X D<0$ shows that $X$ is a rational surface.
  By Serre duality and Kodaira vanishing theorem we have $h^1(X, \, D) = h^1(X, \, -2H)=0$, 
  and analogously $h^2(X, \, D)=h^0(X, \, -2H)=0$, so by the Riemann-Roch theorem we obtain
\begin{equation*}  
h^0(X, \, D)=\chi(X, \, D)=\frac{D(D-K_X)}{2}+ \chi(\mO_X)=6.   
\end{equation*}
  The morphism $\varphi_{|D|} \colon X \to X_1 \subset \PP^5$ is an
  isomorphism of $X$ onto its image $X_1$, which is a surface of
  degree $7$ with $K_{X_1}^2=-1$. Embedded projective varieties of
  degree at most $7$ are classified in \cite{Io84}; in particular, the
  table at page 148 of that paper shows that $X_1$ is a blow-up  at 9
  points $\sigma \colon X_1 \to \FF_n$, with $n \in \{0, \, 1,
  \, 2, \, 3 \}$, and that 
  \begin{equation*} 
  D= 2\frc + (n+4) \frf-\sum_{i=1}^{9} E_i.
  \end{equation*}
  Using $2H=D-K_X$, we obtain \eqref{eq:type-III}. 
  \end{proof}

\begin{remark} \label{rmk:III}
  When $n=0$, the surface $X$ is the blow-up of $\PP^1
  \times \PP^1$ at $9$ points and a birational model of the
  triple plane $f \colon X \to \PP^2$ is obtained by
  using the curves of bidegree $(2, \, 3)$ passing through these points, since \eqref{eq:type-III} becomes $H= 2L_1+3L_2 - \sum_{i=1}^9 E_i$.

When $n=1$, since $\FF_1$ is the blow-up of the plane at one point, we see from \eqref{eq:type-III} that $X$ can be also seen as the blow-up of $\PP^2$ at $10$ points and that $H= 4L - 2 E_{10} - \sum_{i=1}^9 E_i $.

Another description of triple planes of type \ref{III} is the
following. Let $X'$ be a Castelnuovo surface with $K_{X'}^2=2$,
embedded in $\Gr(1, \, \PP^3)$ as a surface of 
  bidegree $(3, \,3)$, see \cite[Theorem 4.1 $(e)$]{G93}. There is
  a family of $3$-secant planes to $X'$; projecting from one of these
  planes, we obtain a birational model of a triple plane $f\colon X
  \to \PP^2$ of type \ref{III} (in fact, $X$ is the blow-up of $X'$ at
  three points). 
\end{remark}

\subsection{Triple planes of type \ref{IV}} 
In this case the invariants are
\begin{equation*}
  K_X^2=-4, \quad b=5, \quad h=7, \quad g(H)=3.
\end{equation*}
By Theorem
\ref{teo:res}, the resolution of
  $\mF=\mE(3)$ is
\begin{equation*}
  0 \to \mO_{\PP^2} (-1) \to \mO_{\PP^2}^3 \to \mF \to 0,
\end{equation*}
hence $\mF \simeq T_{\PP^2}(-1)$ and  \eqref{eq:S3F(6-b)} implies that
$S^3 \mE^{\vee} \otimes \wedge^2 \mE$ is isomorphic to $S^3
(T_{\PP^2}(-1)) \otimes \mO_{\PP^2}(1)$, which is 
globally generated.
By Theorem \ref{thm triple:1} this ensures the existence of  triple
planes of type \ref{IV}, whereas Proposition \ref{prop.TypeIV} below
provides their complete classification. 

\begin{proposition} \label{prop.TypeIV}
  Let $f \colon X \to \PP^2$ be a triple plane of type
  \emph{\ref{IV}}. Then$:$
  \begin{enumerate}[itemsep=2pt, label=\bf{\roman*)}]
  \item \label{IV-i} the surface $X$ is isomorphic to the blow-up of
    the plane at
  a subset $Z$ of $13$ points imposing only $12$ conditions on
  quartic curves, and $|H|$ is the complete linear system of quartics
  passing through $Z;$
  \item \label{IV-ii} $Z$ can be naturally identified with a $0$-dimensional subscheme of $\PD^2$, that we call again $Z$, arising as the zero locus of a global section of $T_{\PD^2}(2)$ canonically associated with the building section $\eta \in H^0(\PP^2,S^3 \mE^{\vee} \otimes \wedge^2 \mE)$ of the triple plane. Furthermore, the subscheme $Z\subset \PD^2$ determines $\eta$ up to a multiplicative constant.
  \end{enumerate}
\end{proposition}

\begin{proof}
Let us show \ref{IV-i}. By  Proposition \ref{prop.K+2H} the divisor
$D:=K_X+2H$ is very ample. Therefore, the first adjunction map
\begin{equation*}
\varphi_1:=\varphi_{|K_X+D|} \colon X \to X_1 \subset
\PP^{5}
\end{equation*}
is a birational morphism onto a smooth surface $X_1$. Moreover, the intersection matrix of $X_1$ is 
\begin{equation*}
\left( \begin{array}{cc} (D_1)^2 & K_{X_1}D_1 \\
 K_{X_1}D_1 & (K_{X_1})^2\\
 \end{array} \right) =
\left( \begin{array}{cc} \; 4 & -6 \\
 -6 &  -4 + \alpha_1\\
 \end{array} \right),
\end{equation*}
where $D_1$ and $\alpha_1$ are defined in \S \ref{subsec:adjunction theory}.
 In particular $K_{X_1}D_1 <0$ shows that $X_1$ (and so $X$) is a rational surface. We have $g(D_1)=0,$ thus by Theorem \ref{adjunction-thm}  the adjoint linear system $|K_{X_1}+D_1|$ has dimension $-1$, i.e. it is empty. By the same result, it follows that the surface $X_1$ is either a
rational normal scroll (and in this case $\alpha_1=12$) or
$\PP^2$ (and in this case $\alpha_1=13$). Let us exclude
 the former case. There are two types of smooth quartic rational normal scroll surfaces: $S(2, \, 2)$, namely $\PP^1 \times \PP^1$ embedded in
$\PP^5$ by $|L_1+ 2 L_2|$, and $S(1, \, 3)$, namely $\FF_2$
embedded in $\PP^5$ by $|\frc + 3 \frf|$. The equality 
$D_1 = 2K_X + 2H$ implies that if $X_1 = \PP^1 \times \PP^1$ we have
\begin{equation*}
2H= 5 L_1 + 6 L_2 - \sum_{i=1}^{12} 2E_i,
\end{equation*}
whereas if $X_1 = \FF_2$ we have
\begin{equation*}
2H= 5\frc + 11 \frf - \sum_{i=1}^{12} 2E_i.
\end{equation*}
In both cases we obtain
a contradiction, since $H$ must be a divisor with integer
coefficients.

It follows that  $(X_1, \,D_1)=(\PP^2, \, \mO_{\PP^2}(2))$,
hence $\alpha_1=13$ and $\varphi_1$ contracts exactly $13$
exceptional lines, i.e. $X$ is isomorphic to the blow-up of $\PP^2$
at $13$ points. Therefore we get
\begin{equation*}
X=\PP^2(p_1, \ldots, p_{13}), \quad
D=5L-\sum_{i=1}^{13}E_i,
\end{equation*}
which implies $H=4L-\sum_{i=1}^{13}E_i$.
Since $h^0(X, \, \mO_X(H))=3$,  the points in the set
$Z:=\{p_1, \ldots, p_{13}\}$ impose only $12$ conditions on plane quartic
curves.

\bigskip

We now prove \ref{IV-ii}.
We use the notation of \S \ref{a}, so that the vector bundle $\mF \simeq
T_{\PP^2}(-1)$ has a resolution of the form \eqref{basic}, with
the $3$-dimensional vector space $U=H^0(\PP^2, \, \mF)$ being naturally
identified with $V^\vee$. By the results in  \S
\ref{sub:proj.bundle},  in this case $\PP(\mF)$ is the point-line incidence correspondence in
$\PP^2 \times \PD^2$, namely a smooth hyperplane section of $\PP^2
\times \PD^2$, so we have
\begin{equation} \label{eq:Kos-IV}
0 \to \mO_{\PP^2 \times \PD^2}(-1, \, -1) \to \mO_{\PP^2 \times \PD^2} \to \mO_{\sP} \to 0.
\end{equation}
Twisting \eqref{eq:Kos-IV} by $\mathfrak{p}^* (\mO_{\PP^2}(1))=\mO_{\PP^2 \times \PD^2}(1, \, 0)$, applying the functor $\mathfrak{q}_*$ and using \eqref{eq:Riq} we obtain 
\begin{equation*}
0 \to \mO_{\PP^2}(-1) \to H^0(\PP^2, \, \mO_{\PP^2}(1)) \otimes \mO_{\PD^2} \to \mathfrak{q}_* \big(\mathfrak{p}^*(\mO_{\PP^2}(1)) \otimes \mO_{\PP(\mF)}) \big) \to 0,
\end{equation*}
so the Euler sequence yields
\begin{equation*}
\mR = \mathfrak{q}_*(\mO_{\sP}(\ell))=\mathfrak{q}_* \big(\mathfrak{p}^* (\mO_{\PP^2}(1)) \otimes \mO_{\sP} \big) \simeq
T_{\PD^2}(-1)
\end{equation*}
and equality \eqref{eq:S3E} implies
\begin{equation*}
H^0(\PP^2, \, S^3 \mE^{\vee} \otimes \wedge ^2 \mE ) =H^0(\PD^2, \, \mR(3))= H^0(\PD^2, \, T_{\PD^2}(2)).
\end{equation*}
This shows that the building section $\eta$ of our triple plane is naturally associated with a global section of $T_{\PD^2}(2)$ that we call $\eta$, too, and whose vanishing locus will be denoted by $Z=D_0(\eta)$. Note that $Z$ is a zero-dimensional subscheme of $\PD^2$ such $\mathrm{length}(Z)=c_2(T_{\PD^2}(2))=13$. 

Furthermore we have $\mR(3)=\mathfrak{q}_* \mL$, where $\mL=\mO_{\sP}(3\xi+\ell)$, and our triple plane $X$ is a smooth divisor in the complete linear system $|\mL|$, see \eqref{eq:linear-L}. Since a global section of $\mL$ corresponds to a non-zero morphism $\mO_{\sP} \to \mL$, we obtain a short exact sequence
\begin{equation} \label{eq:IV-1}
0 \to \mO_{\sP}(-3\xi) \to \mL(-3\xi) \to \mO_X(H) \to 0,
\end{equation}
and so, taking the direct image via $\mathfrak{q}$, we get
\begin{equation} \label{eq:IV-2}
0 \to \mO_{\PD^2}(-3) \to T_{\PD^2}(-1) \to \mI_{Z/\PD^2}(4) \to 0.
\end{equation}
The inclusion $X \simeq \PP(\mO_X(H)) \hookrightarrow \sP$
 corresponds to the surjection $\mL(-3\xi) \to \mO_X(H)$ in
 \eqref{eq:IV-1}; then \eqref{eq:IV-2} shows that $X$ can be
 identified with $\PP(\mI_{Z/\PD^2}(4))$, embedded in  in
 $\PP(T_{\PD^2}(-1))$ via the surjection $T_{\PD^2}(-1) \to
 \mI_{Z/\PD^2}(4)$. 
Hence a model of the triple cover map $f \colon X \to
\PP^2$ is the rational map $\PD^2 \dashrightarrow \PP^2$ given by the
linear system of (dual) quartics through $Z$. This identifies $X$ with the blow-up of $\PD^2$ at $Z$.

Finally, let us show that the subscheme $Z$ determines
$\eta \in H^0(\PD^2, \, T_{\PD^2}(2))$ up to a multiplicative
constant. To this purpose, we use Proposition \ref{propsteiner} with
$t=0$, so we only have to check that
$h^1(\PP^2,\mI_{Z/\PP^2}(4))=1$. But this is clear since
$\chi(\PP^2,\mI_{Z/\PP^2}(4))=-2$ and $h^0(\PP^2,\mI_{Z/\PP^2}(4))=3$.
\end{proof}

\begin{remark} \label{rmk:IV}
  A \emph{Bordiga surface} is a smooth surface of degree $6$ in
  $\PP^4$, given by the blow up of $\PP^2$ at $10$
  points embedded by the linear system of plane quartics through them, see
  \cite[Capitolo 5]{Ott95}. Then Proposition \ref{prop.TypeIV} shows that a
  birational model of a triple plane $f \colon X \to \PP^2$ of type \ref{IV} can be realized as the projection of a Bordiga surface from a $3$-secant line.

Furthermore, contracting one of the exceptional divisors in the Bordiga surface, we obtain a rational surface $X'$ with $K_{X'}^2=0$ that can be embedded
  in $\Gr(1, \, \PP^3)$ as a surface of bidegree $(3,
  \,4)$, see \cite[Theorem 4.1 $(f)$]{G93}. There is a family of
  $4$-secant planes to $X'$; projecting from one of these planes, we
  obtain another birational model of a triple plane $f\colon X \to \PP^2$ of type \ref{IV} (in fact,
  $X$ is the blow-up of $X'$ at four points).
\end{remark}

\subsection{Triple planes of type \ref{V}} 

In this case the invariants are
\begin{equation*}
  K_X^2=-6, \quad b=6, \quad h=11, \quad g(H)=4,
\end{equation*}
and by Theorem \ref{teo:res} the twisted Tschirnhausen bundle
$\mF$ has a resolution of the form
\begin{equation} \label{eq:M-V}
  0 \longrightarrow \mO_{\PP^2}(-1)^2 \stackrel{M}{\longrightarrow} \mO_{\PP^2}^4 \longrightarrow
  \mF \longrightarrow 0.
\end{equation}

Since $\mF$ is globally generated, it follows that $S^3 \mE^\vee \otimes \wedge^2 \mE=S^3
\mF$ is globally generated, too.
Hence triple planes $f \colon X \to
\PP^2$ of type \ref{V} do exist by Theorem \ref{thm triple:1}. The
next result provides their  classification.

\begin{proposition} \label{prop.TypeV}
  Let $f \colon X \to \PP^2$ be a triple plane of type
  \emph{\ref{V}}. Then$:$
  \begin{enumerate}[itemsep=2pt, label=\bf{\roman*)}]
  \item \label{V-1} the surface $X$ is isomorphic to the blow-up
    $\PP^2(p_1, \ldots, p_{15})$ of  $\PP^2$ at $15$ points and the
    triple plane map is induced by the linear system of plane sextics 
    singular at $p_1,\ldots,p_6$ and passing through $p_7,\ldots,p_{15};$
  \item \label{V-2} the nine points $p_7,\ldots,p_{15}$ consists of  
    the intersection $S \cap C$, where $S=\PP^2(p_1,\ldots,p_6)$ is a cubic surface in $\PP^3$, naturally associated with the
    building section  $\eta \in
  H^0(\PP^2,S^3 \mE^{\vee} \otimes \wedge^2 \mE)$, whereas $C$ is a twisted cubic such that $\PP(\mF)$ is the blow-up of $\PP^3$ at $C$.
  \end{enumerate}
\end{proposition}
\begin{proof} Let us show \ref{V-1}.
By Proposition \ref{prop.K+2H} the divisor $D:=K_X+2H$ is very
ample. We have $K_X H=2b-9=3$, and the genus formula yields $g(D)=10$,  so by Theorem \ref{adjunction-thm} we deduce that the first adjoint system $|K_X+D|$ has dimension $9$.  Therefore the first adjunction map
\begin{equation*}
\varphi_1=\varphi_{|K_X+D|} \colon X \to X_1 \subset
\PP^{9}
\end{equation*}
is birational onto its image $X_1$, whose intersection
matrix is
\begin{equation*}
\left( \begin{array}{cc} (D_1)^2 & K_{X_1}D_1 \\
 K_{X_1}D_1 & (K_{X_1})^2\\
 \end{array} \right) =
\left( \begin{array}{cc} \; 12 & -6 \\
 -6 &  -6 + \alpha_1\\
 \end{array} \right).
\end{equation*}
In particular $K_{X_1}D_1 <0$ shows that $X_1$ (and so $X$) is a rational
surface. Now we consider the
second adjunction map $\varphi_2 \colon X_1 \to X_2 \subset
\PP^3$, which is an isomorphism onto its image $X_2$ (Proposition \ref{prop:phi_n}), whose intersection matrix is
\begin{equation*}
\left( \begin{array}{cc} (D_2)^2 & K_{X_2}D_2 \\
 K_{X_2}D_2 & (K_{X_2})^2\\
 \end{array} \right) =
\left( \begin{array}{cc} \; -6+\alpha_1 & -12 + \alpha_1 \\
 -12 + \alpha_1 &  -6 + \alpha_1\\
 \end{array} \right).
\end{equation*}
This shows that $X_2$ is a non-degenerate, smooth rational surface in
$\PP^3$, hence it is either a quadric surface or a cubic surface. If $X_2$
were a quadric then $(D_2)^2=2$, hence $\alpha_1=8$ and the
intersection matrix would give $(K_{X_2})^2=2$, which is a
contradiction. Therefore $X_2$ is a cubic surface $S$, hence
$\alpha_1=9$.
 Moreover $X_1$ is isomorphic to $X_2$, so $X$ is the blow-up of $S$ at $9$ points. 
 It follows
\begin{equation*}
X=\PP^2(p_1, \ldots, p_{15}), \quad D=9L-\sum_{i=1}^{6}3E_i -
\sum_{i=7}^{15} E_j,
\end{equation*}
which implies $H=6L-\sum_{i=1}^{6}2E_i - \sum_{i=7}^{15} E_j$. 

\bigskip 
We turn to \ref{V-2}. Here we use the approach developed in \S \ref{a},  in particular we consider again the resolution \eqref{basic},
where in this case $U=H^0(\PP^2,\mF)$ is a $4$-dimensional vector space.
Set $\PP^3 = \PP (U)$.
By Proposition \ref{prop:P-and-P'}, the projective bundle $\PP(\mF)$
is the complete intersection of two divisors of bidegree $(1, \, 1)$
in $\PP^2 \times \PP^3$, so the corresponding Koszul resolution is 
\begin{equation} \label{eq:Kos-V}
0 \to \mO_{\PP^2 \times \PP^3}(-2, \, -2) \to \mO_{\PP^2 \times \PP^3}(-1, \, -1)^2 \stackrel{d_1} \to \mO_{\PP^2 \times \PP^3} \to \mO_{\sP} \to 0.
\end{equation}
Twisting \eqref{eq:Kos-V} by $\mathfrak{p}^* (\mO_{\PP^2}(1))= \mO_{\PP^2 \times \PP^3 }(1, \, 0)$ and splitting it into short exact sequences, we get 
\begin{equation*}
\begin{split}
0 & \longrightarrow \mO_{\PP^2 \times \PP^3}(-1, \, -2) \longrightarrow \mO_{\PP^2 \times \PP^3}(0, \, -1)^2 \longrightarrow \widetilde{\mK}_1 \longrightarrow 0, \\
0  & \to \widetilde{\mK}_1 \to \mO_{\PP^2 \times \PP^3}(1, \, 0) \to \mO_{\sP}(\ell) \to 0,
\end{split}
\end{equation*}
where $\widetilde{\mK}_1:= \mK_1 \otimes \mO_{\PP^2
  \times \PP^3}(1, \, 0) $ and $\mK_1$ is the image of the
first differential $d_1$ of the Koszul complex, see \S
\ref{L}. Applying the functor $\mathfrak{q}_*$ and using \eqref{eq:Riq}, we infer 
\begin{equation*}
\mathfrak{q}_* \widetilde{\mK}_1  = \mO_{\PP^3}(-1)^2, \quad
R^1 \mathfrak{q}_* \widetilde{\mK}_1 =0,
\end{equation*}
obtaining
\begin{equation} \label{eq:N-V}
0 \longrightarrow \mO_{\PP^3}(-1)^2 \stackrel{N}{\longrightarrow} \mO_{\PP^3}^3 \longrightarrow \mathfrak{q} _*(\mO_{\sP}(\ell)) \longrightarrow 0.
\end{equation}
Hence we can identify $\mathfrak{q}_*(\mO_{\sP}(\ell))$ with
$\mI_{C/\PP^3}(2)$, the ideal sheaf of quadrics in $\PP^3$ containing
a twisted cubic $C$, which is precisely the image in $\PP^3$ of the
conic parametrizing the unstable lines of $\mF$ (Remark \ref{rem:cubic}).
Note that $C$ is given by the vanishing of the
three $2 \times 2$ minors of the matrix of linear forms $N$ appearing
in $\eqref{eq:N-V}$; this matrix coincides 
with the one obtained by ``flipping'' the matrix $M$ in \eqref{eq:M-V} as explained in \S \ref{a}, see in particular 
Remark \ref{rem:trilinear-form}.
Then 
$\mG=\mathfrak{q}_*(\mO_{\sP}(\ell))$, and by Proposition
\ref{prop:P-and-P'} we infer 
\begin{equation*}
\PP(\mF) \simeq \PP(\mG) \simeq \PP(\mI_{C/\PP^3}(2)),
\end{equation*}
that is, $\PP(\mF)$ is isomorphic to the blow-up of $\PP^3$ along the
twisted cubic $C$ and the morphism $\mathfrak{p} \colon \PP(\mF) \to
\PP^2$ is induced by the net $|\mI_{C/\PP^3}(2)|$. 

We also get $\mR \simeq \mathfrak{q}_*\mO_{\sP} \simeq \mO_{\PP^3}$, so \eqref{eq:S3E} yields
\begin{equation*}
H^0(\PP^2, S^3 \mE^\vee \otimes \wedge^2 \mE) = H^0(\PP^3, \, \mR(3)) = H^0(\PP^3, \, \mO_{\PP^3}(3)).
\end{equation*}
This means that the choice of the (proportionality class of the) building section $\eta$ in Theorem
\ref{thm triple:1} is given by the choice of a 
cubic surface $S \subset \PP^3$. Moreover, from the exact sequence
\begin{equation*}
0 \to \mO_{\sP}(-3\xi + \ell) \to \mO_{\sP}(\ell) \to \mO_{X}(H) \to 0
\end{equation*}
it follows that $X \simeq \PP(\mO_{X}(H))$ is the  strict transform of
$S$ in $\PP(\mF)$.
Also, the triple cover map $f \colon X \to \PP^2$ is associated with
$|\mO_{X}(H)|$, so that it is induced on $S$ by the linear system of quadrics 
that contain the intersection $S \cap C$. This intersection consists
of $9$ points $p_7,\ldots,p_{15}$. Identifying $S$ with $\PP^2(p_1,\ldots,p_6)$ 
with exceptional divisors $E_1,\ldots,E_6$, we
get thus $9$
exceptional divisors $E_7,\ldots,E_{15}$ on $X$ corresponding to this
intersection, and
\begin{equation*}
H=2H_S-\sum_{j=7}^{15} E_j=6L -\sum_{i=1}^6 2E_i -\sum_{j=7}^{15} E_j.
\end{equation*}
This identifies the sets $\{p_1, \ldots, p_6 \}$ and $\{p_7, \ldots, p_{15} \}$ with those in part \ref{V-1}.
\end{proof}

\begin{remark} \label{rmk:V}
  A birational model of the triple plane $f \colon X \to \PP^2$ is
  the projection of a hyperplane section $T$ of a Palatini scroll from a $4$-secant line. In fact, $T$ is a surface of degree $7$ in $\PP^4$ and with $K_T^2=-2$ (see \cite[Capitolo
  5]{Ott95}), which is isomorphic to $\PP^2$ blown-up at $11$ points and
  embedded in $\PP^4$ by the complete linear system $|6L -
  \sum_{i=1}^6 2E_i - \sum_{j=7}^{11} E_j|$. Actually, this is the unique non-degenerate,
  rational surface of degree $7$ in $\PP^4$, see \cite[Theorems 4 and 6]{Ok84}.

  Contracting one of the exceptional divisors $E_j$ in $T$, we obtain a rational surface
  $X'$ with $K_{X'}^2=-1$  that can be embedded in $\Gr(1, \,
  \PP^3)$ as a surface of bidegree $(3, \,5)$, see
  \cite[Theorem 4.1 $(g)$]{G93}. So there is a family of $5$-secant
  planes to $X'$; projecting from one of these planes, we obtain a
  birational model of a triple plane $f\colon X \to \PP^2$ of type \ref{V} (in fact, $X$ is the
  blow-up of $X'$ at five points).
\end{remark}

\begin{remark}  \label{rem:Du-Val-I-V}
  Triple planes of types \ref{I} to \ref{V} were
  previously considered via ``classical" methods by Du
  Val in \cite{DuVal33}. For the reader's convenience, let us shortly
  describe in modern language and using our notation Du Val's nice
  geometric constructions. They use part of the mass of results
  on particular rational surfaces proven by nineteenth century  algebraic geometers; the classical, a bit old-fashioned monograph on the subject (in Italian)
  is \cite{Conf45}, for a modern exposition see
  \cite{dolgachev:classical-AG}.    
  \begin{itemize}
  \item[\ref{I})] 
    We have $g(H)=0$, and from this one sees that the net
  $|H|$ is the pull-back of the net of lines $|L|$ in $\PP^2$
  via the projection  of the cubic scroll $S(1, \, 2) \subset
  \PP^4$ from a general line. The generators of the scroll
  become a $\infty^1$ family of lines of index $3$ in $\PP^2$,
  i.e. such that for a general point of the plane pass three lines of
  the family. The envelop of this family is a tricuspidal quartic
  curve, namely the branch locus $B$ of the triple plane.    
\item[\ref{II})] This time $g(H)=1$, so that the surface $X$ is either
  rational or ruled. When $p_g(X)=q(X)=0$ we are in the first case,
  and the only possibility for the triple plane is the projection of a
  smooth cubic surface $S_3 \subset \PP^3$ from an external
  point $p$. Then the ramification locus $R$ is given by the
  intersection of $S_3$ with the polar hypersurface $P_p (S_3)$, which
  is a quadric $Q$. Hence $R$ is a smooth curve of degree $6$ and
  genus $4$ in $\PP^3$, and the six cusps of the branch locus
  $B$ arise from the intersection of $R$ with the second polar of $p$,
  which is a plane $\Pi$. In particular, the cusps of $B$ are
  contained in the projections of both the curves $Q \cap \Pi$ and
  $S_3 \cap \Pi$, namely they are the complete intersection of a conic
  and a plane cubic.
\item[\ref{III})] In this case $g(H)=2$, and a surface $X$ with a net of
  genus $2$ curves is either a double plane with a branch curve of
  order $6$ (i.e., a $K3$ surface) or a rational surface. In the last
  case, a detailed analysis of the possible linear systems
  representing $X$ on $\PP^2$ shows that the only possibility
  in order to have a net $|H|$ inducing a triple plane is that $X$ is the blow-up of $\PP^2$ at $10$ points, so that the curves of $|H|$
  corresponds to quartic with one double and nine simple base
  points. We recovered by modern methods this result, see Remark
  \ref{rmk:III} (since Du Val only works with representative linear
  systems on $\PP^2$, he does not consider the birational
  models of these triple planes arising from linear systems on
  $\FF_n$). It can be observed that this construction
  corresponds to the projection to $\PP^2$ of a quartic surface
  $S_4 \subset \PP^3$, having a double line, from a general
  point $p \in S_4$. In fact, $S_4$ is represented on the plane by
  quartic curves with one double and $8$ simple base points. On the
  surface $S_4$ there is a pencil of conics, corresponding to the
  pencil of lines on $\PP^2$ through the double base point; in
  the triple plane representation, this becomes a family $\infty^1$ of
  conics of index $3$, whose envelop is a curve $B$ of degree $8$
  with $12$ cusps, which is precisely the branch locus of our triple
  plane.    
\item[\ref{IV})] In this case we have $g(H)=3$, and a detailed analysis of the
  linear systems $|H|$ and $|K_X + H|$ shows that a birational model
  of the triple plane is given from the projection of a quintic
  surface $S_5 \subset \PP^2$ having a double twisted cubic
  from a point of the double curve. From this fact one recovers the
  plane representation of the linear system $|H|$ as a net of quartics
  with thirteen simple base points, and the representation of the
  branch curve $B$ as the Jacobian curve of this net. According to
  Proposition \ref{prop.TypeIV}, the base points are not in general
  position. In fact, eleven of them, say $p_1, \ldots, p_{11}$, can be
  taken at random, whereas the remaining two must belong to the
  $g_2^1$ of the unique hyperelliptic curve of degree $7$ having nodes
  at $p_1, \ldots, p_{11}$.   
\item[\ref{V})] 
  In this case $g(H)=4$, and the assumption $p_g(X)=q(X)=0$
  shows that the adjoint linear system $|K_X+H|$ cuts on the general
  curve of the net $|H|$ the complete canonical system $|K_X|$. Then
  the image of $|H|$ via the first adjoint map $\varphi_{|K_X+H|}
  \colon X \to \PP^3$ is a net of canonical curves of genus $4$
  and degree $6$. So there is precisely one quadric surface containing
  each of these curves, and one system of generators of each of these
  quadrics traces a system of $\infty^2$ trisecant lines to the image
  of $X$, that together define a degree $3$ ``involution" (Du Val,
  like his contemporaries, use this term also when dealing with finite
  covers of degree $>2$)  which gives a birational model of our triple
  plane. Pushing this analysis further, it is possible to show that
  such a system of trisecant lines is actually the system of chords of
  a twisted cubic $C$, and this implies that the net $|H|$ can be
  represented on a cubic surface $S \subset \PP^3$ by means
  of sections by quadrics passing through $C$. Correspondingly, $X$ is
  a rational surface that can be represented on the plane by sextic
  curves with six double and nine simple base points, the latter
  corresponding to the intersections of $S$ with $C$. Part \ref{V-2} of Proposition \ref{prop.TypeV} is a modern rephrasing of
  this argument that uses completely different techniques based on
  vector bundles. Finally, by using envelops one computes that the
  branch locus $B$ of the triple plane has degree $12$; its cusps
  arise from the chords of $C$ that are also inflectional tangents of
  $S$, and a Schubert calculus computation shows that their number
  equals $33$.   
  \end{itemize}
\end{remark}

\subsection{Triple planes of type \ref{VI}} \label{subsec:VI}

In this case the invariants are
\begin{equation*}
  K^2=-7, \quad b=7, \quad h=16, \quad g(H)=5
\end{equation*}
and by Theorem \ref{teo:res} the twisted Tschirnhausen bundle $\mF$ has a resolution of the form
\begin{equation} \label{eq:M-VI}
  0 \longrightarrow \mO_{\PP^2}(-1)^3 \stackrel{M}{\longrightarrow} \mO_{\PP^2}^5 \longrightarrow
  \mF \longrightarrow 0.
\end{equation}
The existence and classification of triple planes of type \ref{VI}
 are established in Proposition \ref{prop.TypeVI} below.

\begin{proposition} \label{prop.TypeVI}
  Let $f \colon X \to \PP^2$ be a triple plane of type
  \emph{\ref{VI}}. Then the following holds$:$
  \begin{enumerate}[itemsep=2pt, label=\bf{\roman*)}]
  \item \label{loga} 
    the vector bundle $\mF$ is a logarithmic bundle associated with $6$
    lines in general position in $\PP^2;$
  \item \label{detcubic} the morphism $\mathfrak{q} \colon \PP(\mF) \to \PP^4$ is
  birational onto its image, which is a determinantal cubic threefold
  $Y \subset \PP^4$, which has exactly six nodes as singularities$;$
\item \label{blowbordiga} the surface $X$ is the blow-up of a Bordiga
  surface $X_1 \subset Y$ at the six nodes of $Y$, that belong to
  $X_1$.
  So $X$ is the blow-up $\PP^2$ at $16$ points and the net $|H|$ defining the
  triple cover $f$ is given by 
  \begin{equation} \label{eq:H-VI}
    H=7L - \sum_{i=1}^{10} 2E_i - \sum_{j=11}^{16} E_j.
  \end{equation}
  \end{enumerate}
\end{proposition}
\begin{proof}[Proof of \ref{detcubic}]
  We use again the approach and notation of \S \ref{subsec:steiner}. We look at the exact
  sequence \eqref{basic} and we consider the projective space $\PP^4 = \PP(U)$, that coincides with the space of global sections of the Steiner bundle $\mF$.
  By \eqref{basic-flipped}, the $5 \times 3$ matrix $M$ of linear forms presenting $\mF$ is naturally associated with a  $3 \times 3$ matrix $N$, generically of maximal rank, defining a Steiner sheaf $\mG$ over $\PP^4$, namely
  \begin{equation} \label{eq:twisted-VI}
    0 \longrightarrow \mO_{\PP^4}(-1)^3 \stackrel{N}{\longrightarrow} \mO_{\PP^4}^3 \longrightarrow
    \mG \longrightarrow 0.
  \end{equation}

 Now recall that the morphism $\mathfrak{q}$ is birational onto its image by Lemma
  \ref{contracts}, and that $\PP(\mG) \simeq
  \PP(\mF)$ by Proposition \ref{prop:P-and-P'}, so that $\mathfrak{q}$ maps $\PP(\mF)$ to the support of $\mG$, which is the determinantal hypersurface $Y \subset \PP^4$ defined by $\det(N)=0$. 
Note that Porteous formula says that the threefold $Y$ is singular, expectedly at six points, see \cite[Chapter II]{ACGH}.
  \end{proof}

 \begin{claim} \label{oneclaim}
    The surface $X_1 \subset \PP^4$, image of the first adjunction
    map $\varphi_{|K_X + H|} \colon X \to \mathbb{P}^4$, is a Bordiga surface of degree $6$. It is defined by the vanishing of the maximal minors of a $3 \times
    4$ matrix obtained by stacking a row to the transpose of $N$.
  \end{claim}
  \begin{proof}
  By the results of \S \ref{L}, the surface $X$ corresponds to a global section
  \begin{equation*}
    \eta \in H^0(\sP, \, \mO_{\sP}(3\xi-\ell)) \simeq H^0(\PP^4,\, \mR(3)),
  \end{equation*}
  where $\mR= \mathfrak{q}_*(\mO_{\sP}(-\ell))$. 
  The idea is to directly relate $\mR$ to the sheaf $\mG$ appearing in \eqref{eq:twisted-VI} or, equivalently, to the matrix $N$.

  By Proposition \ref{prop:P-and-P'} the projective bundle
  $\sP $ is a $3$-fold linear section of $\PP^2 \times \PP^4$, i.e.
  the complete intersection of three divisors of bidegree $(1, \, 1)$ in
  $\PP^2 \times \PP^4$. Tensoring the Koszul resolution \eqref{KOSZUL} of $\mO_{\sP}$
  inside $\mO_{\PP^2 \times \PP^4}$ with $\mathfrak{p}^*(\mO_{\PP^2}(-1))=\mO_{\PP^2 \times \PP^4}(-1, \, 0)$ and splitting it into short exact sequences, we obtain
  \begin{gather}
    0  \to \mO_{\PP^2 \times \PP^4} (-4, \, -3) \to \mO_{\PP^2 \times
      \PP^4} (-3, \, -2)^3 \to \widetilde{\mK}_2 \to 0, \label{eq.VI.1} \\
    0  \to \widetilde{\mK}_2  \to \mO_{\PP^2 \times \PP^4} (-2, \, -1)^3
    \to \widetilde{\mK}_1 \to 0, \label{eq.VI.2} \\
    0  \to \widetilde{\mK}_1 \to \mO_{\PP^2 \times \PP^4} (-1, \, 0) \to \mO_{\PP(\mF)}(-\ell) \to 0 \label{eq.VI.3},
  \end{gather}
  where $\widetilde{\mK}_i:=\mK_i \otimes \mO_{\PP^2 \times \PP^4}(-1, \, 0)$ and $\mK_i$ denotes the image of the $i$-th differential of the Koszul complex, see \S \ref{sub:proj.bundle}. Applying the functor $\mathfrak{q}_*$ to \eqref{eq.VI.1} and using \eqref{eq:Riq},
  we deduce $\mathfrak{q}_*
  \widetilde{\mK}_2=0$ and we get
\begin{equation} \label{eq.VI.4}
0 \to R^1 \mathfrak{q}_* \widetilde{\mK}_2 \to \mO_{\PP^4}(-3)^3 \to \mO_{\PP^4}(-2)^3 \to R^2 \mathfrak{q}_* \widetilde{\mK}_2 \to 0.
\end{equation}
By \eqref{eq.VI.3} the sheaf $\widetilde{\mK}_1$ injects into
$\mO_{\PP^2 \times \PP^4}(-1, \, 0)$, so  we have $\mathfrak{q}_*
\widetilde{\mK}_1=0$. Therefore, applying $\mathfrak{q}_*$ to \eqref{eq.VI.2}, we get
\begin{equation} \label{eq.VI.A}
R^1 \mathfrak{q}_* \widetilde{\mK}_2=0, \quad R^1 \mathfrak{q}_* \widetilde{\mK}_1\simeq R^2 \mathfrak{q}_*
\widetilde{\mK}_2.
\end{equation}
Finally, applying the functor $\mathfrak{q}_*$ to \eqref{eq.VI.3} we infer
\begin{equation} \label{eq.VI.B}
\mR=\mathfrak{q}_*(\mO_{\PP(\mF)}(-\ell)) \simeq R^1 \mathfrak{q}_* \widetilde{\mK}_1.
\end{equation}
Using \eqref{eq.VI.A} and \eqref{eq.VI.B}, the exact sequence
\eqref{eq.VI.4} becomes
\begin{equation*}
0 \to \mO_{\PP^4}(-3)^3 \to \mO_{\PP^4}(-2)^3 \to \mR\to 0,
\end{equation*}
that can be rewritten as
\begin{equation} \label{eq.VI.5}
0 \to \mO_{\PP^4}^3 \xrightarrow{{}^t N} \mO_{\PP^4}(1)^3 \to \mR(3) \to 0.
\end{equation}
Indeed, the self-duality of the Koszul complex implies 
\begin{equation*}
\mR \simeq \mG^\vee \simeq \mathscr{E}xt^1_{\mO_{\PP^4}}(\mG(3), \, \mO_{\PP^4}),
\end{equation*}
where the second isomorphisms is Grothendieck duality, see
\cite[Chapter III, Proposition 7.2]{hartshorne:residues-duality}. 

Let us consider now a non-zero global section $\eta \colon
\mO_{\PP^4} \to \mR(3)$ of $\mR(3)$, whose cokernel we denote by
$\mH$. The section $\eta$ lifts to a map $\mO_{\PP^4} \to
\mO_{\PP^4}(1)^3$, so by \eqref{eq.VI.5} we get an  exact sequence
\begin{equation} \label{eq:bordiga}
  0 \longrightarrow \mO_{\PP^4}(-3) \longrightarrow \mO_{\PP^4}^4
  \stackrel{{}^t (N,\eta)}{\longrightarrow} \mO_{\PP^4}(1)^3 \to \mH
  \longrightarrow 0. 
\end{equation}
The sheaf $\mH$ is supported on the surface $X_1 \subset 
\PP^4$. More precisely, this surface is defined by the vanishing
of the $3 \times 3$ minors of the $3 \times 4$ matrix ${}^t(N,\eta)$ of linear
forms appearing in \eqref{eq:bordiga}, hence it is a Bordiga
surface of degree $6$, see \cite[Capitolo 5]{Ott95}. 
\end{proof}

By the results of \S \ref{b} it follows that the bundle $\mF$ has either six or infinitely many
unstable lines. Let us give the proof of \ref{blowbordiga} in the former case.
\begin{proof}[Proof of \ref{blowbordiga}]
  We assume that $\mF$ has six unstable lines. 
  Using Claim \ref{oneclaim} and Remark \ref{rmk:IV}, we can see $X_1$ as the blow-up
  of $\PP^2$ at $10$ points, with exceptional divisors
  $E_1,\ldots,E_{10}$, embedded in $\PP^4$ by the linear system $|4L
  -\sum_{i=1}^{10} E_i|$. On the other hand, by Proposition
  \ref{prop:adjunction} the first adjoint map
  $\varphi:=\varphi_{|K_X+H|} \colon X \to X_1$ is a birational
  morphism, contracting precisely the six exceptional divisors
  $E_{11}, \ldots, E_{16}$ on $X$ coming from the blow-up of $X_1$
  at the six nodes of $Y$.  Hence we obtain
  \begin{equation*}
\begin{split}
& K_X  = \varphi^* K_{X_1}+ \sum_{j=11}^{16} E_j = \varphi^* \bigg(-3L + \sum _{i=1}^{10} E_i \bigg)+ \sum_{j=11}^{16} E_j \quad \textrm{and} \\
& K_X + H  = \varphi^* \mO_{X_1}(1) = \varphi^* \bigg(4L - \sum _{i=1}^{10} E_i \bigg),
\end{split}
\end{equation*}
so \eqref{eq:H-VI} follows.
\end{proof}
If, instead, $\mF$ has infinitely many unstable lines then it is of Schwarzenberger type. The next result shows that this case cannot occur, proving \ref{loga} 
and so completing the proof of Proposition
\ref{prop.TypeVI}.

\begin{claim} \label{twoclaim}
  If $\mF$ is a Schwarzenberger bundle, then the vanishing locus of any non-zero global section $\eta \in
  H^0(\sP, \, \mO_{\sP}(3\xi-\ell))$ is a reducible surface. In particular, if $f \colon X \to \mathbb{P}^2$ is a triple plane of type $\ref{VI}$, then its Tschirnhausen bundle is a logarithmic one.
\end{claim}
\begin{proof}
If $\mF$ is a Schwarzenberger bundle then, up to
    a change of coordinates, the matrix $M$ defining it is given by \eqref{eq:Schw-7} and so, using Remark
    \ref{rem:trilinear-form}, one easily finds that the matrix $N$ is 
    \[
      N=\left(
        \begin{array}{ccc}
          z_0 & z_1 & z_2 \\
          z_1 & z_2 & z_3 \\
          z_2 & z_3 & z_4 \\
        \end{array}
      \right).
      \]
The singular locus of $Y$ is the determinantal variety given by the
vanishing of the $2 \times 2$ minors of $N$, and this is
a rational normal curve of degree four $C_4 \subset \PP^4$. This
curve is also the base locus of the net $|T_L|$ generated the three determinantal surfaces
$T_i$ defined by the $2\times 2$ minors of the matrix $N_i$ obtained
from $N$ by removing the $i$-th line.
By \cite[Proposition 1.2]{Val00a} we have
\begin{equation} \label{eq:S2-6}
h^0(\PP^2, \, S^2 \mF (-2))=1.
\end{equation}
This global section gives a relative quadric $Q$ in
$|\mO_{\sP}(2\xi-2\ell)|$ over $\PP(\mF)$. 
The morphism $\mathfrak{q} \colon
\PP(\mF) \to Y$ is the blow-up along $C_4$, and $Q$ is its exceptional
divisor.

The divisor $Q \in |\mO_{\sP}(2\xi-2\ell)|$ gives a sheaf map
$\mO_{\sP}(\xi+\ell) \to \mO_{\sP}(3\xi-\ell)$, which is injective on
global sections. 
Since $h^0(\PP^2, \, S^2 \mF (-2))=1$, this gives an inclusion
\begin{equation} \label{eq:inclusion-6}
H^0(\PP^2, \, S^2 \mF (-2)) \otimes H^0(\PP^2, \,
\mF(1)) \subseteq H^0(\PP^2, \, S^3 \mF (-1)).
\end{equation}
On the other hand,  we can compute
\begin{equation} \label{eq:S3-6}
h^0(\PP^2, \, \mF (1))=12, \quad h^0(\PP^2, \, S^3 \mF (-1))=12.
\end{equation}
Indeed, the first equality in \eqref{eq:S3-6} is just obtained twisting \eqref{eq:M-VI} by
$\mO_{\PP^2}(1)$ and taking global sections. For the second equality,
we tensor the third symmetric power of the exact sequence
\eqref{eq:M-VI} with $\mO_{\PP^2}(-1)$, obtaining 
\begin{equation*}
0 \to \mO_{\PP^2}(-4) \to \mO_{\PP^2}(-3)^{15} \to  \mO_{\PP^2}(-2)^{45} \xr{r_1} \mO_{\PP^2}(-1)^{35}
\xr{r_0} S^3 \mF(-1) \to 0.
\end{equation*}
Taking cohomology, we get
\begin{equation*}
H^i(\PP^2, \, S^3\mF(-1))  \simeq H^{i+1}(\PP^2, \,\ker r_0)  \simeq H^{i+2}(\PP^2, \, \ker r_1)
\end{equation*}
for all $i$, which implies $H^i(\PP^2,S^3\mF(-1))=0$ for $i > 0$. Then
\begin{equation*}
h^0(\PP^2, \, S^3\mF(-1))=\chi(\PP^2, \, S^3\mF(-1))=12.
\end{equation*}
By \eqref{eq:S2-6} and \eqref{eq:S3-6} it follows that the inclusion
in \eqref{eq:inclusion-6} is actually an equality. Geometrically,
this means that any non-zero global section of $S^3 \mF (-1)$ vanishes
along the relative quadric $Q$, that is its vanishing locus is the
union of this relative quadric and a relative plane. This proves Claim \ref{twoclaim}.
\end{proof}

\begin{remark} \label{rmk:VI}
Another way to describe triple planes of type \ref{VI} is the following.
Let $X'$ be the blow-up of $\PP^2$ at $10$ points, embedded in $\Gr(1, \, \PP^3)$ as a
surface of bidegree $(3, \,6)$ via the complete linear system $|7L -
\sum_{i=1}^{10} 2E_i|$, see \cite[Theorem 4.2 $(i)$]{G93}.
There is a family of $6$-secant planes to $X'$; projecting from
one of these planes, we obtain a birational model of a the triple
plane $f\colon X \to \PP^2$ of type $\mathrm{VI}$ (in fact, $X$ is the blow-up of $X'$ at six points).
\end{remark}

\begin{remark}\label{rem:Du-Val-VI}
Triple planes of types \ref{VI} were previously considered (using
methods of synthetic projective geometry) by Du Val in
\cite[page 72]{DuVal35}. Let us give a short description of his construction. 

We have $g(H)=5$, and the assumption $p_g(X)=q(X)=0$ shows that the
adjoint linear system $|K_X+H|$ cuts on the general curve of the net
$|H|$ the complete canonical system $|K_X|$. Then the image of $|H|$
via the first adjoint map $\varphi_{|K_X+H|} \colon X \to \PP^4$ is a
net of canonical curves of genus $5$ and degree $10$. There is a
$\infty^2$ system of trisecant lines to these curves,  that together
give a degree $3$ ``involution" on the image of $X$. Such  trisecant
lines generate a threefold $Y \subset \PP^3$, that Du Val recognizes
as a determinantal cubic threefold. At this point, the triple cover is
constructed by blowing up a Bordiga surface $X_1 \subset Y$ at the six
nodes of $Y$, that belong to $X_1$. Part \ref{blowbordiga} of Proposition
\ref{prop.TypeVI} is a modern rephrasing of this argument that uses
completely different techniques based on vector bundles. By using his
remarkable knowledge of ``classical" algebraic geometry, at the end of
his analysis Du Val is also able to identify $X$ as a congruence of
type $(3, \, 6)$ inside $\Gr(1, \, \PP^3)$, see Remark \ref{rmk:VI}. 
\end{remark}

\subsection{Triple planes of type \ref{VII}} \label{subsec:VII}

In this case we have
\begin{equation*}
K_X^2=-7, \quad b=8, \quad h=22, \quad g(H)=6
\end{equation*}
and by Theorem \ref{teo:res} the twisted Tschirnhausen bundle $\mF$ has a resolution of the form
\begin{equation} \label{eq:M-VII}
  0 \longrightarrow \mO_{\PP^2}(-1)^4 \stackrel{M}{\longrightarrow} \mO_{\PP^2}^6 \longrightarrow
  \mF \longrightarrow 0.
\end{equation}
By Remark \ref{rem:H}, we can start the adjunction process on $X$ by
using the first adjoint divisor $K_X + H$. According to \S
\ref{subsec:adjunction theory}, we denote by $\alpha_n$ the number of
exceptional curves contracted by the $n$-th adjunction map $\varphi_n
\colon X_{n-1} \to X_n$. Recall that  $\alpha_1$, the number of lines
contracted by the first adjunction map, is precisely the
number of unstable lines of the twisted Tschirnhausen bundle $\mF$,
see Proposition \ref{prop:adjunction}.

\subsubsection{The occurrences for triple planes of type VII} 

\begin{proposition} \label{prop.typeVII}
If $f \colon X \to \PP^2$ is a triple plane of type \emph{\ref{VII}}, then $X$ belongs to the following list. The cases marked with $(*)$ do actually exist.
\begin{itemize}
\item[$\mathbf{(VII.1a)}$] $\alpha_1=1, \, \alpha_2=14:$ $X$ is the blow-up at $15$ points of a Hirzebruch surface $\FF_n$, with $n \in \{0, \, 2 \}$, and
\begin{equation*}
H=5\frc+ \bigg(\frac{5}{2} n + 6 \bigg) \frf - \sum_{i=1}^{14} 2E_i - E_{15};
\end{equation*}
\item[$\mathbf{(VII.1b) (*)}$] $\alpha_1=1, \, \alpha_2=15:$ $X$ is the blow-up of $\PP^2$ at $16$ points and
\begin{equation*}
H=8L - \sum_{i=1}^{15}2E_i-E_{16};
\end{equation*}
\item[$\mathbf{(VII.2) (*)}$] $\alpha_1=2:$ $X$ is the blow-up of $\PP^2$ at $16$ points and
\begin{equation*}
H=9L -\sum_{i=1}^4 3E_i - \sum_{j=5}^{14} 2E_j - \sum_{k=15}^{16} E_k;
\end{equation*}
\item[$\mathbf{(VII.3)(*)}$] $\alpha_1=3:$ $X$ is the blow-up of $\PP^2$ at $16$ points and
\begin{equation*}
H=10L - 4E_1 -\sum_{i=2}^7 3E_i - \sum_{j=8}^{13} 2E_j - \sum_{k=14}^{16} E_k;
\end{equation*}
\item[$\mathbf{(VII.4a)}$] $\alpha_1=4, \,
\alpha_2=2:$ $X$ is the blow-up of $\FF_n$ $($with $n \in \{0, \, 1, \, 2, \, 3 \})$ at $15$ points and
\begin{equation*}
H= 6\frc+(3n+8)\frf - \sum_{i=1}^9 3E_i - \sum_{j=10}^{11}2E_j - \sum_{k=12}^{15} E_k;
\end{equation*}
\item[$\mathbf{(VII.4b) (*)}$] $\alpha_1=4, \,
\alpha_2=3:$ $X$ is the blow-up of $\PP^2$ at $16$ points and
\begin{equation*}
H=10L -\sum_{i=1}^9 3E_i - \sum_{j=10}^{12} 2E_j - \sum_{k=13}^{16} E_k;
\end{equation*}
\item[$\mathbf{(VII.4c)}$] $\alpha_1=4, \,
\alpha_2=4:$ $X$ is the blow-up of $\PP^2$ at $16$ points and
\begin{equation*}
H=12L -\sum_{i=1}^7 4E_i - 3E_8- \sum_{j=9}^{12} 2E_j-\sum_{k=13}^{16} E_k;
\end{equation*}
\item[$\mathbf{(VII.5a)}$] $\alpha_1=5, \, \alpha_2=0:$ $X$ is the blow-up of
$\PP^1 \times \PP^1$ at $15$ points, and
\begin{equation*}
H=7L_1+7L_2 - \sum_{i=1}^{10} 3E_i - \sum_{j=11}^{15} E_j;
\end{equation*}
\item[$\mathbf{(VII.5b)(*)}$] $\alpha_1=5, \,
\alpha_2=1:$ $X$ is the blow-up of $\PP^2$ at $16$
points and
\begin{equation*}
H=12L -\sum_{i=1}^6 4E_i - \sum_{j=7}^{10} 3E_j - 2E_{11} -\sum_{k=12}^{16} E_k;
\end{equation*}
\item[$\mathbf{(VII.6)(*)}$] $\alpha_1=6:$
$X$ is the blow-up of $\PP^2$ at $16$ points and
\begin{equation*}
H=13L -\sum_{i=1}^{10} 4E_i -  \sum_{j=11}^{16} E_j;
\end{equation*}
\item[$\mathbf{(VII.7)(*)}$] $\alpha_1=7:$ $X$ is the blow-up of an Enriques surface at $7$ points.
\end{itemize}
\end{proposition}
\begin{proof}
We have a birational morphism 
\begin{equation*}
\varphi_{|K_X+H|} \colon X \to X_1 \subset \PP^5
\end{equation*}
and an intersection matrix
\begin{equation*}
\left( \begin{array}{cc} (D_1)^2 & K_{X_1}D_1 \\
 K_{X_1}D_1 & (K_{X_1})^2\\
 \end{array} \right) =
\left( \begin{array}{cc} \; 10 & 0 \\
 0 &  -7+ \alpha_1 \\
 \end{array} \right).
\end{equation*}
By Hodge Index Theorem we infer $0 \leq \alpha_1 \leq 7$. Let us consider separately the different cases.\\

$\bullet \, \mathbf{\alpha_1=0.}$ The second adjunction map gives a pair $(X_2, \, D_2)$, such that the intersection matrix on the surface $X_2 \subset \PP^5$ is 
\begin{equation*}
\left( \begin{array}{cc}
(D_2)^2 & K_{X_2}D_2 \\
 K_{X_2}D_2 & (K_{X_2})^2\\
 \end{array} \right) =
\left( \begin{array}{cc} \; 3 & -7 \\
 -7 &  -7 + \alpha_2\\
 \end{array} \right).
\end{equation*}
This gives a contradiction, since a smooth surface of degree $3$ in $\PP^5$ is necessarily contained in a hyperplane. Hence the case $\alpha_1=0$ cannot occur. 

\bigskip \bigskip

$\bullet \, \mathbf{\alpha_1=1.}$  The second adjunction map gives a pair $(X_2, \, D_2)$, such that  the intersection    matrix on the surface $X_2 \subset \PP^5$ is 
\begin{equation*}
\left( \begin{array}{cc}
(D_2)^2 & K_{X_2}D_2 \\
 K_{X_2}D_2 & (K_{X_2})^2\\
 \end{array} \right) =
\left( \begin{array}{cc} \; 4 & -6 \\
 -6 &  -6 + \alpha_2\\
 \end{array} \right).
\end{equation*}
A smooth, linearly normal surface of degree $4$ in $\PP^5$ is either a rational scroll or the Veronese surface. In the former case we have $(K_{X_2})^2=8$, hence $\alpha_2=14$ and, using the classification of rational scrolls in $\PP^5$ (see the proof of Proposition \ref{prop.TypeIV}),
we get $(\textbf{VII.1a})$. In the latter case we have $(K_{X_2})^2=9,$ hence $\alpha_2=15$. This gives ($\textbf{VII.1b}$).   

\bigskip \bigskip

$\bullet \, \mathbf{\alpha_1=2.}$ The second adjunction map gives a pair $(X_2, \, D_2)$, such that the intersection matrix on the surface $X_2 \subset \PP^5$ is 
\begin{equation*}
\left( \begin{array}{cc}
(D_2)^2 & K_{X_2}D_2 \\
 K_{X_2}D_2 & (K_{X_2})^2\\
 \end{array} \right) =
\left( \begin{array}{cc} \; 5 & -5 \\
 -5 &  -5 + \alpha_2\\
 \end{array} \right).
\end{equation*}
In particular $X_2$ has degree $5$, hence it must be a Del Pezzo surface. So $(K_{X_2})^2=5$, that is $\alpha_2=10$. This gives ($\textbf{VII.2}$). \bigskip \bigskip

$\bullet \, \mathbf{\alpha_1=3.}$ The second adjunction map gives a pair $(X_2, \, D_2)$, such that the intersection matrix on the surface $X_2 \subset \PP^5$ is 
\begin{equation*}
\left( \begin{array}{cc}
(D_2)^2 & K_{X_2}D_2 \\
 K_{X_2}D_2 & (K_{X_2})^2\\
 \end{array} \right) =
\left( \begin{array}{cc} \; 6 & -4 \\
 -4 &  -4 + \alpha_2\\
 \end{array} \right).
\end{equation*}
The Hodge Index Theorem implies $\alpha_2 \leq 6$. On the other hand, Theorem \ref{adjunction-thm} implies $(K_{X_2}+D_2)^2 \geq 0$, hence $\alpha_2 \geq 6$. It follows $\alpha_2=6$, hence $(K_{X_2}+D_2)^2=0$. So $X_2$ is a conic bundle of degree $6$ and sectional genus $2$ in $\PP^5$, containing precisely $6$ reducible fibres because $(K_{X_2})^2=2$. It turns out that $X_2$ is the blow-up of $\PP^2$ at $7$ points, embedded in $\PP^5$ via the linear system 
\begin{equation*}
D_2=4L - 2E_1 - \sum_{i=2}^7 E_i,
\end{equation*}
see \cite{Io81}. This is case ($\mathbf{VII.3}$). 

\bigskip \bigskip

$\bullet \, \mathbf{\alpha_1=4.}$ The second adjunction map gives a pair $(X_2, \, D_2)$, such that the intersection    matrix on the surface $X_2 \subset \PP^5$ is 
\begin{equation*}
\left( \begin{array}{cc}
(D_2)^2 & K_{X_2}D_2 \\
 K_{X_2}D_2 & (K_{X_2})^2\\
 \end{array} \right) =
\left( \begin{array}{cc} \; 7 & -3 \\
 -3 &  -3 + \alpha_2\\
 \end{array} \right).
\end{equation*}
The Hodge Index Theorem implies $\alpha_2 \leq 4$, whereas the condition $(K_{X_2}+D_2)^2 \geq 0$ gives $\alpha_2 \geq 2$; then $2 \leq \alpha_2 \leq 4$. 

$\diamond$ If $\alpha_2=2$ then by \cite[p. 148]{Io84} it follows that $X_2$ is the blow-up at $9$ points of $\FF_n$, with $n \in \{0, \, 1, \, 2, \, 3 \})$,  and that
\begin{equation*} 
D_2 = 2\frc + (n+4) \frf - \sum_{i=1}^9 E_i.
\end{equation*}
This is case $(\textbf{VII.4a})$.

$\diamond$ If $\alpha_2=3$ then the third adjunction map gives a pair $(X_3, \, D_3)$ whose intersection matrix is 
\begin{equation*}
\left( \begin{array}{cc}
(D_3)^2 & K_{X_3}D_3 \\
 K_{X_3}D_3 & (K_{X_3})^2\\
 \end{array} \right) =
\left( \begin{array}{cc} \; 1 & -3 \\
 -3 &  \alpha_3\\
 \end{array} \right).
\end{equation*}
This implies $(X_3, \, D_3)= (\PP^2, \, \mO_{\PP^2}(1))$, so $\alpha_3 = 9$. This is case  $(\textbf{VII.4b})$.

$\diamond$ If $\alpha_2=4$ then $(X_2, \, D_2)$ is as in case $(6)$ of Theorem \ref{adjunction-thm}. This is  $(\textbf{VII.4c})$. 

\bigskip \bigskip

$\bullet \, \mathbf{\alpha_1=5.}$ The second adjunction map gives a pair $(X_2, \, D_2)$, such that the intersection   matrix on the surface $X_2 \subset \PP^5$ is 
\begin{equation*}
\left( \begin{array}{cc}
(D_2)^2 & K_{X_2}D_2 \\
 K_{X_2}D_2 & (K_{X_2})^2\\
 \end{array} \right) =
\left( \begin{array}{cc} \; 8 & -2 \\
 -2 &  -2 + \alpha_2\\
 \end{array} \right).
\end{equation*}
Then the Hodge Index Theorem implies $0 \leq \alpha_2 \leq 2$. 

$\diamond$  If $\alpha_2=0$ then the third adjunction map gives a pair $(X_3, \, D_3)$, where $X_3 \subset \PP^3$ and whose intersection matrix is 
\begin{equation*}
\left( \begin{array}{cc}
(D_3)^2 & K_{X_3}D_3 \\
 K_{X_3}D_3 & (K_{X_3})^2\\
 \end{array} \right) =
\left( \begin{array}{cc} \; 2 & -4 \\
 -4 &  -2 + \alpha_3\\
 \end{array} \right).
\end{equation*}
 Hence $(X_3, \, D_3)= (\PP^1 \times \PP^1, \, \mO_{\PP^1 \times \PP^1}(1, \, 1))$, so in particular $\alpha_3=10$. This is case $(\textbf{VII.5a})$. 

$\diamond$  If $\alpha_2=1$ then the third adjunction map gives a pair $(X_3, \, D_3)$, with $X_3 \subset \PP^3$ and whose intersection matrix is 
\begin{equation*}
\left( \begin{array}{cc}
(D_3)^2 & K_{X_3}D_3 \\
 K_{X_3}D_3 & (K_{X_3})^2\\
 \end{array} \right) =
\left( \begin{array}{cc} \; 3 & -3 \\
 -3 &  -1 + \alpha_3\\
 \end{array} \right).
\end{equation*}
Therefore $X_3= \PP^2(p_1, \ldots, p_6)$ is a smooth cubic surface, in particular $\alpha_3=4$ and $D_3= 3L - \sum_{i=1}^6 E_i$. This is case $(\textbf{VII.5b})$. 

$\diamond$  If $\alpha_2=2$ then the third adjunction map gives a pair $(X_3, \, D_3)$, with $X_3 \subset \PP^3$ and whose intersection matrix is 
\begin{equation*}
\left( \begin{array}{cc}
(D_3)^2 & K_{X_3}D_3 \\
 K_{X_3}D_3 & (K_{X_3})^2\\
 \end{array} \right) =
\left( \begin{array}{cc} \; 4 & -2 \\
 -2 &  \alpha_3\\
 \end{array} \right).
\end{equation*}
Therefore $X_3$ is a smooth quartic surface, a contradiction because we are assuming $p_g(X)=0$. This case cannot occur. 

\bigskip \bigskip

$\bullet \, \mathbf{\alpha_1=6.}$ The second adjunction map gives a pair $(X_2, \, D_2)$, such that the intersection   matrix on the surface $X_2 \subset \PP^5$ is 
\begin{equation*}
\left( \begin{array}{cc}
(D_2)^2 & K_{X_2}D_2 \\
 K_{X_2}D_2 & (K_{X_2})^2\\
 \end{array} \right) =
\left( \begin{array}{cc} \; 9 & -1 \\
 -1 &  -1 + \alpha_2\\
 \end{array} \right).
\end{equation*}
Then the Hodge Index Theorem implies $0 \leq \alpha_2 \leq 1$. 

$\diamond$  If $\alpha_2=0$ then the third adjunction map gives a pair $(X_3, \, D_3)$, with $X_3 \subset \PP^4$ and whose intersection matrix is 
\begin{equation*}
\left( \begin{array}{cc}
(D_3)^2 & K_{X_3}D_3 \\
 K_{X_3}D_3 & (K_{X_3})^2\\
 \end{array} \right) =
\left( \begin{array}{cc} \; 6 & -2 \\
 -2 &  -1+\alpha_3\\
 \end{array} \right).
\end{equation*}
Then $X_3$ is a smooth surface of degree $6$ and sectional genus $3$ in $\PP^4$. Looking at the classification given in \cite{Io81} we see that $X_3$ is a Bordiga surface, see Remark \ref{rmk:IV}, so $\alpha_3=0$ and
\begin{equation*}
D_3=4L - \sum_{i=1}^{10} E_i.
\end{equation*}
This gives case $\textbf{(VII.6)}$. 

$\diamond$  If $\alpha_2=1$ then the third adjunction map gives a pair $(X_3, \, D_3)$, with $X_3 \subset \PP^4$ and whose intersection matrix is 
\begin{equation*}
\left( \begin{array}{cc}
(D_3)^2 & K_{X_3}D_3 \\
 K_{X_3}D_3 & (K_{X_3})^2\\
 \end{array} \right) =
\left( \begin{array}{cc} \; 7 & -1 \\
 -1 &  \alpha_3\\
 \end{array} \right).
\end{equation*}
By Hodge Index Theorem we obtain $\alpha_3=0$, hence $(K_{X_3})^2=0$. This is a contradiction, because the unique non-degenerate, smooth rational surface of degree $7$ in $\PP^4$ has $K^2=-2$, see Remark \ref{rmk:V}. So this case does not occur. 

\bigskip \bigskip

$\bullet \, \mathbf{\alpha_1=7.}$ In this case the intersection matrix on the surface $X_1 \subset \PP^5$ is 
\begin{equation*}
\left( \begin{array}{cc}
(D_1)^2 & K_{X_1}D_1 \\
 K_{X_1}D_1 & (K_{X_1})^2\\
 \end{array} \right) =
\left( \begin{array}{cc} \; 10 & 0 \\
 0 &  0 \\
 \end{array} \right).
\end{equation*}
The Hodge Index Theorem implies that $K_{X_1}$ is numerically trivial. So $X_1$ is a minimal Enriques surface, and $X$ is the blow-up of $X_1$ at $7$ points. This yields 
$\textbf{(VII.7)}$. 

\bigskip 

The proof of the existence for the cases marked with $(*)$  goes as follows. We first choose
$\alpha_1 \in \{1,\ldots,7\}$. According to Proposition
\ref{prop:adjunction}, we need a rank two Steiner bundle $\mF$ on $\PP^2$ with a
resolution like \eqref{eq:M-VII} and having precisely $\alpha_1$ distinct unstable
lines. Bundles with these properties are described in Proposition \ref{prop:unstable-b-8}.

Then, we take $\PP(\mF)$ and we choose a sufficiently general
global section $\eta$ of $\mL = \mO_{\PP(\mF)}(3\xi-2\ell)$. We do this by looking directly
at the image $Y$ of $\mathfrak{q} \colon \PP(\mF) \to \PP^5$, namely we consider $\eta$ as a global section of $\mR(3)$ via the natural identification given by \eqref{eq:S3E}. In this setting, $Y$ is a scroll of degree $6$ in $\PP^5$ defined by the minors of
order $3$ of the $3 \times 4$ matrix of linear forms $N$ over $\PP^5$ obtained via the
construction of \S \ref{sub:proj.bundle}, i.e.
\begin{equation*}
\mO_{\PP^5}(-1)^4 \stackrel{N}{\longrightarrow} \mO_{\PP^5}^3,
\end{equation*}
and the zero locus of $\eta$ is a cubic hypersurface of $\PP^5$ containing the union of two surfaces $S_1$ and $S_2$ in $Y$, both obtained as the image via $\mathfrak{q}$ of a divisor belonging to $|\mO_{\PP(\mF)}(\ell)|$. 

Concretely, $S_1$ and $S_2$ lie in the
net generated by the rows of $N$, i.e.  they can be defined by the $2\times 2$ minors of $4\times 2$ matrices obtained taking random linear combinations of these rows. 

Now we compute the resolution of the
homogeneous ideal defining $S_1 \cup S_2$ in $\PP^5$, we take a general cubic in this ideal and we consider the residual surface $X_1$ in $Y$.
The image of the first adjunction map
\begin{equation*}
\varphi_{|K_X+H|} \colon X \to  \PP^5
\end{equation*}
is precisely $X_1,$ so that $X$ is the blow-up of $X_1$ at $\alpha_1$ points.

It remains to compute $\alpha_2$, or equivalently $(K_{X_2})^2$. To do
this, we observe that the second adjunction map of $X$ is defined by the restriction to $X_1$ of the linear system $|\mO_{Y}(2\xi -\ell)|$, and this in turn coincides with the
restriction to $X_1$ of the linear system generated by the six quadrics in the ideal defining $S_1$.

The image of $X_1$ via this linear system is the surface $X_2$, hence we compute 
$(K_{X_2})^2$ by taking the dual of the resolution of the homogeneous
ideal of $X_2$ in the target $\PP^5$. All this, together with the verification that 
 $X_1$ (and hence $X$) is smooth,
is done with the help of  \texttt{Macaulay2}.  
In the Appendix at the end of the paper we explain in detail how this computer-aided construction is performed. 
\end{proof}

\begin{remark}  \label{rmk:Alexander}
In \cite{Al88}, Alexander showed the existence of a non-special,
linearly normal surface of degree $9$ in $\PP^4$, obtained by
embedding the blow-up of $\PP^2$ at $10$ general points via the very
ample complete linear system  
\begin{equation*}
\bigg|13L - \sum_{i=1}^{10} 4 E_i \bigg|.
\end{equation*}
By using LeBarz formula, see \cite[Th\'{e}or\`{e}me 5]{LeB90}, we can see that Alexander
surface has precisely one $6$-secant line. Projecting from this line
to $\PP^2$, one obtains a birational model of a general triple cover;
it is immediate to see that this corresponds to case
$\textbf{(VII.6)}$ in Proposition \ref{prop.typeVII}.  
\end{remark}

\begin{remark} \label{rmk:Enriques}
Let us say something more about case (VII.7). Since $\alpha_1=7$, we
deduce that $\mF$ has $7$ unstable lines, hence it is a logarithmic
bundle (see Proposition \ref{prop:unstable-b-8}). In this situation,
the surface $X_1$ is a smooth Enriques surface of degree $10$ and
sectional genus $6$ in $\PP^5$, that is a so-called \emph{Fano
  model}. Actually, one can check that $X_1$ is contained into the
Grasmannian $\Gr(1, \, \PP^3)$ as a \emph{Reye congruence}, i.e. a
$2$-dimensional cycle of bidegree $(3, \, 7)$, see \cite[Theorem
4.3]{G93}. In particular, $X_1$ admits a family of $7$-secant planes,
and the projection from one of these planes provides a birational
model of the triple cover $f \colon X \to \PP^2$ (in fact, $X$ is the
blow-up of $X_1$ at $7$ points).      

For more details about Fano and Reye models, see \cite{Cos83, CV93}.
\end{remark}

\subsubsection{Some further considerations on
 triple planes of type VII} \label{subsec:further-cons}

We mentioned in the previous subsection that we are able to construct many, but not all cases of triple planes of type \ref{VII} (see Proposition \ref{prop.typeVII}). We conjecture that the remaining cases do not exist. More precisely, our expectation is that the values of $\alpha_1$ and $\alpha_2$
should necessarily satisfy the rule
\begin{equation*}
\alpha_2={{7-\alpha_1}\choose 2}.
\end{equation*}

Let us explain now what is the geometric evidence beyond our conjecture. 
The second adjunction map $\varphi_2 \colon X_1 \to X_2 \subset \PP^5$ 
 can be lifted to the map $\zeta \colon \PP(\mF) \to \PP^5$ associated
with the linear system $|\mO_{\PP(\mF)}(2\xi - \ell)|$. Note that
\begin{equation*}
H^0(\PP(\mF),\mO_{\PP(\mF)}(2\xi - \ell)) \simeq H^0(\PP^2,S^2\mF(-1))
\simeq \wedge^2 W^\vee,
\end{equation*}
where the last isomorphism is obtained taking global sections in the
second exterior power of the short exact sequence 
\begin{equation*}
 0 \to W^\vee \ts \mO_{\PP(V)}(-1) \to U \ts \mO_{\PP(V)} \to \mF \to 0
\end{equation*}
defining $\mF$ (see \eqref{eq:M-N}), namely 
\begin{equation*}
0 \to \wedge^2 W^\vee \ts \mO_{\PP^2}(-3) \to W^\vee \ts U \ts
\mO_{\PP^2}(-2) \to S^2U \ts \mO_{\PP^2}(-1) \to S^2 \mF(-1) \to 0.
\end{equation*}

One can show that the projective closure $Y'$ of the image of the map $\zeta \colon \PP(\mF) \dashrightarrow
\PP(\wedge^2 W^\vee)$ is contained in the Plücker quadric $\GG=\Gr(1,
\, \PP(W^\vee))$ and that $Y'$ is the degeneracy locus of a map on
$\GG$ defined by the tensor $\phi \in U \otimes V \otimes W$ considered 
in \S \ref{a}. More precisely, denoting by $\mU$ the
tautological rank two subbundle on $\GG$, once noted that
$H^0(\GG, \, \mU^\vee)=W$ we see that $\phi$ gives a morphism
\begin{equation*}
V^\vee \ts \mU \to U \ts \mO_{\GG}.
\end{equation*}

The variety $Y'$ is the vanishing locus of the determinant of this morphism,
so that $Y'$ can be expressed as a complete
intersection of the Plücker quadric and a cubic hypersurface in $\PP^5$. 

The locus where this
morphism has rank $\le 4$ is contained in the singular locus of $Y'$
and coincides with it for a general choice of $\mF$. By Porteous' formula, for such a general choice we expect that $Y'$ has $21$ singular points. One can see that these points are precisely the images of the sections of negative self-intersection of the Hirzebruch surfaces in $\PP(\mF)$  lying above the smooth
conics in $\PP^2$ where $\mF$ splits as $\mO_{\PP^1}(1) \oplus
\mO_{\PP^1}(7)$, once chosen an isomorphism to $\PP^1$ (it would be
natural to call these conics \emph{unstable conics}, and the argument above shows that there are
 in general $21$ of them).

Also, the indeterminacy locus of $\zeta$ is exactly the union of the
sections of negative self-intersection on the Hirzebruch surfaces
lying above the unstable lines of $\mF$. So,
$\alpha_1$ and $\alpha_2$ should depend only on $\mF$ and not on $X$,
and moreover $\alpha_1$ should determine $\alpha_2$. However, it
is not clear yet how the number of unstable lines determines the precise number of
unstable conics. 

\section{Moduli spaces} \label{sec:moduli}

\def\De{\mathrm{Def}}
\def\Aut{\mathrm{Aut}}
\def\PGL{\mathrm{PGL}}
\def\GL{\mathrm{GL}}
\def\FR{\mathfrak{R}}
\def\FT{\mathfrak{T}}
\def\FM{\mathfrak{M}}
\def\FN{\mathfrak{N}}

In this section we  describe some moduli problems related to our triple planes. For $b \in \{2, \, 3, \, 4 \}$ we set
\begin{equation*}
\mE_b:= \begin{cases} 
\mO_{\PP^2}(-1) \oplus \mO_{\PP^2}(-1) &\mbox{if } b=2 \\
\mO_{\PP^2}(-1) \oplus \mO_{\PP^2}(-2) &\mbox{if } b=3 \\
\mO_{\PP^2}(-2) \oplus \mO_{\PP^2}(-2) &\mbox{if } b=4, 
\end{cases}
\end{equation*}
whereas for $b \in \{5, \, 6, \, 7, \, 8\}$ we denote by
$\mF_b=\mE_b(b-2)$ a rank $2$ Steiner bundle on $\PP^2$ having
sheafified minimal graded free resolution of the form   
\begin{equation*}
  0  \to \mO_{\PP^2}(1-b)^{b-4} \to \mO_{\PP^2}(2-b)^{b-2} \to \mE_b \to 0.
  \end{equation*}
Then, for any $b \in \{2, \ldots, 8\}$, we define two spaces $\FN_b$ and $\FM_b$ as follows: 
\begin{equation*}
\FN_b=\left\{(\mE_b,\, \eta) \;  \Bigg| \minibox{$\; \eta \in \PP
    (H^0(\PP^2, \, S^3\mE_b^\vee \otimes \wedge^2 \mE_b))$ is the building section  \\
    of a general triple plane with $p_g=q=0$ \\
  and Tschirnhausen bundle $\mE_b$} \right\} / \simeq
\end{equation*}
\begin{equation*}
\FM_b= \left\{(\mE_b,\, \eta) \;  \Bigg| \minibox{$\; \eta \in \PP H^0(\PP^2, \, S^3\mE_b^\vee \otimes \wedge^2 \mE_b)$ and  $D_0(\eta)$ provides
  a general triple \\ \; plane with $p_g=q=0$ and Tschirnhausen bundle $\mE_b$} \right\} / \sim
\end{equation*}
where we set $(\mE_b, \, \eta) \simeq (\mE_b', \, \eta')$ if and only if there is an isomorphism $\Psi \colon \mE \to \mE'$ such that $\Psi^* \eta' = \eta$ and the following diagram commutes 
\begin{equation*} 
\begin{CD}
\mE_b  @>{\Psi} >> \mE'_b\\
@VVV  @VVV\\
\PP^2 @> {id}>> \PP^2.\\
\end{CD}
\end{equation*}
whereas $(\mE_b, \, \eta) \sim (\mE_b', \, \eta')$ if and only if there is an isomorphism $\Psi \colon \mE \to \mE'$ and an automorphism $\lambda \colon \PP^2 \to \PP^2$ such that $\Psi^* \eta' = \eta$ and the following diagram commutes 
\begin{equation*} 
\begin{CD}
\mE_b  @>{\Psi} >> \mE'_b\\
@VVV  @VVV\\
\PP^2 @> {\lambda}>> \PP^2.\\
\end{CD}
\end{equation*}
We have $\FM_b = \FN_b/ \mathrm{PGL}_3(\CC)$, because the equivalence  $(\mE_b, \, \eta) \simeq (\mE_b', \, \eta')$ is obtained from $(\mE_b, \, \eta) \sim (\mE_b', \, \eta')$ via the natural $\mathrm{PGL}_3(\CC)$-action on the base. Note that, with the terminology of \cite[Chapter 4]{HuyLehn10}, the
 pair $(\mE_b, \, \eta)$ consisting of the Tschirnhausen bundle and of
 the building section is a \emph{framed sheaf}. 

Given a general triple plane $f \colon X \to \PP^2$  branched over a
curve of degree $2b$, by Theorem \ref{thm triple:1} and \ref{teo:res}
we can functorially associate with $(X, \, f)$ a framed sheaf $(\mE_b,
\, \eta)$, and conversely.  In other words, considering the set of
framed sheaves $(\mE_b, \, \eta)$ up to the equivalence relation
$\simeq$ or $\sim$ defined above actually amounts to
consider the set of pairs $(X, \, f)$ up to the corresponding equivalence relation.  

Thus, from this point of view, $\mathfrak{M}_b$ can be identified with
the moduli space of the pairs $(X, \, f)$ up to isomorphisms, and
$\mathfrak{N}_b$ with the moduli space of the pairs $(X, \, f)$ up to
\emph{cover} isomorphisms.
 
In the sequel, we will use interchangeably the above notation $\FN_b$
and  $\FM_b$, with $b \in \{2, \ldots, 8\}$, and  $\FN_i$
and $\FM_i$, with $i \in \{\ref{I},\ldots,\ref{VII}\}$. In each case, the moduli space 
$\FN_b$ can be constructed as follows: 
\begin{itemize}
\item take the versal deformation space $\De(\mE_b)$ of $\mE_b;$
\item stratify $\De(\mE_b)$ in such a way that $H^0(\PP^2, \, S^3 \mE_b^\vee \otimes \wedge^2 \mE_b)$ has constant rank and consider the locally trivial projective bundle over each
  stratum whose fibres are given by $\PP H^0(\PP^2, \, S^3 \mE_b^\vee \otimes \wedge^2 \mE_b)$;
\item consider the quotient of this projective bundle by the natural action of the group $\Aut(\mE_b)$.
\end{itemize}
In order to obtain $\FM_b$, we must further take the quotient of the above moduli space by the natural action of $\PGL_3(\CC)$. In particular, the expected dimensions of $\FN_b$ and $\FM_b$ will be given by
\begin{equation} \label{eq:exp-dim}
\begin{split}
\textrm{exp-dim} \, \FN_b & = \dim \De(\mE_b) + h^0(\PP^2, \, S^3 \mE_b^\vee \otimes \wedge^2 \mE_b) - \dim \Aut(\mE_b), \\
\textrm{exp-dim} \, \FM_b & = \dim \De(\mE_b) + h^0(\PP^2, \, S^3 \mE_b^\vee \otimes \wedge^2 \mE_b) - \dim \Aut(\mE_b)-8.
\end{split}
\end{equation}
From now on, we will simply write $\mE$ instead of $\mE_b$ if no confusion can arise. 

\subsection{Moduli  of triple planes with decomposable
  Tschirnhausen bundle}

Let us first consider cases $\textrm{I, II, III}$. Here $\mE$ splits
as a sum of two line bundles and it is rigid.

\begin{theorem} \label{primimoduli}
The following holds$:$
\begin{enumerate}[itemsep=2pt, label=\bf{\roman*)}]
\item \label{moduli:I} the moduli space $\moduli_{\ref{I}}$ consists of a single point$;$
\item \label{moduli:II} the moduli space $\moduli_{\ref{II}}$ is
  unirational of dimension $7;$
\item \label{moduli:III} the moduli space $\moduli_{\ref{III}}$ is
  unirational of dimension $12.$
\end{enumerate}
\end{theorem}
\begin{proof} 
As a preliminary step, note that in all these cases the bundle 
$S^3 \mE^\vee \otimes \wedge^2 \mE$ is globally generated. Therefore, 
Theorem \ref{thm triple:1} applies and shows that the moduli spaces
$\moduli_{\ref{I}}$, $\moduli_{\ref{II}}$ and 
$\moduli_{\ref{III}}$ are obtained as a quotient of a 
Zariski dense open subset of $H^0(\PP^2,\, S^3 \mE^\vee \otimes \wedge^2 \mE)$
by the action of some linear group, so that all of them are irreducible, 
unirational varieties.

Let us check \ref{moduli:I}. In this case, the branch curve $B \subset
\PP^2$ is a tricuspidal plane quartic curve, which is
unique up to projective transformations. By a topological monodromy
argument (see \cite[\S 58]{ST80}) and Grauert-Remmert extension
theorem (see \cite[XII.5.4]{SGA1}) this implies that the number of
triple planes of type \ref{I} up to isomorphisms equals the number of
group epimorphisms
\begin{equation*}   
\varrho \colon \pi_1(\PP^2- B) \to \mathfrak{S}_3
\end{equation*}
up to conjugation in $\mathfrak{S}_3$. Now, it is well-known that
\begin{equation*}
\pi_1(\PP^2- B)=\textrm{B}_3(\PP^1)= \langle \alpha, \, \beta \; | \; \alpha^3 = \beta^2 = (\beta \alpha)^2 \rangle,
\end{equation*}
see \cite[Chapter 4, Proposition 4.8]{Dim92}, and this group has a unique epimorhism $\varrho$ to
$\mathfrak{S}_3$ up to conjugation. In fact, $\varrho(\alpha)$ must be a $3$-cycle whereas $\varrho(\beta)$ 
must be a transposition, so we may assume 
\begin{equation*}
\varrho(\alpha)=(1 \, 2 \, 3), \quad \varrho(\beta)=(1 \, 2).
\end{equation*}
This proves that $\moduli_{\ref{I}}$ consists of a
single point.
\bigskip

Let us now analyze \ref{moduli:II}. Recall that in this case the
branch locus $B \subset \PP^2$ is a plane sextic curve with six cusps
lying on the same conic. Each of these curves can be written as 
\begin{equation} \label{eq:sextic}
(f_2)^3+(f_3)^2=0, 
\end{equation}
where $f_k$ denotes a homogeneous form of degree $k$, and the construction depends on 
\begin{equation*}
6+10-1-\dim \, \PGL_3(\CC)=7
\end{equation*} 
parameters. The same monodromy argument used in part 
\ref{moduli:I} shows that this also computes the effective dimension  $\dim \, \moduli_{\ref{II}}$. More precisely, we can see that every fixed curve $B$ of equation $\eqref{eq:sextic}$ is the branch locus of a unique triple cover up to isomorphisms, namely the one whose birational model is provided by the hypersurface
\begin{equation*}
z^3+bz+c=0, 
\end{equation*}
where  $b=-f_2 /\sqrt[3]{4}$ and $c= f_3 / \sqrt{-27}$. In fact, we have
\begin{equation*}
\pi_1(\PP^2-B)=(\ZZ/2 \ZZ) \ast (\ZZ/3 \ZZ) = \langle \alpha, \, \beta \; | \; \alpha^3 = \beta^2 = 1 \rangle, 
\end{equation*} 
see \cite[Chapter 4, Proposition 4.16]{Dim92}, and this group has a unique
epimorhism to $\mathfrak{S}_3$ up to conjugation.
\bigskip

We finally look at \ref{moduli:III}, where $\mE= \mO_{\PP^2}(-2) \oplus \mO_{\PP^2}(-2)$. The automorphism group of $\mE$ is isomorphic to $\GL_2(\CC)$.
Moreover
\begin{equation*}
h^0(\PP^2, \, S^3 \mE^\vee \otimes \wedge^2 \mE)= h^0 \big(\PP^2, \, \mO_{\PP^2}(2)^4 \big)= 24,
\end{equation*}
hence \eqref{eq:exp-dim} implies
\begin{equation*}
\textrm{exp-dim} \, \moduli_{\ref{III}}= 24-4-8=12.
\end{equation*}
This number coincides with the effective dimension $\dim \, \moduli_{\ref{III}}$. In fact,
in this case $X$ is the blow-up at $9$ points of $\FF_n$, with $n \in \{0, \, 1, \, 2, \,  3\}$.
The stratum of maximal dimension corresponds to the value of $n$ such that $\Aut(\FF_n) = H^0(\FF_n, \, T_{\FF_n})$ has minimal dimension, namely to $n=0$ 
for which we have
\begin{equation*}
\dim \, \moduli_{\ref{III}}= 2 \cdot 9 - \dim \, \Aut(\FF_n) = 18 -6 =12. 
\end{equation*}
\end{proof}

\subsection{Moduli of triple planes with stable
  Tschirnhausen bundle}

We now start the analysis of the cases \ref{IV}, \ldots, \ref{VII}, where $\mE$ is indecomposable.
Using the notation introduced in \S \ref{sec.general}, we will write $\mF = \mE(b-2)$, so that $\mF$ fits into the short exact sequence
\begin{equation*}
 0  \to \mO_{\PP^2}(-1)^{b-4} \to \mO_{\PP^2}^{b-2} \to  \mF \to 0.
\end{equation*}
Thus $\De(\mE) = \De(\mF)$ and
\begin{equation*}
 H^0(\PP^2, \, S^3 \mE^\vee \otimes \wedge ^2 \mE ) = H^0(\PP^2, \, S^3 \mF (6-b)).
\end{equation*}
The vector bundle $\mF$ is stable (Theorem \ref{teo:res}), so $\Aut(\mF) = \CC^*$; its deformation space $\De(\mF)$ is described for instance in \cite[Introduction]{Ca02}, and we have
\begin{equation*}
\dim \, \De(\mF) = 3(b-4)(b-2)-1 = (b-1)(b-5).
\end{equation*}
Then \eqref{eq:exp-dim} yields
\begin{equation} \label{eq:exp-dim-irr}
\begin{split}
\dim \, \FN_b= \textrm{exp-dim} \, \FN_b & = (b-1)(b-5) + h^0(\PP^2, \, S^3 \mF (6-b))-1, \\
\textrm{exp-dim} \, \FM_b & = (b-1)(b-5) + h^0(\PP^2, \, S^3 \mF (6-b))-9. 
\end{split}
\end{equation}
Furthermore, the equality $\textrm{exp-dim} \, \FM_b = \dim \, \FM_b$ holds if and only if $\textrm{PGL}_3(\CC)$ acts on $\FN_b$ with generically finite stabilizer.

\begin{theorem} \label{secondimoduli}
For $i \in \{\ref{IV}, \ref{V}, \ref{VI}\}$ the moduli space $\FN_i$ is rational and irreducible, while
$\moduli_i$ is unirational of dimension $\dim \, \FN_i-8,$ where
\begin{enumerate}[itemsep=2pt, label=\bf{\roman*)}]
\item  $\dim \, \FN_{\ref{IV}}=23;$
\item $\dim \, \FN_{\ref{V}}=24;$
\item $\dim \, \FN_{\ref{VI}}=23.$ 
\end{enumerate}
Moreover
the moduli space $\FN_{\ref{VII}}$ has at least seven irreducible
components, all unirational of dimension $20$, that are distinguished by the
number $\alpha_1 \in \{1, \ldots, 7\}$ of unstable lines for $\mF.$ 
\end{theorem}

First of all we note that, as in the proof of Theorem \ref{primimoduli}, in cases
$\ref{IV}$ and $\ref{V}$ the bundle  
$S^3 \mE^\vee \otimes \wedge^2 \mE \simeq S^3 \mF(6-b)$ is globally
generated.
Indeed, in these cases $b\le 6$ and $\mF$ is globally generated, so
the same is true for $S^3 \mF$ and for $S^3 \mF(6-b)$.
Therefore, 
the spaces $\moduli_{i}$ 
and $\FN_{i}$ are irreducible 
 as soon as the parameter space of the bundle $\mE$, or
equivalently of $\mF$, is irreducible. Moreover, since
$\FN_{i}$ is an open subset of a projective bundle over such
parameter space, rationality of the latter will imply rationality of
the former, and also unirationality of $\FM_i$.

The proof of Theorem \ref{secondimoduli} is based on a case-by-case analysis, that will be done in 
\S \ref{subsec:moduli-IV}, \ref{subsec:moduli-V}, \ref{subsec:moduli-VI}, \ref{subsec:moduli-VII} below. Our strategy is to compute
$\dim \FN_i$ and to show that $\PGL_3(\CC)$ acts on $\FN_i$ with generically
finite stabilizers for all $i \in \{\ref{IV}, \, \ref{V}, \,\ref{VI},  \, \ref{VII}\}$, to prove that $\FN_i$ is rational and
irreducible for $i \in \{\ref{IV}, \, \ref{V}, \,\ref{VI}\}$, and finally to find at
least $7$ irreducible unirational components of $\FN_{\ref{VII}}$.

\subsubsection{Moduli of triple planes of type \ref{IV}} \label{subsec:moduli-IV}

\begin{proposition}
The moduli space $\FN_{\ref{IV}}$ is an open dense
subset of $\PP^{23}$, in particular it is irreducible and
rational. The space $\FM_{\ref{IV}}$ has dimension $15$.
\end{proposition}

\begin{proof}
Case $\ref{IV}$, i.e. $b=5$, was analyzed in Proposition \ref{prop.TypeIV}.
We have $\mF = T_{\PP^2}(-1)$ and a natural identification
\begin{equation*} 
H^0(\PP^2, \, S^3 \mF (1)) = H^0(\PD^2, \, T_{\PD^2}(2))=\CC^{24}.
\end{equation*}
Set $\PP^{23} = \PP H^0(\PD^2, \, T_{\PD^2}(2))$ and 
observe that the bundle $\mF$ is rigid, stable and unobstructed, so the 
moduli space consists of a single, reduced point. Consequently, the triple cover
$f \colon X \to \PP^2$ only depends on the section $\eta \in
H^0(\PD^2, \, T_{\PD^2}(2))$ or, better, on its proportionality class $[\eta]$, that lies 
in a Zariski dense open subset of $\PP^{23}$.

By \eqref{eq:exp-dim-irr} we have $\textrm{exp-dim} \, \FM_{\ref{IV}} =15$. 
It remains to show that $\textrm{exp-dim} \,\moduli_{\ref{IV}} = \dim  \, \moduli_{\ref{IV}}$ or, equivalently, that $\PGL_3(\CC)$
acts on $\PP^{23}=\PP H^0(\PD^2, \, T_{\PD^2}(2))$ with
generically finite stabilizer. Take a generic element $\eta \in \PP^{23}$ and let
$Z=D_0(\eta) \subset \PD^2$ be its vanishing locus and $G=G_{\eta} \subset
\PGL_3(\CC)$ its stabilizer. So $Z$ consists of $13$ reduced points
and we want to show that $G$ is finite. Every homography
in $G$ must preserve $Z$ and hence permute its $13$ points, so we
obtain a group homomorphism  
\begin{equation*}
\psi \colon G \to \mathfrak{S}_{13}. 
\end{equation*}
If $L \subset \mathbb{P}^2$ is a line, we have
\begin{equation} \label{eq:34}
T_{\PD^2}(2) | _L = \mO_{L}(3) \oplus \mO_L(4).  
\end{equation}  
Now set $Z':=Z \cap L$ and  $c:=\mathrm{length}(Z')$. Arguing as in part \ref{u:3} of Lemma \ref{lem:unstable-line}, we deduce the existence of a surjection $T_{\PD^2}(2) | _L \to \mO_L(7-c)$, and using 
\eqref{eq:34} this yields $c \leq 4$. So there are no more than $4$
points of $Z$ on a single line, hence the support of
$Z$ contains at least $4$ points in general linear position.   

Now, a homography in $\ker \psi$ must fix the subscheme $Z$
pointwise. Since a homography of the plane fixing at least $4$
points in general position is the identity, we have that $\psi$ is
injective. So $G$ is a subgroup of $\mathfrak{S}_{13}$, hence a finite
group. 
\end{proof}

\subsubsection{Moduli of triple planes of type \ref{V}} \label{subsec:moduli-V}

\begin{proposition}
  The moduli space $\FN_{\ref{V}}$ is a Zariski open dense subset of a $\PP^{19}$-bundle over $\PP^5$, in particular it is rational and irreducible of dimension $24$. The space $\FM_{\ref{V}}$ has dimension $16$.
\end{proposition}

\begin{proof}
Case \ref{V}, i.e. $b=6$, was analyzed in Proposition \ref{prop.TypeV}.
The bundle $\mF=\mF_6$ is determined by its set of unstable lines,
which form a smooth conic $\mathscr{W}(\mF) \subset \PD^2$, so we can 
identify the moduli space of $\mF$ with the open subset $\mathscr{U} \subset \PP^5$
consisting of smooth conics via the Veronese embedding.
This is the base of our $\PP^{19}$-bundle.

Proposition \ref{prop.TypeV} (cf. also
\S \ref{a}) shows that, once
chosen the Tschirnhausen bundle $\mF$, we have a $4$-dimensional space
$U=H^0(\PP^2,\mF)$ and a corresponding projective space $\PP^3=\PP(U)$, together with a fixed twisted
cubic $C \subset \PP^3$ such that $\PP(\mF)$ is the blow-up of $\PP^3$
at $C$. Moreover, 
the building sections $\eta$ of the triple plane are in bijection with an
open dense subset of the space of cubic surfaces, in view of the identification
\begin{equation} \label{allcubics}
H^0(\PP^2, \, S^3 \mF(6-b)) = H^0(\PP^3, \, \mO_{\PP^3}(3)) = \CC^{20},
\end{equation}
so their proportionality classes belong to an open dense subset of $\PP^{19}=\PP H^0(\PP^2,\, S^3 \mF(6-b))$, and our claim about $\FN_{\ref{V}}$ is proven.

Now \eqref{eq:exp-dim-irr} yields $\textrm{exp-dim} \, \FM_{\ref{V}}=16$, so it only remains 
to show that $\PGL_3(\CC)$
acts on the set of pairs $(\mF, \, \eta)$ with
generically finite stabilizer. Let $G = G_{(\mF, \, \eta)} \subset
\PGL_3(\CC)$ be the stabilizer of the pair $(\mF, \, \eta)$. Then every element
$g \in G$ must fix $\mF$, and hence the conic $\mathscr{W}(\mF)$. By
\cite[p. 154]{FulHa91}, the subgroups of automorphisms of $\PP^n$ that preserves a
rational normal curve $C_n$ is precisely $\textrm{PGL}_2(\CC)$, so $G$
is a subgroup of a copy of $\textrm{PGL}_2(\CC)$ inside
$\textrm{PGL}_3(\CC)$. On the other hand, $g$ fixes $\eta \in
H^0(\PP^3, \, \mO_{\PP^3}(3))$, hence it fixes the cubic surface $S
\subset \PP^3$. 

Next, we have seen in Lemma \ref{contracts} (cf. also Remark \ref{rem:cubic}) 
that the image in $\PP^3$ of the negative sections lying above the lines of
$\mathscr{W}(\mF)$ is precisely the
twisted cubic $C$. The whole construction is therefore $g$-invariant, so $g$ must preserve the intersection $S \cap C$. 

Furthermore, the construction giving
rise to the $2\times 3$ matrix $N$ whose $2\times 2$ minors define $C$,
cf. \eqref{eq:N-V}, can be reversed in order to give back the matrix $M$ presenting
$\mF$, cf. \eqref{eq:M-V}. Since $M$ is generic, this implies that $N$ and $C$ are generic. In addition, by \eqref{allcubics} we also know that the cubic $S$ corresponding
to the building section $\eta$ can be chosen generically.
In particular, the intersection $S \cap C$ is
reduced for a general choice of our data, i.e. it consists of $9$ distinct points.

Summing up, we get a group homomorphism
\begin{equation*}
\psi \colon G \to \mathfrak{S}_9
\end{equation*} 
that must be injective since an element of $\mathrm{PGL}_2(\CC)$ fixing  
at least $3$  distinct points is necessarily the identity. 
So $G$ is a subgroup of $\mathfrak{S}_9$, hence a finite group.
\end{proof}

\subsubsection{Moduli of triple planes of type \ref{VI}} \label{subsec:moduli-VI}

We denote by $\Hilb_d(\PD^2)$ the Hilbert scheme of $0$-dimensional subschemes of length $d$ of $\PD^2$.

\begin{proposition}
The moduli space $\FN_{\ref{VI}}$ is a Zariski open dense subset of a
$\PP^{11}$-bundle over $\Hilb_6(\PD^2)$, in
particular it is a rational variety of dimension $23$. The moduli space $\FM_{\ref{VI}}$
 has dimension $15$.
\end{proposition}
\begin{proof}
Case $\ref{VI}$ was analyzed in Proposition
\ref{prop.TypeVI}. We mentioned in \S \ref{b}, cf. case $b=7$ before 
Proposition \ref{prop:unstable-b-8}, that $\mF=\mF_7$ is a logarithmic
bundle, i.e. it has six unstable lines 
which are in general linear position,
and that these six lines in turn uniquely determine $\mF$. This identifies 
the moduli space of Steiner bundles of type $\mF_7$ as an open dense subset $\mathscr{U}$ of the  Hilbert scheme of six points of $\PD^2$. 

We have a direct image sheaf $\mR(3)$, fitting into  \eqref{eq.VI.5}, and a natural identification
\begin{equation*}
H^0(\PP^2, \, S^3 \mF(6-b))= H^0( \PP^4, \,
\mR(3)) = \CC^{12},
\end{equation*} 
see the proof of Claim \ref{oneclaim}. The sheaf $\mR(3)$ is supported on a
determinantal cubic threefold $Y \subset \PP^4$. In addition, the vanishing 
locus of a general global section of $\mR(3)$
is a Bordiga surface $X_1 \subset \PP^4$ and, moreover, the divisor $X =D_0(\eta) \subset \PP(\mF)$ is the blow-up of $X_1$ at the six nodes of $Y$. Summing up, the proportionality classes $[\eta]$ 
of building sections of triple covers of
type \ref{VI} lie in a dense open subset of $\PP^{11} = \PP H^0(\PP^2,
\, S^3 \mF(6-b))$, and this proves our claim about $\FN_{\ref{VI}}$.

We now consider the moduli space $\FM_{\ref{VI}}$.
First, \eqref{eq:exp-dim-irr} implies $\dim \FM_{\ref{VI}} =15$.
In order to conclude the proof, we must show that 
$\PGL_3(\CC)$
acts on the set of pairs $(\mF, \, \eta)$ with
generically finite stabilizer. Let $G = G_{(\mF, \, \eta)} \subset
\PGL_3(\CC)$ be the stabilizer of the pair $(\mF, \, \eta)$. Then every element
$g \in G$ must fix $\mF$, and hence the set of its six unstable
lines. Consequently, $g$ permutes the corresponding six points in
$\PD^2$, which are in general position. This in turn defines a group homomorphism 
\begin{equation*}
\psi \colon G \to \mathfrak{S}_{6},
\end{equation*}
which must be injective since a
homography of the plane that fixes at least $4$ points in general
position is the identity. So $G$ is a subgroup of $\mathfrak{S}_{6}$, hence a finite group.
\end{proof}

\subsubsection{Moduli of triple planes of type \ref{VII}} \label{subsec:moduli-VII}

Let us finally consider case \ref{VII},
i.e. $b=8$. We need the following preliminary result.

 \begin{proposition} \label{prop:S2-8}
 Assume $b=8$ and let $\mF:=\mF_8$ be a Steiner bundle with $\alpha_1$ unstable lines. Then
  \begin{equation} \label{eq:geq-alpha}
  h^0(\PP^2, \, S^3 \mF(-2)) \geq \alpha_1.
  \end{equation}
\end{proposition}

\begin{proof}
  Let $L_1,\ldots,L_{\alpha_1}$ be the unstable lines of $\mF$. 
  We can perform the reduction of $\mF$ along
  such unstable lines,
  i.e., a sequence of elementary transformations of $\mF$ along the $L_i$, see
  \cite[\S 2.7 - 2.8]{DK} and \cite[Proposition 2.1]{Val00b}.
  This gives an exact sequence
  \begin{equation} \label{eq:reduction-alpha}
  0 \to \mK \to \mF \to \bigoplus_{i=1}^{\alpha_1} \mO_{L_i} \to 0,
  \end{equation}
  where $\mK$ is a vector bundle of rank $2$. 
  From \eqref{eq:reduction-alpha} we get $H^i(\PP^2,\mK(-1))=0$ for all $i$. Computing Chern classes and 
  applying Proposition \ref{Fsteiner} to $\mK$, we see that $\mK$ behaves according to the following table:
  \begin{equation}
    \label{K}
  \begin{tabular}[h]{c||c|c|c|c|c}
    $\alpha_1$ & $1, \, 2, \, 3$ & $4$ & $5$ & $6$ & $7$ \\
    \hline 
    $\mK$ & $\mF_{8-\alpha_1}$ & $\mO_{\PP^2}^2$ & $\mO_{\PP^2}(-1)
    \oplus \mO_{\PP^2}$ & $\mO_{\PP^2}(-1)^2$ &  $\Omega_{\PP^2}^1$ 
  \end{tabular}
  \end{equation}
Indeed, $\mK$ is a Steiner bundle for $\alpha_1=1, \,2,\,3, \,4$ (corresponding 
to the cases $b=7, \, 6, \, 5, \, 4$ in Proposition \ref{Fsteiner}).
  For $\alpha_1=5$ (the
  case $b=3$ in Proposition \ref{Fsteiner}) we have $\mK \simeq
  \mO_{\PP^2}(-1) \oplus \mO_{\PP^2}$. Finally, for $\alpha_1=6, \, 7$ (the cases $b=2, \, 1$  in Proposition
  \ref{Fsteiner}) we have that $\mK^\vee(-1)$ is a
  Steiner bundle respectively of the form $\mO_{\PP^2}^2$ for
  $\alpha_1=6$ or $T_{\PP^2}(-1)$ for $\alpha_1=7$, and hence $\mK \simeq \mO_{\PP^2}(-1)^2$ or $\mK \simeq \Omega_{\PP^2}^1$.

  From Pieri's formulas (cf. \cite[Corollary 2.3.5 p. 62]{W03}) we obtain
  \begin{equation} \label{pieri}
    \mF \otimes S^2 \mF(-2) \simeq S^3 \mF(-2) \oplus \wedge^2\mF \otimes \mF(-2) \simeq S^3 \mF(-2) \oplus \mF(2).
  \end{equation}
  Also, the fact that $L_i$ is unstable implies
  \begin{equation} \label{alpharette}
    S^2 \mF(-2) |_{L_i} \simeq \mO_{L_i}(-2) \oplus \mO_{L_i}(2) \oplus \mO_{L_i}(6).
  \end{equation}
  So, tensoring \eqref{eq:reduction-alpha} with $S^2 \mF(-2)$ we
  get
  \begin{equation} \label{eq:reduction-alpha-2}
    0 \to \mK \otimes S^2 \mF(-2) \to S^3 \mF(-2) \oplus \mF(2) \to \bigoplus_{i=1}^{\alpha_1} \left(
      \mO_{L_i}(-2) \oplus \mO_{L_i}(2) \oplus \mO_{L_i}(6) \right) \to 0.
  \end{equation}

  Twisting \eqref{eq:reduction-alpha} by $\mO_{\PP^2}(2)$ and taking
  cohomology we get $H^1(\PP^2, \, \mF(2))=0$. Now, since we are in characteristic $0$, the stability of $\mF$ implies that $S^2 \mF^\vee(-1)$ is semistable, of slope $-5$. On the other hand, by Table \eqref{K}, each summand of $\mK^\vee$
  is semistable (and $\mK$ is even stable for $\alpha_1 \ne 3, \, 4, \,5$) of slope between $-3/2$ (for $\alpha_1=1$) and $3/2$
  (for $\alpha_1=7$). In any case, all summands of $\mK^\vee \otimes S^2
  \mF^\vee(-1)$ are semistable of strictly negative slope, so using Serre duality we get
   \begin{equation*}
  H^2(\PP^2,\mK \otimes S^2 \mF(-2)) \simeq H^0(\PP^2,\mK^\vee \otimes S^2 \mF^\vee(-1))^\vee=0.
  \end{equation*}
  
  Therefore, taking cohomology in \eqref{eq:reduction-alpha-2} we obtain $H^2(\PP^2, \, S^3 \mF(-2))=0$ and 
  a surjection
  \begin{equation*}
    H^1(\PP^2, \, S^3 \mF(-2)) \to  \bigoplus_{i=1}^{\alpha_1} H^1(L_i, \, \mO_{L_i}(-2)) \to  0,
  \end{equation*}
  which in turn implies $h^1(\PP^2, \, S^3 \mF(-2)) \geq
  \alpha_1$. By Riemann-Roch theorem we have $\chi(\PP^2, \, S^3 \mF(-2)) = 0$, hence $h^0(\PP^2, \, S^3 \mF(-2)) \ge \alpha_1$, that is \eqref{eq:geq-alpha}.
\end{proof}

Let us now state the result concluding the proof of Theorem \ref{secondimoduli}.
\begin{proposition}
The moduli space $\FN_{\ref{VII}}$ has at least seven
connected, irreducible, unirational components, all of dimension $20$, that are distinguished by the number $\alpha_1 \in \{1, \ldots, 7\}$ of unstable lines for $\mF$.
\end{proposition}
\begin{proof}
  Proposition \ref{prop.typeVII} shows the existence of seven families 
\begin{equation*}
\FN_{\ref{VII}}^1, \ldots, \FN_{\ref{VII}}^7
\end{equation*}  
  of triple planes, one for each value of the number $\alpha_1 \in \{1,\ldots,7\}$ of unstable lines of $\mF$. Such families are pairwise disjoint subsets of $\FN_{\ref{VII}}$, because $\alpha_1$ coincides with the number of
lines contracted by the first adjunction map of $X$, and this number
is an invariant of the triple cover. Moreover, all the cases missing
the star in Proposition \ref{prop.typeVII} have different values of
$\alpha_2$ than the covers belonging to the
$\FN_{\ref{VII}}^{\alpha_1}$. Since also $\alpha_2$ is an invariant of the triple cover, the connected components of $\FN_{\ref{VII}}$ possibly 
containing the missing cases are necessarily disjoint from all the $\FN_{\ref{VII}}^{\alpha_1}$. This shows that our seven families actually are seven connected components of $\FN_{\ref{VII}}$.  

Let us show now that such connected components are also irreducible and unirational. Consider  
the $21$-dimensional (rational) moduli space $\mathrm{M}_{\PP^2}(2, \, 4, \, 10)$  of rank-$2$ stable bundles on
$\PP^2$ with Chern classes $(4, \, 10)$ and having a Steiner-type
resolution, and let $\mathscr{U}^{\alpha_1} \subset
\mathrm{M}_{\PP^2}(2, \, 4, \, 10)$ be the stratum corresponding to
vector bundles having  $\alpha_1$ unstable lines. 
These strata are irreducible and unirational and their codimension is precisely
$\alpha_1$, see \cite[Theorem 5.6]{ancona-ottaviani:steiner}.

Our computations with {\tt Macaulay2} (cf. Appendix) show that  there
exist examples of bundles $\mF$ with $\alpha_1$ unstable lines and
satisfying 
\begin{equation} \label{eq:S^2=alpha1}
h^0(\PP^2, \, S^3\mF(-2))=\alpha_1. 
\end{equation}
So, by Proposition \ref{prop:S2-8} and semicontinuity, equality \eqref{eq:S^2=alpha1} holds for the
general member of the stratum $\mathscr{U}^{\alpha_1}$. 
Each $\FN_{\ref{VII}}^{\alpha_1}$ has an open dense subset which is an
open dense piece of a $\PP^{\alpha_1-1}$-bundle over
$\mathscr{U}^{\alpha_1}$, and as such it is an irreducible, unirational
variety.
For every $\alpha_1 \in \{1, \ldots, 7\}$, using \eqref{eq:S^2=alpha1} we obtain
\begin{equation*}
\dim \, \FN_{\ref{VII}}^{\alpha_1}  = \dim \mathscr{U}^{\alpha_1} + h^0(\PP^2, \, S^3\mF(-2)) -1 = (21- \alpha_1)+ \alpha_1 - 1 = 20.
\end{equation*}
 
Summing up, every 
$\FN_{\ref{VII}}^{\alpha_1}$ is a connected, irreducible, unirational
20-dimensional component of $\FN_{\ref{VII}}$.
\end{proof}

\begin{remark} \label{rmk:moduli-Alexander}
We could also give a geometric interpretation of the equality  $\dim
\, \FN_{\mathrm{VII}}^{\alpha_1} =20$ by using in each case the
explicit description of the surface $X$ provided by Proposition
\ref{prop.typeVII}. We will not develop this point here, and we will
limit ourselves to discussing as an example the case $\alpha_1=6$. In
this situation, we know that $X$ is isomorphic to the blow-up at six
points of an Alexander surface of degree $9$ in $\PP^4$, see Remark
\ref{rmk:Alexander}. Such points are the intersection of the Alexander
surface with its unique $6$-secant line, and they completely determine
the triple cover map $f \colon X \to \PP^2$. So the dimension of the
component $\FN_{\mathrm{VII}}^{6}$ equals the dimension of an open,
dense subset of $S^{10}(\PP^2)$, that is $20$.  
\end{remark}

\section*{Appendix: The computer-aided construction of triple planes}

Here we explain how we can use the Computer Algebra System {\tt Macaulay2} in order to  show the existence of general triple
planes in the cases marked with $(\ast)$ in Proposition \ref{prop.typeVII}. 
The computation can be performed either over $\QQ$ or over a prime
field (the latter being considerably faster).

\subsection*{The setup for adjunction}

Define the coordinate ring of $\PP^2$ and of $\PP^{b-3}= \PP^5$ needed for
the first adjunction map, together with
a second $\PP^5$ (the projectivization of the six-dimensional polynomial ring \verb|V|) that will be the target space for the second adjunction.

\begin{verbatim}
b = 8;
k = QQ;
T = k[x_0..x_2];
S = k[y_0..y_(b-3)];
R = T**S;
V = k[t_0..t_5];
\end{verbatim}

The command {\tt fliptensor} takes as input the matrix $M$ and gives as output the matrix $N$, cf. \S \ref{L}.
\begin{verbatim}
fliptensor := M->(Q = substitute(vars S,R) * (substitute(M,R));
    sub((coefficients(Q,(Variables=>{x_0,x_1,x_2})))_1,S));
\end{verbatim}

The $3$-fold scroll $Y\subset \PP^5$ is defined by the $3\times 3$
minors of $N$.
The command {\tt twosections} gives back the ideal of the union of two surface
sections $S_1$ and $S_2$ of the scroll $Y$, with
$S_i \in |\mO_Y(\ell)|$ and
$\mO_Y(\ell)=\mathfrak{p}^*\mO_{\PP^2}(1)$, cf. \S \ref{a} and \S
\ref{L}.
 Each of them is defined
by the $2\times 2$ minors of a random submatrix of $N$, obtained by
composing $N$ with a random matrix of scalars.

\begin{verbatim}
twosections := N->(A = random(S^{3:0},S^{3:0});
Nrandom = (transpose(N)*A);
N1 = submatrix(Nrandom, {0,1});
N2 = submatrix(Nrandom, {0,2});
IS1 = minors(2, N1);
IS2 = minors(2, N2);
I12 = intersect(IS1,IS2));
\end{verbatim}

The command {\tt cubicgenerator} takes a random cubic 
in the ideal of cubics of $Y$ through $S_1 \cup S_2$, and call $X_1$ the residual surface. This surface is precisely the image of the first adjunction map $\varphi_1 \colon X \to X_1 \subset \PP^5$, see the last part of the proof of Proposition \ref{prop.typeVII}.
\begin{verbatim}
cubicgenerator := I12 -> (SU = super basis(3,I12);
cubic = SU*random(S^{rank(source(SU)):0},S^{1:0});
ideal(cubic));
\end{verbatim}

\subsection*{The cases according to the number of unstable lines}

Here we define the Steiner bundle $\mF$ by giving its presentation matrix $M$.
More precisely, for any $\alpha_1 \in \{1, \ldots, 7\}$ we define a random Steiner bundle with $\alpha_1$ unstable lines.

\subsubsection*{The cases $1 \leq \alpha_1 \le 6$}
For $1 \leq \alpha_1 \le 6$, we
 put random coefficients in the layout of
 Proposition \ref{prop:unstable-b-8} in order to define $\mF$. 
The command
\verb|GenM| takes an integer \verb|a|, picks \verb|a| random linear forms,
multiplies each of them by a column matrix of size $4$ of random
scalars, and stacks them together with a random matrix of linear forms in order to obtain a matrix $M$ of size $4 \times 6$, given as output.

\begin{verbatim}
use T
GenM:=(a)->(
    for j from 0 to a-1 do 
        M_j=((random(T^{1},T^{0}))_(0,0))*random (T^{4:0},T^{1:0});
    Mcu = transpose M_0;
    for j from 1 to a-1 do Mcu=(Mcu||transpose(M_j));
    Mco = (random(T^{6-a:0},T^{4:-1}));
    ((transpose Mcu) | (transpose Mco)))
\end{verbatim}

We choose $\alpha_1$, define the Steiner sheaf as cokernel of
$M$ and check that
it is locally free of rank $2$.
\begin{verbatim}
for a from 1 to 6 do F_a = coker transpose map(T^{b-4:1},T^{b-2:0},GenM(a))
for a from 1 to 6 do print dim (minors(4,presentation F_a))
\end{verbatim}
The output of this is $0$ in all seven cases, so the sheaves are
locally free.

\subsubsection*{The case $\alpha_1 = 7$}

In this case $\mF$ is a logarithmic bundle, so its dual appears as the first syzygy of
the Jacobian map $\nabla f$ of partial derivatives of the product $f$ of the $7$
linear forms that define the $7$ unstable lines. In other words, we
have an exact sequence 
\begin{equation*}
0\to \mF^\vee \to \mO_{\PP^2}(1)^3 \xr{\nabla f} \mO_{\PP^2}(7),
\end{equation*}
cf. for instance \cite[(1.10)]{faenzi-matei-valles}. We choose 
these 7 lines 
randomly and define $\mF$ as the dual of $\ker(\nabla f)$.

\begin{verbatim}
f = 1_T; for j from 1 to 7 do f=f*(random(T^{1},T^{0}))_(0,0)
M=transpose ((res ker diff(vars T,f))).dd_1
MM = map(T^{b-4:1},T^{b-2:0},M);
dim minors(4,MM) == 0
F_7 = coker transpose MM;
\end{verbatim}

We check incidentally that the vanishing $H^0(\PP^2,S^2 \mF(-2))=0$
and the equality $h^0(\PP^2,S^3 \mF(-2))=\alpha_1$ hold true 
for all values of $\alpha_1$ (this fact was needed in the proof of
Proposition \ref{prop:S2-8} and \ref{prop:S2-8}).
\begin{verbatim}
for a from 1 to 7 do print(
    HH^0((sheaf (symmetricPower(2,F_a)))(-2)),
    rank HH^0((sheaf (symmetricPower(3,F_a)))(-2)))
\end{verbatim}
The output is $(0, \, a)$ with $a=\alpha_1 \in \{1, \, \ldots, 7 \}$.

\subsection*{Construction of the triple plane}

We take $\mF$ and extract the matrices $M$ and $N$.
\begin{verbatim}
for a from 1 to 7 do NN_a = fliptensor(presentation (F_a));
for a from 1 to 7 do IY_a = minors(rank target NN_a,NN_a);
for a from 1 to 7 do singY_a = ideal singularLocus variety IY_a;
\end{verbatim}
Singularity test: the only singular points of $Y$ are $\alpha_1$
points of multiplicity $6$. They all come from the locus where the
matrix $N$ defining $Y$ has rank at most $1$.

\begin{verbatim}
for a from 1 to 7 do I2Y_a = minors(rank (target NN_a)-1,NN_a);
for a from 1 to 7 do print (dim singY_a, degree singY_a)
for a from 1 to 7 do print (dim(singY_a:I2Y_a),degree(singY_a:I2Y_a))
\end{verbatim}
The output of the last command is $(1, \, a)$
where $a=\alpha_1$ goes from $1$ to $7$ in the seven cases, and means that $Y$
is singular precisely at the $a$ double points coming from the $a$
unstable lines.
Define now $X_1$ as a random cubic in the ideal of the union of two
surface sections of $Y$ from $|\mO_Y(\ell)|$. Perform a degree,
genus and singularity test. 

\begin{verbatim}
for a from 1 to 7 do II12_a = twosections(NN_a);
for a from 1 to 7 do IC3_a = cubicgenerator(II12_a);
for a from 1 to 7 do IX1_a = ((IC3_a + IY_a):II12_a);
for a from 1 to 7 do X1_a = variety(IX1_a);
for a from 1 to 7 do print(dim X1_a, degree X1_a,genera X1_a)
for a from 1 to 7 do (dim singularLocus X1_a)
\end{verbatim}

In all seven cases, the output of the penultimate command is {\tt{(2, 10, \{0, 6, 9\})}}.
which means
that $X$ is a surface of degree $10$ with sectional genera $(0, \, 6, \, 9)$.
The output of the last command is $-\infty$, i.e. $X$ is smooth. This
takes about $15$ minutes on a laptop if performed on a prime field.

\medskip

The second adjunction map of $\varphi_2 \colon X_1 \to X_2 \subset \mathbb{P}^5$ is defined by the restriction to $X_1$ of the linear system $|\mO_{Y}(2\xi -\ell)|$, and this in turn coincides with the
restriction to $X_1$ of the linear system generated by the six quadrics in the ideal defining $S_1$, see again  the proof of Proposition \ref{prop.typeVII}.

Having this in mind, we can finally compute the ideal of $X_2$ and its canonical sheaf $\omega_{X_2} = \mO_{X_2}(K_{X_2})$ in order to find $K_{X_2}^2$.
\begin{verbatim}
for a from 1 to 7 do X2_a = (ker map(S/(IX1_a),V,
        gens  minors(2,random(S^{2:0},S^{3:0})*NN_a)))
for a from 1 to 7 do omegaX2_a = (Ext^2(X2_a,V^{1:-6}))**(V/(X2_a));
for a from 1 to 7 do print(euler(dual omegaX2_a)-1)
\end{verbatim}

Here is the output of the last command, providing the value of
$K_{X_2}^2$ for all $\alpha_1 = a \in \{1, \ldots, 7\}$:
\begin{verbatim}
9, 5, 2, 0, -1, -1, 0
\end{verbatim}


\bibliographystyle{alpha}
\bibliography{bibliography}

\bigskip \bigskip

\noindent Daniele Faenzi, \\ Institut de Mathématiques de Bourgogne,
  UMR 5584  CNRS,\\
  Université de Bourgogne Franche-Comté, \\
  9 Avenue Alain Savary, BP 47870, \\
  21078 Dijon Cedex,
  France \\
\emph{E-mail address:} \verb|daniele.faenzi@u-bourgogne.fr| \\ \\

\noindent Francesco Polizzi, \\ Dipartimento di Matematica, Universit\`{a}
della Calabria, Cubo 30B, \\ 87036  Arcavacata di Rende, Cosenza, Italy\\
\emph{E-mail address:} \verb|polizzi@mat.unical.it| \\ \\

\noindent Jean Vall\`{e}s, \\Université de Pau et des Pays de l'Adour \\
  Avenue de l'Université - BP 576  \\ 64012 PAU Cedex, France \\
 \emph{E-mail address:} \verb|jean.valles@univ-pau.fr|

\end{document}